\documentclass[11pt, reqno]{amsart}
\usepackage{latexsym, amsmath, amssymb, longtable, booktabs,amscd, color}
\usepackage[english]{babel}
\usepackage[utf8]{inputenc}
\usepackage{mathtools}
\usepackage{ stmaryrd }
\usepackage{verbatim}
\usepackage{enumitem}
\usepackage[all]{xy}
\usepackage{tikz}
\usepackage{tikz-cd}
\numberwithin{equation}{section}
\usepackage{xcolor}
\usepackage{hyperref}
\hypersetup{
    colorlinks,
    linkcolor={red!50!blue},
    citecolor={blue!50!black},
    urlcolor={blue!80!black}
}

\newcommand\Q{\mathbb{Q}}
\newcommand\ind{\mathrm{ind}}

\newcommand\N{\mathbb{N}}

\newcommand\sF{{\mathcal{F}}}

\newcommand\GL{{\mathrm{GL}}}

\newcommand{\Z}{\mathbb{Z}}

\newcommand{\F}{\mathbb{F}}

\newcommand\Gal{{\mathrm {Gal}}}

\newcommand\Sym{{\mathrm {Sym}}}

\newcommand\linee{{
        {\begin{tiny}
        \begin{xymatrix}{
         \bullet  \ar@{-}[d] \\ 
         \bullet  \ar@{-}[d] \\
         \bullet} 
       \end{xymatrix}
       \end{tiny}} }}  
\newcommand\diamondd{{
        \begin{tiny}
        \begin{xymatrix}{
         & \bullet  \ar@{-}[dl] \ar@{-}[dr] &  \\ 
         \bullet  \ar@{-}[dr] &   &   \bullet  \ar@{-}[dl] \\
         & \bullet & } 
       \end{xymatrix}
       \end{tiny} }}

\newtheorem{thm}{Theorem}[section]
\newtheorem{theorem}[thm]{Theorem}
\newtheorem{cor}[thm]{Corollary}

\newtheorem{prop}[thm]{Proposition}
\newtheorem{lemma}[thm]{Lemma}
\theoremstyle{definition}

\theoremstyle{remark}
\newtheorem{remark}[thm]{Remark}

\theoremstyle{definition}

\theoremstyle{remark}

\theoremstyle{remark}

\title[]{Behaviour of certain crystalline representations modulo $2$}

\author{Shalini Bhattacharya and Arathy Venugopal}
\address{University of Hyderabad and IISER Tirupati  }
\email{shalinib@uohyd.ac.in, arathyvenugopal@students.iisertirupati.ac.in}


\date{May, 2026}
\begin{document}
\begin{abstract}
We compute the explicit form of the semisimplified reduction modulo $2$  of the $2$-adic crystalline Galois representations $V_{k,a_2}$ at small slopes in $(0,1]$, using the compatibility of $2$-adic and mod-$2$ local Langlands correspondence. We find parameters $\alpha'(k,a_{2})$ and $\alpha(k,a_{2})$, which play a crucial role in determining the reduction of $V_{k,a_{2}}$ for slopes in the range $(0,1)$ and slope $1$ respectively.  
\end{abstract}
\maketitle

\section{Introduction}\label{introduction}
Let $p$ be any prime number and $\wp$ be the maximal ideal in $\bar{\Z}_{p}$. Let $v$ be the $p$-adic valuation on $\bar{\Q}_{p}$ normalised such that $v(p)=1$. 
Let $G_p$ denote the absolute Galois group $\Gal(\bar\Q_p|\Q_p)$. 
Let $E$ be a finite extension of $\Q_{p}$ and let $a_{p} \in E$ satisfy $\nu :=v(a_{p})>0$.  Given an integer $k\geq 2$ and $a_{p}$ as above, the symbol $V_{k,a_{p}}$ denotes the unique irreducible two-dimensional crystalline representation of $G_p$ with Hodge-Tate weights $0$ and $k-1$, such that the characteristic polynomial of the crystalline Frobenius is $X^2-a_{p}X+p^{k-1}$. The value $\nu=v(a_p)$ is referred to as the `slope' of the representation $V_{k,a_p}$.
   One can 
   choose a Galois stable integral lattice in $V_{k,a_p}$ and reduce it modulo $\wp$. 
   The semisimplification of this reduction of $V_{k,a_p}$, denoted by $\bar{V}_{k,a_{p}}$, is known to be independent of the choice of the lattice. For odd primes or sufficiently large primes, the explicit structure of $\bar V_{k,a_p}$ under various conditions can be found in the literature. However, due to technical difficulties, the prime $p=2$ has been avoided in most of these works. In this article we attempt to fill this gap in the existing literature, by treating the problem for $p=2$ at small slopes.
   
   The method of using  local Langlands correspondence (LLC) for the reduction of Galois representations was first introduced by Breuil in 2003 \cite{[Br03]}. 
   In this method, one reduces modulo $p$ the  unitary $p$-adic  Banach space representation  $B(V_{k,a_{p}})$  which is attached to $V_{k,a_{p}}$ via the $p$-adic LLC; a functorial construction of the $p$-adic LLC using $(\varphi,\Gamma)$-modules can be found in \cite{Colmez10}. Berger proved in \cite{Berger10} that the reduction $\overline{B(V_{k,a_{p}})}$ lies in the image of  mod-$p$ LLC  and  corresponds to $\bar{V}_{k,a_{p}}$ under the same. Since mod $p$ LLC is an injective correspondence, it is enough to compute $\overline{B(V_{k,a_{p}})}$ to determine $\bar{V}_{k,a_p}$.
	This way of computing the reduction in the `other side' of LLC is found to be particularly effective at small slopes, applied in the papers \cite{[BG09]}, \cite{Buzzard-Gee13}, \cite{[GG15]}, \cite{Bhattacharya-Ghate}, \cite{Bhattacharya-Ghate-Rozensztajn18}, \cite{Arsovski}, \cite{[GR18]},  \cite{[NP25]}. More details about our methodology can be found in \S 2, with focus on the special prime $p=2$. 

    Before stating our main theorems, we will introduce some notations. Let $I_{p}$ be the inertia subgroup of $G_{p}$ and
$\omega_1,\omega_2:I_{p}\rightarrow \bar\F_p^\times$ denote the fundamental characters of level one and  two respectively. Note that $\omega_{1}$ can be extended to all of $G_{p}$ as the cyclotomic character mod $p$, also denoted by $\omega_1$ here. For $n\geq 1$ and $p+1 \nmid n$, let $\rho=\mathrm{ind}(\omega_2^n)$ denote the unique irreducible mod $p$ representation of $G_{p}$ such that $\det(\rho)=\omega_1^n$ and $\rho|_{I_p}=\omega_2^n\oplus\omega_2^{np}$.  For $p=2$, the character $\omega_2$ is of order $3$. Let us also mention here that for $p=2$, the fundamental character of level one is the trivial character on $G_p$ and up to twists by unramified characters, any irreducible two-dimensional representation of $G_p$ over $\bar\F_p$ is isomorphic to $\ind (\omega_2)$. Now let us state our first result, proved in \S 3, giving the complete structure of the reduction $\bar V_{k,a_2}$ when the slope $v(a_2)$ is less than $1$. 

\begin{theorem}\label{ThmA}
			Let $k\geq 4$, $0 <v(a_{2}) < 1$. For $r:=k-2$, let $\alpha':= \frac{a_{2}^{2}-2r^{2}}{2a_{2}}$, $\tau': = v\left(\alpha'\right)$ and $t:=v(r-1)$. Then we have: 
		\begin{enumerate}
				\item If $\tau'< t$, then $\bar{V}_{k,a_{2}} \cong \ind(\omega_{2})$.
				\item If $\tau'\geq t$, then $\bar{V}_{k,a_{2}} \cong \mu_{\lambda} \oplus \mu_{\lambda^{-1}}$,\\ where $\lambda \in \bar{\F}^{\times}_{2}$ is determined by the equation $\lambda^2+c\lambda+1=0$ with $c=\overline{\frac{\alpha'}{r-1}}\in\bar\F_2.$
			\end{enumerate}
		\end{theorem}
        An easy analysis of the quantity $\alpha'=\alpha'(k,a_2)$ and its valuation $\tau'$ defined above leads us to the following quick consequences:
        \begin{cor}\label{cor1}
        Let us assume  $k\geq 2$ and $0<v(a_2)<1$. Then
        
        (a) if $k$ is even,  the reduction $\bar V_{k,a_2}$ must be irreducible.
         
           (b)  if $\bar V_{k,a_2}$ is reducible, then $k$ is odd and the slope $v(a_2)=1/2$.
         \end{cor}
        We can compare Theorem \ref{ThmA} with the main result in \cite{Buzzard-Gee13}, where a similar dichotomy is stated for odd primes, slopes in $(0,1)$ and $k\equiv 3\mod (p-1)$. There the shape of $\bar V_{k,a_p}$ is shown to be determined by $\tau:=v(\alpha)$, where $\alpha:=\frac{a_p^2-pr}{pa_p}$. This can also be seen as a special case of the Zig-zag phenomenon conjectured in \cite{[G21]} and proved in \cite{[G22]}. However, note that the relevant expression $\alpha'$ in Theorem \ref{ThmA} above is slightly different from $\alpha$, or  the general expression in the Zig-zag conjecture. In fact for $p=2$, the values $t=v(r-1)$ and $\tau=v(\alpha)$ solely cannot determine the structure of $\bar V_{k,a_2}$ in all cases. For a comparative study, we refer to the examples  in Table $1$, $\S5$.

        Next let us state our second result giving a complete structure of $\bar V_{k,a_2}$ when slope $v(a_2)$ equals $1$. The proof of this theorem can be found in \S4, and some relevant numerical examples can be found in Table $2$, $\S5$.
\begin{theorem}\label{ThmB}
   Let $k\geq4$ and $v(a_{2})=1$. For $r:=k-2$, let
    $\alpha := \frac{a_{2}^2-2^2{r \choose 2}}{2a_{2}}, \tau:= v(\alpha)$ and  $t:=v(r-2)$. Then we have: 
    \begin{enumerate}
        \item If $\tau\leq t-1$, then $\bar{V}_{k,a_{2}} \cong \mu_{\lambda}\oplus \mu_{\lambda^{-1}}$, where $\lambda \in \bar{\F}_{2}^{\times}$ satisfies $\lambda^2-c\lambda + 1=0$, \\with $c = \overline{\frac{a_{2}}{2} + \frac{r-2}{2\alpha}}$ unless $r-2=\alpha=0$, in which case we take $c=\overline{\frac{a_2}{2}}$.
        \item If $\tau>t-1, \,\bar{V}_{k,a_{2}} \cong \ind(\omega_{2})$.
    \end{enumerate}
\end{theorem}
To avoid confusion, we caution the reader that the symbol $t$ here is different from that in Theorem \ref{ThmA}. Analysing the valuation of $\alpha=\alpha(k,a_2)$ defined in Theorem \ref{ThmB}, we observe the following, which is in contrast with part $(a)$ of Corollary \ref{cor1} above.
\begin{cor}\label{cor2}
    If $v(a_2)=1$ and $k\geq 3$ is odd, then $\bar V_{k,a_2}$ must be irreducible.
\end{cor}
Theorem \ref{ThmB} can be compared with similar cases for odd primes, e.g., the `$b=2$' part of Theorem 1.1 of \cite{Bhattacharya-Ghate-Rozensztajn18}, or in general with the $\nu=1$ case of the Zigzag conjecture. We note that though the determining factors $\tau=v(\alpha)$  and $t=v(k-4)$ remain the same for odd primes ($>3$) and  $p=2$, the behavior of the general reduction $\bar V_{k,a_p}$ with variation of $\tau $ and $t$ are not quite the same for odd and even primes.  From both the theorems above, we observe a very common trend in number theory: the even prime $p=2$ behaves  differently than the other primes.

Let us mention here that the condition $k\geq 4$ in the main theorems is justified since the reduction is well-known for $2\leq k\leq 3=p+1$ by a general result of Fontaine and Edixhoven in \cite{Edixhoven92}  valid for all primes including $2$. The same result can be used to check the Corollaries \ref{cor1} and \ref{cor2} for $k=2$ and $3$. The smallest weight  
treated in this paper is $k=4$, for which our answer coincides with the one predicted by Breuil in a comment after Theorem $1.4$ in \cite{[Br03]}. We also note that by the main result of \cite{Berger-Li-Zhu04} and \cite{Bergdall-Levin}, the shape of $\bar V_{k,a_2}$ is known when $v(a_2)>\lfloor \frac{k-1}{2}\rfloor$. However, the slope range $(0,1]$ treated in this article is even lower than  this bound for all weights $k\geq 4$.   

Now let us give a brief outline of the proof of our theorems. Applying Local Langlands correspondence, the problem of determining $\bar V_{k,a_2}$ reduces to computing the reduction of certain lattice $\Theta_{k,a_2}$ in the automorphic side. One knows that the reduction $\bar\Theta_{k,a_2}$ is a quotient of an induced representation $\mathcal{I}(V_r)$, with $r=k-2$. For notations $V_r$, we refer to $\S 2$ below. Further in the case of slope $<1$, we checked that $\bar \Theta_{k,a_2}$ is in fact a quotient of $\mathcal{I}(V_0)$. Then the most tricky part of the proof is to distinguish between the possibilities of reducible and irreducible reductions, discussed in detail in $\S3$. For the case of slope one, discussed in $\S4$, our preliminary analysis shows that generically (for $r>4$) two sub-quotients $\mathcal{I}(V_0)$ and $\mathcal{I}(V_1)$ of $\mathcal{I}(V_r)$ can possibly contribute to $\bar\Theta_{k,a_2}$. Further we show that in most cases $\mathcal{I}(V_1)$ contributes, and $\bar\Theta_{k,a_2}$  
is either reducible or irreducible depending on whether $\tau\leq t-1$ or $\tau>t-1$. However, in some special cases, in addition to $\mathcal{I}(V_1)$ the factor $\mathcal{I}(V_0)$ may also contribute to $\bar\Theta_{k,a_2}$ resulting in a reducible answer. This can possibly occur only if $\tau=t-1$ and moreover if $\overline{2/a_2}=\overline{(2-r)/2\alpha}\in\bar\F_2^\times$.

The primary hindrance as to why the methodology for odd primes could not be emulated in the case of $p=2$ is technical. For instance, the identity $\displaystyle\sum_{\mu \in \F_{p}}[\mu]=0$ is significantly used in the calculation for odd primes whereas it fails trivially for $p=2$. Also some of the combinatorial lemmas which were crucial in the analysis of functions on the Bruhat-Tits tree in \cite{Buzzard-Gee13} and \cite{Bhattacharya-Ghate-Rozensztajn18} do not hold as such for $p=2$. We have given the modified versions of these lemmas for $p=2$ in this paper. While the modification appears to be minor in some cases, it has a non-trivial impact on the calculation leading to completely different results for $p=2$.


\section{background}
In this section, we will briefly recall the method of applying the local Langlands correspondence to compute the mod-$p$ reduction of crystalline representations. This method is applicable for all primes, but in this paper we will consider the specific case of $p=2$.
 %
\subsection{Hecke operators and mod-2 LLC}
  Let $G$ and $K$ denote the groups $\GL_{2}(\Q_{p})$ and $\GL_{2}(\Z_{p})$ respectively. We shall denote the compact induction functor $\ind_{K\Q_{p}^{\times}}^{G}$ by $\mathcal{I}$, with $\Q_p^\times$ being the centre of the group $G$. For any $\Z_{p}$-algebra $R$, the symmetric power  $\Sym^{r}R^{2}$ is modeled by the space of  homogeneous polynomials of degree $r$ in two variables $X$ and $Y$ over $R$. The group $K$ acts on the polynomials  in $ \Sym^{r}R^2$ by 
    $$ \left(\begin{smallmatrix}
      a & b\\
      c & d
  \end{smallmatrix}\right)\cdot f(X,Y)= f(aX+cY, bX+dY).$$ This action is extended to an action of $K\Q_{p}^{\times}$ by making the scalar matrix $\left(\begin{smallmatrix}
      p & 0\\
      0 & p
  \end{smallmatrix}\right) $ act trivially.
Given $g \in G$ and $v =v(X,Y)\in \Sym^{r}{R^{2}}$, the symbol $\left[g,v\right]$ denotes the $\Sym^{r}R^2$-valued function on $G$ defined as follows:   
       $$[g,v](g') = \begin{cases}
       (g'g)\cdot v &\text{ if } g' \in K{\Q}_{p}^{\times}g^{-1}\\
       0 & \text{ otherwise}.
       \end{cases}$$
       It is known that a general element of  $\mathcal{I}(\Sym^{r}R^{2})$ is a finite sum of elementary functions of the above type. Next let us recall the Hecke operator $T$, a distinguished operator acting on $\mathcal{I}(\Sym^r R^2)$ defined by the formula:
      \begin{equation}
          T[g,v(X,Y)] = \sum_{\mu \in \Z_{p}: \mu^{p}=\mu}\left[g\begin{pmatrix}
              p  & \mu\\
              0 & 1
          \end{pmatrix}, v(X, pY-\mu X)          \right] + \left[g\begin{pmatrix}
              1  & 0\\
              0 & p
          \end{pmatrix}, v(pX, Y) \right]     \end{equation} 
 It can be checked that $T$ is a $G$-linear map from $\mathcal{I}(\Sym^r R^2)$ to itself.

For $r \in \Z_{\geq 0}$, the group $\Gamma:=\GL_2(\F_p)$ acts naturally on the space $\Sym^{r}(\bar{\F}_{p}^2)$, and this action can be further inflated to an action of $K$. Let this representation space  $\Sym^{r}(\bar{\F}_{p}^2)$  be denoted by the symbol $V_{r}$. If $0\leq r\leq p-1$ is an integer, $\lambda \in \bar{\F}_{p}$ and $\eta:\Q_{p}^{\times} \rightarrow\bar{\F}_{p}^{\times}$ is a smooth character,  then we  know
$$\pi(r,\lambda,\eta) := \frac{\mathcal{I}(V_{r})}{(T-\lambda)} \otimes \eta$$ 
is a smooth admissible representation of $G$ over $\bar{\F}_{p}$ and is reducible if and only if $(r,\lambda) \in \{(0,\pm1),(p-1,\pm1)\}$. If irreducible, these representations are known as principal series when $\lambda \neq 0$ and supercuspidals when $\lambda=0$. All smooth admissible irreducible representations of $G$ are  sub-quotients of $\pi(r,\lambda, \eta)$ for some tuple $(r,\lambda,\eta)$ due to work of Barthel-Livne \cite{Barthel-Livne94} and Breuil \cite{Breuil1}. For a complete statement of classification of smooth admissible irreducible representations of $G$, we refer to Theorem $4.4$ of \cite{Breuil07}.

We shall now discuss the semi-simple mod-$p$ local Langlands correspondence. A general statement can be found in definition $1.1$ of \cite{[Br03]}. Here we will state the specific version for $p=2$. Let $\lambda \in \bar{\F}_{2}^\times$ and $\eta$ be a smooth character of $\Q_{2}^{\times}$. Then the mod-$2$ local Langlands correspondence can be stated as:
	\begin{eqnarray*}
		\textbf{Galois side} & & \textbf{Automorphic side}\\
			\ind(\omega_{2}) \otimes \eta &\leftrightarrow& \pi(0,0,\eta)\\
			(\mu_{\lambda} \oplus \mu_{\lambda^{-1}})\otimes \eta 	&\leftrightarrow & \pi(0,\lambda,\eta)^{ss} \oplus \pi(0,\lambda^{-1}, \eta)^{ss}
		\end{eqnarray*}
We note that unlike the odd primes, the statement of mod $p$ LLC for $p=2$ does not involve a variable $r$. 
This is not surprising since any irreducible two-dimensional representation of $G_{p}$ for $p=2$ can be expressed as $\ind(\omega_{2})\otimes \eta$ and one knows that in the automorphic side $\pi(1,\lambda, \eta)^{ss} \cong \pi(0, \lambda, \eta)^{ss}$. The later isomorphism can be verified from Theorem $1.3$ of \cite{Breuil1} for supercuspidals; part $2(b)$ of corollary $36$ of \cite{Barthel-Livne94} for principal series, and the description of composition factors for the reducible cases given in parts $2(b)$ and $2(c)$ of Theorem $30$ of \cite{Barthel-Livne94}.
        
%
 Let $B(V_{k,a_{2}})$ be the unitary $2$-adic Banach space representation attached to $V_{k,a_{2}}$ in the sense of \cite{Berger-Breuil10}. We recall that by the compatibility of $p$-adic and mod-$p$ LLC by Breuil and Berger, the image of $\bar{V}_{k,a_{2}}$ under mod-$2$ LLC coincides with $\overline{B(V_{k,a_{2}})}$.
 We now consider the locally algebraic representation $\Pi_{k,a_{2}}$:= $\dfrac{\mathcal{I}(\Sym^{k-2}\bar{\Q}_{2}^{2})}{(T-a_{2})}$ and let the image of $\mathcal{I}(\Sym^{k-2}\bar{\Z}_{2}^2)$ in $\Pi_{k,a_{2}}$ be denoted by $\Theta_{k,a_{2}}$. The explicit description for $\Theta_{k,a_{2}}$ is given by the following formula:
    $$  \Theta_{k,a_{2}}= \frac{\mathcal{I}(\Sym^{k-2}\bar{\Z}_{2}^{2})}{(T-a_{2})(\mathcal{I}(\Sym^{k-2}\bar{\Q}_{2}^{2}))\cap\mathcal{I}(\Sym^{k-2}\bar{\Z}_{2}^{2})}.$$ One knows that
    $\Theta_{k,a_{2}}$ is a lattice in $\Pi_{k,a_{2}}$ in the sense of \cite{[Br03]}. Further $B(V_{k,a_{2}})$ is the completion of $\Pi_{k,a_{2}}$ with respect to the `gauge' of $\Theta_{k,a_{2}}$ in the context of locally convex topology. The semi-simplified reduction $\overline{B(V_{k,a_{2}})}$ is isomorphic to $ \bar{\Theta}_{k,a_{2}}^{ss}:=(\Theta_{k,a_2}\otimes\bar\F_2)^{ss}$.
    Thus to compute $\bar{V}_{k,a_{2}}$, it is enough to compute $\bar{\Theta}_{k,a_{2}}^{ss}$ and then use the injectivity property of mod-$2$ LLC. 
\subsection{Filtration of $V_{r}$} \label{fil}
There is a natural map $P: \mathcal{I}(V_{r}) \to \bar\Theta:=\bar\Theta_{k,a_2}.$
Let us define
$\theta(X,Y):=X^{2}Y-XY^{2},$  which is fixed under the action of 
$\GL_{2}(\F_{2})$. Then the submodules $$V_{r}^{(n)} :=\begin{cases}\theta^{n}V_{r-3n} &\text{ if } r\geq 3n\\0 &\text{ if } r< 3n
\end{cases}$$ for $n\geq 0$ define a filtration of $V_r$. 
The following lemma enables us to characterize elements in $V_{r}^{(n)}$ for $n=1$ and $2$, which will be particularly useful in our context.

\begin{lemma}\label{Vr*}
Let $S(X,Y) = \displaystyle\sum_{0\leq j \leq r}a_{j}X^{r-j}Y^{j} \in \bar{\mathbb{F}}_{2}[X,Y]$. Then 
\begin{enumerate}
	\item	$S \in V_{r}^{(1)} \iff a_{0} = a_{r} = 0$ and $\sum_{j}a_{j} = 0 \in \bar{\mathbb{F}}_{2}$\\
	\item	$S \in V_{r}^{(2)} \iff a_{0} = a_{1} = a_{r-1}= a_{r} = 0$ and\\ $\displaystyle\sum_{j}a_{j} = \sum_{j}ja_{j}=0 \in \bar{\F}_{2},        \,   i.e.,\displaystyle\sum_{j \text{ even }}a_{j} = 0$ and $\displaystyle\sum_{j  \text{ odd }}a_{j} = 0$ in $\bar{\mathbb{F}}_{2}$.
\end{enumerate}	
\end{lemma}
The proof of the above lemma is elementary and is left to the reader. Let $X_{r-m}:=\langle X^{r-m}Y^{m}  \rangle_K\subseteq V_r$ for $0\leq m \leq r$. With this notation, we now recall Remark $4.4$  from \cite{[BG09]} for a general prime $p$ and state the version specific for $p=2$.

\begin{lemma}\label{RemarkBG09}
     If $v(a_2)<n\in\N$, 
     then $\mathcal{I} (V_r^{(n)})\subseteq\ker P$.\\
     Further if $0\leq m< v(a_2)$ and $r\geq 3m+2$, then $\mathcal{I}( X_{r-m})\subseteq\ker P$.
\end{lemma}

We now state a proposition which gives the structure of $V_{r}/V_{r}^{(1)}$.

\begin{prop}
    \label{VrtoV0}
For $r\geq 2$,  we have a split short exact sequence of $\bar{\F}_{2}[\Gamma]$-modules:
\begin{equation}\label{es0}
0\rightarrow V_1\rightarrow\frac{V_r}{V_r^{(1)}}\rightarrow V_0\rightarrow 0,
\end{equation}
 where the injection is given by 
             $ X \longmapsto \overline{X^{r}},\,Y \longmapsto \overline{Y^{r}}$ and the surjection is given by $\underset{0\leq j\leq r}\sum c_j \overline{X^{r-j}Y^j}\longmapsto\underset{1\leq j\leq r-1}\sum c_j\in\bar\F_2.$
\end{prop}
\begin{proof}
Let $X_r^{(1)}:=X_r\cap V_r^{(1)}$. We have $X_r/X_r^{(1)}= \langle \overline{X^{r}},\overline{Y^{r}}\rangle_{\bar{\F}_{2}}.$ We have $V_{1}\cong \frac{X_r}{X_r^{(1)}}$ is a submodule of $\frac{V_{r}}{V_{r}^{(1)}}$ by $(4.5)$ in \cite{Glover}. By dimension count, the quotient $\frac{V_{r}}{X_{r}+V_{r}^{(1)}}$ is one-dimensional and hence isomorphic to $V_{0}$. Thus the short exact sequence follows. The one-dimensional fixed subspaces of upper unipotents in $V_{1}$ and $\frac{X_{r}}{X_{r}^{(1)}}$ are spanned by $X$ and $\overline{X^{r}}$ respectively. By $K$-linearity, the elements  $X$ and $Y=w.X$ map to $\overline{X^{r}}$ and $\overline{Y^{r}}$ respectively up to a scalar. By exactness, both the elements $\overline{X^{r}},\overline{Y^{r}}\in \frac{V_{r}}{V_{r}^{(1)}}$ map to zero in $V_{0}$. We note that the space $V_r/V_r^{(1)}$ is spanned by the images of $X^{r},Y^{r}$ and $X^{r-1}Y$ which are linearly independent. Since the surjection map is non-zero, the image of $X^{r-1}Y$ must map to $1$ up to a scalar. Now the result follows as $ X^{r-j}Y^{j} \equiv X^{r-1}Y \mod V_{r}^{(1)}$  for $1\leq j\leq r-1$.  
\end{proof}
\begin{remark}\label{remVrVr1}
    Note that for all $n\geq 1$ and $r \geq 3n+2$, we have $V_{r}^{(n)}/V_{r}^{(n+1)}\cong
    V_{r-3n}/V_{r-3n}^{(1)}$ and therefore the structure of $V_{r}^{(n)}/V_{r}^{(n+1)}$ for any $n$ is also given by \eqref{es0} above. 
\end{remark}

\subsection{Analysis of the Hecke operator $T$}
Let $R$ be a $\Z_{2}$-algebra. The induction space $\mathcal{I}(\Sym^{r}R^2)$ can be identified with the space of finitely supported $\Sym^rR^2$- valued functions on the Bruhat-Tits tree of $G=GL_{2}(\Q_{2})$. This is known to be a tree of uniform valency $3$ with its vertices corresponding to  integral lattices in $\mathbb{Q}_2^2$ up to homothety. It follows that the vertices are in one to one correspondence with the cosets of $G/K\Q_{p}^{\times}$ for $p=2$.

Let us define $I_{0}=\{0\}$ and 
\begin{equation*}
I_{n}= \{\mu_{0}+\cdots+\mu_{n-2}2^{n-2}+ \mu_{n-1}2^{n-1}|\, \mu_{i} \in \{0,1\}\,\, \text{ for } 0\leq i\leq n\} \text{ for all } n\in \N.
\end{equation*}  For  $n\geq0$ and $\lambda \in I_{n}$, let $g_{n,\lambda}^{0} = \left(\begin{smallmatrix}
  2^{n} & \lambda\\
  0 & 1
\end{smallmatrix}\right)$ and $g_{n,\lambda}^{1} = \left(\begin{smallmatrix}
  1 & 0\\
  2\lambda & 2^{n+1}
\end{smallmatrix}\right).$
Then one has the coset decomposition 
$$ G = \bigsqcup_{n\geq 0,\lambda\in I_{n}}K\Q_{2}^{\times}(g_{n,\lambda}^{0})^{-1} \bigsqcup \bigsqcup_{n\geq 0,\lambda\in I_{n}}K\Q_{2}^{\times}(g_{n,\lambda}^{1})^{-1}.$$ We label the  vertices of the Bruhat-Tits tree by these coset representatives $g_{n,\lambda}^i$ (figure \ref{tree}).
 The vertices $g_{0,0}^{1}$ and $g_{0,0}^{0}$ together constitute the `center of the tree'. As seen in figure \ref{tree}, the vertices $g_{n,\lambda}^{0}$'s lie in the `right half of the tree' and $g_{n,\lambda}^{1}$'s lie in the `left half of the tree'. Given an arbitrary $g \in G$, we say that $g$ is in the right (or left) half of the tree  if $g^{-1}g_{n,\lambda}^{0}\, (\text{or } g^{-1}g_{n,\lambda}^{1}$) belongs to the subgroup $ K\Q_{2}^{\times}$ for some $n$ and $\lambda$.
On identifying elements of $\mathcal{I}(\Sym^rR^2)$ with certain functions on the Bruhat-Tits tree, the formula for $T$ can be broken down into two parts, $T^+$ and $T^-$. Basically $T^+$ is the part of $T$ that takes  support of an elementary function $[g,v]$ away from the center of the tree, and $T^-$ is the  part that takes the support closer to the center. 

For $P(X,Y)\in\Sym^rR^2$,  if $g$ lies in the right half of the tree then we have 
\begin{equation}
\begin{split}
     T^+[g,\enspace P(X,Y)]&=[gg_{1,0}^{0},\enspace P(X,2Y)]+[g g_{1,1}^{0},  \enspace P(X,2Y-X)] \\
    T^-[g,\enspace P(X,Y)]&=[gg_{0,0}^{1},\enspace P(2X,Y)],  
\end{split}
    \label{Tright}
\end{equation}    
and  for $g$ in the left half of the tree we have
   \begin{equation}
   \begin{split}
    T^{+}[g, \enspace P(X,Y)]&=
    [gg_{0,0}^{1},\enspace  P(2X,Y)]+[gg_{1,1}^{0},\enspace P(X,2Y-X)]\\
    T^-[g,\enspace P(X,Y)]&= [gg_{1,0}^{0},\enspace P(X,2Y)]. 
   \end{split}
   \label{Tleft}
   \end{equation}

\begin{figure}[h!]
\begin{center}
    \begin{tikzpicture}
        \draw[black] (0, 0) -- (2, 2);
        \draw[black] (0, 0) -- (2, -2);
        \draw[black] (0, 0) -- (-2.83, 0);
        \draw[black] (2, 2) -- (3.5, 2);
        \draw[black] (2, 2) -- (1.25, 3.3);
        \draw[black] (2,-2) -- (3.5, -2);
        \draw[black] (2, -2) -- (1.25,-3.3);
        \draw[black] (-2.83, 0) -- (-3.58, -1.3);
        \draw[black] (-2.83,0) -- (-3.58, 1.3);
        \draw[black] (3.5,2) -- (4,2.86);
        \draw[black] (3.5,2) -- (4,1.14);
        \draw[black] (1.25,3.3) -- (1.75,4.16);
        \draw[black] (1.25,3.3) -- (0.39,3.3);
        \draw[black] (3.5,-2) -- (4,-1.14);
        \draw[black] (3.5,-2) -- (4,-2.86);
        \draw[black] (1.25,-3.3) -- (1.75,-4.16);
        \draw[black] (1.25,-3.3) -- (0.39,-3.3);
        \draw[black] (-3.58,1.3) -- (-3.08,2.16);
        \draw[black] (-3.58, 1.3) -- (-4.58,1.3) ;
        \draw[black] (-3.58, -1.3) -- (-3.08, -2.16);
        \draw[black] (-3.58, -1.3) -- (-4.58, -1.3);
        \filldraw[black] (0, 0) circle node[anchor=west]{\>\> $g_{0,0}^{0}$};
        \filldraw[black] (2, 2) circle node[anchor= north]{\!\!\!\! $g_{1,0}^{0}$};
        \filldraw[black] (2, -2) circle node[anchor=east]{\!\!\!\! $g_{1,1}^{0}$};
        \filldraw[black] (-2.83, 0) circle node[anchor=east]{\!\!\!\! $g_{0,0}^{1}$};
        \filldraw[black] (3.5,2) circle node[anchor= north]{\!\!\!\! $g_{2,2}^{0}$};
        \filldraw[black] (1.25,3.3) circle node[anchor= north]{\!\!\!\! $g_{2,0}^{0}$};
        \filldraw[black] (3.5,-2) circle node[anchor= north]{\!\!\!\! $g_{2,1}^{0}$};
        \filldraw[black] (1.25,-3.3) circle node[anchor= north]{\!\!\!\! $g_{2,3}^{0}$};
        \filldraw[black] (-3.58,1.3) circle node[anchor= north]{\!\!\!\! $g_{1,1}^{1}$};
        \filldraw[black] (-3.58,-1.3) circle node[anchor= north]{\!\!\!\! $g_{1,0}^{1}$};
        \draw[dashed] (-1.415,5) -- (-1.415,-5);
         \node at (0.5,4.75) {right side};
         \node at (-3,4.8){left side };
    \end{tikzpicture}
    
    \caption{ Bruhat-Tits tree for $p = 2$.}
    \label{tree}     
    \end{center}
      
  \end{figure}
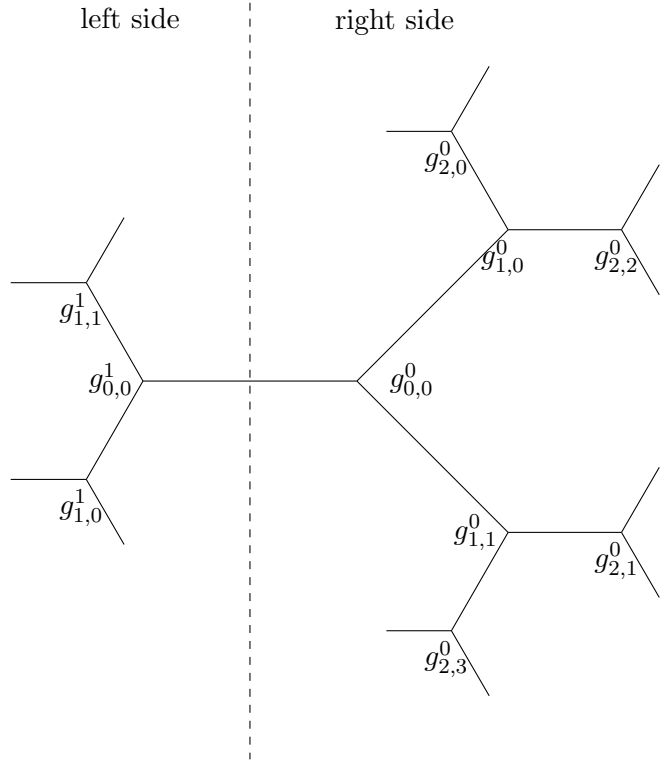
  The vertex $g_{n,\lambda}^i$ is said to be a point of radius $n$ or $-(n+1)$, depending on whether $i=0$ or $1$, indicating its distance from the vertex labeled by the identity matrix. For example, $g_{0,0}^{0}$ has radius $0$, $g_{1,0}^{0}$ and $g_{1,1}^{0}$ has radius 1 and $g_{0,0}^{1}$ has radius $-1$.
%
%

If $\lambda=\mu_{0}+ \mu_{1}2+\cdots + \mu_{n-1}2^{n-1} \in I_{n}$, then the truncation map $[.]_{n-1}: I_{n} \to I_{n-1}$ is defined as follows: $$[\lambda]_{n-1} := \begin{cases}
\mu_{0}+ \mu_{1}2+\cdots + \mu_{n-2}2^{n-2} & \text{ if } n\geq 2,\\
0 & \text{ if } n=1.
\end{cases}$$
We note that $[\lambda]_{n-1} = \lambda$ if $\lambda \in I_{n-1}$ and $\lambda-2^{n-1}$ if $\lambda\in I_{n}\setminus I_{n-1}$.
Since $g_{n,\lambda}^{i}$'s give the coset representatives of $G/K\Q_{2}^{\times}$, any element in $\mathcal{I}(\Sym^{r}R^2)$ is a linear combination of elementary functions of the form $\left[g_{n,\lambda}^{i}, v\right]$ where $v$ is an arbitrary element in $\Sym^{r}R^2$ and $i \in \{0,1\}$. We will now state a lemma giving the explicit action of $T^{+}$ and $T^{-}$ on these elementary functions.
\begin{lemma}\label{T}
 Let $P(X,Y) \in \Sym^{r}R^{2}$ where $R$ is a $\mathbb{Z}_{2}$-algebra. Let $\lambda \in I_{n},\,n\geq 0$. 
 Then the action of $T=T^+ + T^-$ on $\mathcal{I}(\Sym^{r}(R^{2}))$ can be explicitly stated as follows:
\begin{eqnarray*}
T^+\left[g_{n,\lambda}^0,\enspace P(X,Y)\right]& = &\left[g_{n+1,\lambda}^{0},\enspace P(X,2Y)\right] +  \left[g_{n+1, \lambda+2^n}^{0},\enspace  P(X,2Y-X)\right] \,\forall \,n\geq0,\\
T^-\left[g_{n,\lambda}^0,\enspace P(X,Y)\right] &= & \begin{cases} \left[g_{n-1,\lambda}^{0}, \enspace P(2X,Y)\right]& \text{ if } \lambda \in I_{n-1}\\
\left[g_{n-1,[\lambda]_{n-1}}^{0},\enspace P(2X, X+Y)\right]& \text{ if } \lambda \in I_{n}\setminus I_{n-1}
\end{cases}\forall\, n\geq1,\\
T^+\left[g_{n,\lambda}^1,\enspace P(X,Y)\right]& = &\left[g_{n+1,\lambda}^{1},\enspace P(2X,Y)\right] +  \left[g_{n+1, \lambda+2^n}^{1},\enspace P(2X-Y,Y)\right] \,\forall\,n\geq0,\\
T^-\left[g_{n,\lambda}^1,\enspace P(X,Y)\right] &= & \begin{cases} \left[g_{n-1,\lambda}^{1},\enspace P(X,2Y)\right]& \text{ if } \lambda \in I_{n-1}\\
\left[g_{n-1,[\lambda]_{n-1}}^{1},\enspace P(X+Y, 2Y)\right]& \text{ if } \lambda \in I_{n}\setminus I_{n-1}
\end{cases}\, \forall\, n\geq1,\\
T^-\left[g_{0,0}^0,\enspace P(X,Y)\right] &= & \left[g_{0,0}^{1},\enspace P(2X,Y)\right],\\
T^-\left[g_{0,0}^1,\enspace P(X,Y)\right] &= & \left[g_{0,0}^{0},\enspace  P(X,2Y)\right].
\end{eqnarray*}
\end{lemma}
\begin{proof}
We use the description of $T^{+}$ and $T^{-}$ given in \eqref{Tright} and \eqref{Tleft} with $g= g_{n,\lambda}^{i}$ with $\lambda\in I_n$, $n\geq 0$ and for $i=0$ or $1$. The last two formulae follow directly, as we know that $g_{0,0}^0\cdot g_{0,0}^1=g_{0,0}^1$ and $g_{0,0}^1\cdot g_{1,0}^0=g_{0,0}^0$. For the rest, we check that
\begin{eqnarray*}
g_{n,\lambda}^{0}\cdot g_{1,0}^{0} &=& g_{n+1,\lambda}^{0} \text{ and } \\
g_{n,\lambda}^{0}\cdot g_{1,1}^{0} &=& g_{n+1,\,\lambda+2^{n}}^{0} \text{ for } n\geq0,  \\
g_{n,\lambda}^{0} \cdot g_{0,0}^{1} &=& \left(\begin{smallmatrix}
        2^{n} & 2\lambda\\ 
        0 & 2
    \end{smallmatrix}\right)
   =\begin{cases}
    2\cdot g_{n-1,\lambda}^0  & \text{ if } \lambda \in I_{n-1},\\
 2\cdot g_{n-1, [\lambda]_{n-1}}^{0} \cdot \left(\begin{smallmatrix}
        1&1\\
        0&1
    \end{smallmatrix}\right) & \text{ if } \lambda \in I_{n}\setminus I_{n-1}
\end{cases} \text{ for } n\geq1.
  \end{eqnarray*}    
 Now applying \eqref{Tright} with the matrix identities above, we have
 \begin{eqnarray*}
T^{+}\left[g_{n,\lambda}^{0}, \enspace P(X,Y)\right] 
&=& \left[g_{n+1,\,\lambda}^{0},\enspace P(X,2Y)\right] +  \left[g_{n+1,\, \lambda+ 2^n}^{0},\enspace P(X,2Y-X)\right] \text{ for } n\geq0.\\
T^{-}\left[g_{n,\lambda}^{0},\enspace  P(X,Y)\right]&=& 
\begin{cases}
        \left[2\cdot g_{n-1,\lambda}^{0}, \enspace  P(2X,Y) \right] &\text{ if } \lambda \in I_{n-1},\\
        \left[2\cdot g_{n-1,[\lambda]_{n-1}}^{0}\cdot\left(\begin{smallmatrix}
        1&1\\
        0&1
\end{smallmatrix}\right), \enspace P(2X,Y) \right] &\text{ if } \lambda \in I_{n}\setminus I_{n-1}.
\end{cases} \\
\\
&=& 
\begin{cases}
        \left[ g_{n-1,\lambda}^{0}, \enspace  P(2X,Y) \right] &\text{ if } \lambda \in I_{n-1},\\
        \left[g_{n-1,[\lambda]_{n-1}}^{0}, \enspace P(2X,X+Y) \right] &\text{ if } \lambda \in I_{n}\setminus I_{n-1}
\end{cases} \text{ for } n\geq1.    \end{eqnarray*}
In the last step we use that $2\in\Q_2^\times$ acts trivially on $\Sym^{r}(R^{2})$, the matrix $\left(\begin{smallmatrix}
        1&1\\
        0&1
    \end{smallmatrix} \right)\in K$ takes $ P(2X,Y)$ to $P(2X,X+Y)$.
    
Similarly on the left side of the tree, we have 
 \begin{eqnarray*}
g_{n,\lambda}^{1} \cdot g_{1,0}^{0} 
=\left(\begin{smallmatrix}
        2 & 0\\ 
        2^{2}\lambda & 2^{n+1}
\end{smallmatrix}\right)
   &=& \begin{cases}
    2\cdot g_{n-1,\,\lambda}^1  & \text{ if } \lambda \in I_{n-1},\\
  2\cdot g_{n-1,\, [\lambda]_{n-1}}^{1}\cdot \left(\begin{smallmatrix}
        1&0\\
        1&1
\end{smallmatrix}\right) & \text{ if } \lambda \in I_{n}\setminus I_{n-1}
    \end{cases} \text{ for } n\geq1, \\
g_{n,\lambda}^{1}\cdot g_{1,1}^{0}= \left(\begin{smallmatrix}
    2 & 1\\
   2^2\lambda & \:\:2\lambda+2^{n+1}
\end{smallmatrix}\right)
&=& g_{n+1,\,\lambda+2^{n}}^{1}\cdot\left(\begin{smallmatrix}
2 & 1\\
-1 & 0
\end{smallmatrix}\right) \text{ and} \\
g_{n,\lambda}^{1}\cdot   g_{0,0}^{1}&= &g_{n+1,\,\lambda}^{1} \text{ for } n\geq0.
\end{eqnarray*}
Hence by formula \eqref{Tleft} and the matrix identities above we obtain
\begin{eqnarray*}
T^{+}\left[g_{n,\lambda}^{1},\enspace  P(X,Y)\right] &=& \left[g_{n+1,\,\lambda}^{1},\enspace P(2X,Y)\right] +  \left[g_{n+1,\, \lambda+2^n}^{1},
\,\left(\begin{smallmatrix}
    2&1\\
    -1&0
\end{smallmatrix}\right)\cdot P(X,2Y-X)\right]\\
&=&\left[g_{n+1,\,\lambda}^{1},\enspace P(2X,Y)\right] +  \left[g_{n+1, \,\lambda+2^n}^{1},\enspace P(2X-Y,Y)\right] \text{ for } n\geq 0,
\end{eqnarray*}
since the matrix $\left(\begin{smallmatrix}
2 & 1\\
-1 & 0
\end{smallmatrix}\right)\in K$ takes $P(X,2Y-X)$ to $P(2X-Y,Y)$. 
Further we have
\begin{eqnarray*}
T^{-}\left[g_{n,\lambda}^{1},\enspace P(X,Y)\right] &=&\begin{cases}
        \left[2 \cdot g_{n-1,\lambda}^{1}, \enspace  P(X,2Y) \right] &\text{ if } \lambda \in I_{n-1},\\
        \left[[2 \cdot g_{n-1,\,[\lambda]_{n-1}}^{1},\,\left(\begin{smallmatrix}
        1&0\\
        1&1
\end{smallmatrix}\right)\cdot  P(X,2Y) \right] &\text{ if } \lambda \in I_{n}\setminus I_{n-1},
    \end{cases} \\
    &=& \begin{cases}
        \left[ g_{n-1,\,\lambda}^{1}, \enspace  P(X,2Y) \right] &\text{ if } \lambda \in I_{n-1},\\
         \left[g_{n-1,\,[\lambda]_{n-1}}^{1}, \enspace P(X+Y,2Y) \right] &\text{ if } \lambda \in I_{n}\setminus I_{n-1}
    \end{cases} \text{ for } n\geq1.
 \end{eqnarray*}
 
The last step follows since $2 \in \Q_{2}^{\times}$ acts trivially on $\Sym^{r}(R^{2})$ and the matrix $\left(\begin{smallmatrix}
        1&0\\
        1&1
\end{smallmatrix}\right)\in K$ takes $P(X,2Y)$ to $P(X+Y,2Y)$.
\end{proof}

The following lemma is to be used while finding the image of the Hecke operator by transferring the computation from the left side of the tree to the right side of the same.
\begin{lemma}\label{T+toT-}
    Let $n \in \N \cup \{0\}$ and $\lambda \in I_{n}$, $w=\left(\begin{smallmatrix}
        0 & 1\\
        1 & 0
    \end{smallmatrix}\right)$ and let $P(X,Y)\in \Sym^rR^2$ where $R$ is a $\Z_{2}$-algebra. Then we have:  \begin{eqnarray}T^+\left[g_{n,\lambda}^{1}, \enspace w \cdot P(X,Y)\right]&=&\left(\begin{smallmatrix}
    0 & 1\\
    1 & \lambda
\end{smallmatrix}\right)\cdot T^+\left[g_{n+1,\,\lambda}^{0},\enspace P(X,Y)\right]
   \label{T+toT-a}\\
   &=&\left(\begin{smallmatrix}
        0 & 1\\
        1 & \lambda -2^n
\end{smallmatrix}\right)\cdot T^+\left[g_{n+1,\, \lambda+2^n}^{0},\enspace P(X,Y)\right].\label{T+toT-b}\\
T^-\left[g_{n,\lambda}^1, \enspace w \cdot P(X,Y)\right] &=& \left(\begin{smallmatrix}
       0&1\\
       1&\lambda
   \end{smallmatrix}\right)\cdot T^-\left[g_{n+1,\,\lambda}^{0}, \enspace P(X,Y)\right]
   \label{T+toT-c}\\
   &=& \left(\begin{smallmatrix}
       0 & 1\\
       1 & \lambda-2^n
   \end{smallmatrix}\right)\cdot T^-\left[g_{n+1,\,\lambda+2^n}^{0},\enspace P(X,Y)\right]. \label{T+toT-d}
  \end{eqnarray} 
    \end{lemma}

\begin{proof}
If $\lambda\in I_{n}$ for some $n\geq 0$, then we note that
\begin{equation}\label{matrixrel}
     \begin{psmallmatrix}
    0 & 1\\
    1 & \lambda
    \end{psmallmatrix}\cdot g_{n+2,\,\lambda}^{0}= g_{n+1,\,\lambda}^{1}\cdot w \text{ and } 
    \begin{psmallmatrix}
    0 & 1\\
    1 & \lambda
    \end{psmallmatrix}\cdot g_{n+2,\, \lambda + 2^{n+1}}^{0} = g_{n+1,\,\lambda+2^n}^{1}\cdot w .
    \end{equation}
    Using the formula for $T^+$ in Lemma \ref{T} and the matrix identities above, we get that
    \begin{eqnarray*}
     \begin{psmallmatrix}
    0 & 1\\
    1 & \lambda
\end{psmallmatrix}
\cdot T^+\left[g_{n+1,\lambda}^{0},\enspace P(X,Y)\right]&= & \left[g_{n+1,\lambda}^{1}\cdot w, \enspace P(X,2Y)\right] + \left[g_{n+1,\lambda+2^n}^{1}\cdot w,\enspace P(X,2Y-X)\right]\\
    &=& \left[g_{n+1,\lambda}^{1} ,\enspace P(Y,2X)\right] + \left[g_{n+1,\lambda+2^n}^{1},\enspace  P(Y,2X-Y)\right]\\
    &=& T^+\left[g_{n,\lambda}^{1},\enspace w\cdot P(X,Y)\right].   
    \end{eqnarray*}
    Now using Lemma \ref{T}, and  the matrix identities
    \begin{equation} \label{matrixrel2}
     \begin{psmallmatrix}
    0 & 1\\
    1 & \lambda-2^n
\end{psmallmatrix}\cdot g_{n+2,\,\lambda+2^n}^{0} = g_{n+1,\lambda}^{1}\cdot w \text{ and } 
    \begin{psmallmatrix}
    0 & 1\\
    1 & \lambda-2^n
\end{psmallmatrix} \cdot g_{n+2,\, \lambda+2^{n} +2^{n+1}}^{0} = g_{n+1,\,\lambda+2^n}^{1} \cdot w \hspace{1.5mm}\forall\, n \geq 0,
    \end{equation}
     we get that
    \begin{eqnarray*}
     \begin{psmallmatrix}
    0 & 1\\
    1 & \lambda-2^n
\end{psmallmatrix}
\cdot T^+\left[g_{n+1,\,\lambda+2^n}^{0},\enspace P(X,Y)\right]&= & \left[g_{n+1,\,\lambda}^{1}\cdot w, \enspace P(X,2Y)\right] + \left[g_{n+1,\,\lambda+2^n}^{1}\cdot w, \enspace P(X,2Y-X)\right]\\
    &=& \left[g_{n+1,\,\lambda}^{1} , \enspace P(Y,2X)\right] + \left[g_{n+1,\,\lambda +2^n}^{1}, \enspace P(Y,2X-Y)\right]\\
     &=& T^+\left[g_{n,\,\lambda}^{1}, \enspace  w\cdot P(X,Y)\right].   
    \end{eqnarray*}
To find the explicit expression for $T^{-}$, we first consider the matrix identity
\begin{equation*}
\begin{psmallmatrix}
    0 & 1\\
    1 & \lambda
\end{psmallmatrix} \cdot g_{n,\lambda}^{0} = \begin{cases}g_{n-1,\,\lambda}^{1}\cdot w &\text{ if } \lambda \in I_{n-1},\\
g_{n-1,\,[\lambda]_{n-1}}^{1}\cdot w \cdot \begin{psmallmatrix}
    1 & 1\\
    0 & 1
\end{psmallmatrix}& \text{ if } \lambda \in I_{n}\setminus I_{n-1}.
\end{cases}
\end{equation*}
Using the above identity and Lemma \ref{T} twice, we get that
\begin{eqnarray*}
\left(\begin{smallmatrix}
       0&1\\
       1&\lambda
   \end{smallmatrix}\right)\cdot T^-\left[g_{n+1,\,\lambda}^{0}, \enspace P(X,Y)\right]
&=& \begin{cases}
       \left[g_{n-1,\,\lambda}^{1},\enspace w \cdot P(2X,Y)  \right] & \text{ if } \lambda \in I_{n-1} \\
        \left[g_{n-1,\,[\lambda]_{n-1}}^{1}, \enspace w \cdot \begin{psmallmatrix}
    1 & 1\\
    0 & 1
\end{psmallmatrix}\cdot P(2X,Y)  \right] & \text{ if } \lambda \in I_{n}\setminus I_{n-1}    
\end{cases}\\
&=& \begin{cases}
       \left[g_{n-1,\,\lambda}^{1},\enspace P(2Y,X)  \right] & \text{ if } \lambda \in I_{n-1} \\
        \left[g_{n-1,\,[\lambda]_{n-1}}^{1}, \enspace  P(2Y, X+Y)  \right] & \text{ if } \lambda \in I_{n}\setminus I_{n-1}       
\end{cases}\\
& = & T^-\left[g_{n,\lambda}^1, \enspace P(Y,X)\right] =\, T^-\left[g_{n,\lambda}^1, \enspace w \cdot P(X,Y)\right]. 
\end{eqnarray*}
Similarly using Lemma \ref{T} and the matrix identities
\begin{equation*}
\begin{psmallmatrix}
    0 & 1\\
    1 & \lambda-2^n
\end{psmallmatrix} \cdot g_{n,\lambda}^{0} = \begin{cases}
g_{n-1,\lambda}^{1}\cdot w \cdot\begin{psmallmatrix}
    1 & -1\\
    0 & 1
\end{psmallmatrix} & \text{ if }  \lambda \in I_{n-1}\\
 g_{n-1,[\lambda]_{n-1}}^{1}\cdot w & \text{ if } \lambda \in I_{n}\setminus I_{n-1}
\end{cases}
\end{equation*}
 we get equation \eqref{T+toT-d}.
\end{proof}
 In the next part of the paper, we shall construct a number of functions with big supports spreading over both sides of the tree. Sometimes, we will have the functional value to be $P(X,Y)$ at some vertex in the right side and $w.P(X,Y)$ on a `compatible' vertex in the left side. In those cases we can apply the lemma above to compute the image of $T$ on the left half of the function, once the right half is computed.

\section{slope $\nu <1$}\label{slope<1}

Let $r\geq2$ and $0<\nu=v(a_2)<1$. Then Lemma \ref{RemarkBG09} and Proposition \ref{VrtoV0} together implies that the map $P$ factors through $$P:\mathcal{I}\left(\frac{V_r}{X_{r} + V_r^{(1)}}\right)\cong\mathcal{I}(V_0)\rightarrow \bar\Theta:=\bar\Theta_{k,a_2}.$$ We further analyze the kernel of $P: \mathcal{I}(V_{0}) \rightarrow \bar\Theta$ by studying functions on the  Bruhat-Tits tree taking values in $\mathcal{I}(\Sym^{r}\bar{\Q}_{2}^2)$. 
We shall now state a few combinatorial lemmas which are crucial in our analysis.
\begin{lemma}\label{CMB1}
	  If $1\leq n\leq r-1$ and $t=v(r-1)$,
		\begin{equation}
			v\left({r-1 \choose n}\right)+n \geq t+1.
		\end{equation}
	\end{lemma}
 \begin{proof}
  We check that
  \begin{align*}	
	n+v\left({r-1 \choose n}\right) &= n+ v\left( \frac{(r-1)!}{n!(r-1-n)!}\right)\\
		                  &=\displaystyle
                    n+\sum_{i=1}^{n}v(r-i)-v(n!)\\
	                    &\geq n+ v(r-1)-v(n!)\\
                     &= t+n-v(n!)
\end{align*}
Using Legendre's formula, we get that $n-v(n!) = S_{2}(n)$, where $S_{2}(n)$ is the sum of digits in the $2$-adic expansion of $n$. Clearly $S_{2}(n)\geq 1$. Therefore
$v({r-1 \choose n}) +n \geq t+1.$  
\end{proof}
\begin{remark}
    The lemma above is non-trivial only if $r$ is odd. It is vacuously true if $r$ is even, i.e., $t=0$.
\end{remark}

 \begin{lemma}\label{CMB2}
 For all $3\leq n\leq r$ and $t=v(r-1)$ we have
 \begin{equation}v\left({r\choose n}\right)+n\geq t+2.
 \end{equation}
\begin{proof}
For $n\geq 3$, we note that
    \begin{align*}	
	n+v\left({r \choose n}\right) 
&=\displaystyle
     n+\sum_{i=0}^{n-1}v(r-i)-v(n!)\\
	        & =S_2(n)+\sum_{i=0}^{n-1}v(r-i).
\end{align*}
If $r$ is even, then the RHS above is $\geq 1+v(r)+v(r-1)+v(r-2)\geq 3=t+3>t+2$. If $r$ is odd and $n=3$, then $S_2(n)=2$ and so the RHS 
$= 2+v(r-1)=t+2$. 
If $r$ is odd and $n>3$, then RHS $\geq 1+v(r-1)+v(r-3)\geq t+2$. 
\end{proof}     
 \end{lemma}

Next we define a number of special polynomials (for $r\geq2$) to be used in our computations:

\begin{eqnarray}
\label{Fdef} F(X,Y) &=& Y^r-X^{r-1}Y
\\
\label{Gdef} G(X,Y) &=& X^r-XY^{r-1}
\\
\label{Hdef} H(X,Y) &=& \displaystyle\sum_{0<j<r}{r \choose j}X^{r-j}Y^{j}\\
\label{H'}H'(X,Y) &=& H(X,Y)-rX^{r-1}Y-rX^{r-2}Y^2\\
\label{H''} H''(X,Y) &=& H'(Y,X)= H(X,Y)-rXY^{r-1}-rX^2Y^{r-2}.
\end{eqnarray}
 Keeping in mind the formula for the Hecke operator $T$, we next compute few congruences for the polynomials defined above. Let 
$\delta:\N \to \{0,1\}$ be the characteristic function defined as
\begin{equation}\label{delta}\delta(r) = \begin{cases}
    1 & \text{ if } r \text{ is even,}\\
    0 & \text{ if } r \text{ is odd.}\end{cases}  \end{equation}
\begin{lemma}\label{F}
 We have the following congruences modulo $2^{t+2}$ where $t=v(r-1)$:
\begin{enumerate}[label=(\alph*)]
\item $F(X,2Y) \equiv -2X^{r-1}Y$.\\
\item
    $F(X,2Y-X) 
    \equiv 2\delta(r)X^{r}+2(r-1)X^{r-1}Y- 2r(r-1)X^{r-2}Y^{2}.$ \label{Fb}
     \\
\item $F(2X,Y) \equiv
\begin{cases}
Y^{r} - 2^{r-1}X^{r-1}Y &\text{ if } r=2,3 \\
Y^{r} &\text{ for } r\geq 4.
\end{cases} \label{Fc}$
\end{enumerate} 
 \end{lemma}

\begin{proof}
For all $r>1$, we have $r>r-2\geq v(r-1)=t$, thus $r\geq t+2$. Therefore
\begin{eqnarray*}
   F(X,2Y)& = &2^{r}Y^{r}-2X^{r-1}Y \equiv -2X^{r-1}Y \mod 2^{t+2}.\end{eqnarray*}
   For part \ref{Fb}, we check that
   \begin{eqnarray*}
   F(X,2Y-X)& =& (2Y-X)^{r}-X^{r-1}(2Y-X)\\
    \label{F(X,2Y-X)}
            &=& \sum_{j=0}^{r}2^j{r \choose j}(-X)^{r-j}Y^{j}-2X^{r-1}Y+ X^{r}\\
            &\equiv & \sum_{j=0}^{2}2^{j}{r \choose j}(-X)^{r-j}Y^{j}-2X^{r-1}Y+ X^{r}\mod 2^{t+2},
\end{eqnarray*}
 where the last congruence follows from Lemma \ref{CMB2}.  
 If $r$ is odd, then the congruence above simplifies as
 $$F(X, 2Y-X)\equiv  2(r-1)X^{r-1}Y -2r(r-1)X^{r-2}Y^{2} \mod 2^{t+2}.
$$
 If $r$ is even, then $t=0$ and we get
\begin{eqnarray*}             
F(X, 2Y-X)&\equiv & 2X^{r}-2(r+1)X^{r-1}Y + 2r(r-1)X^{r-2}Y^{2}\mod 2^{t+2}\\
&\equiv& 2X^{r}+ 2(r-1)X^{r-1}Y \mod 2^{2}.
\end{eqnarray*}
For part \ref{Fc}, we compute
\begin{eqnarray*}
    F(2X,Y)	&= &Y^{r}-2^{r-1}X^{r-1}Y\\
  &\equiv&\begin{cases} Y^{r}-2^{r-1}X^{r-1}Y \mod 2^{t+2} &\text{ for } r=2,3,\\
  Y^{r}\mod 2^{t+2} &\text{ for } r \geq 4,\, \text{ since } r-1\geq t+2
   \end{cases}
\end{eqnarray*}
hence the lemma follows.
\end{proof}
 
   \begin{lemma}\label{H}
  We have the following congruences modulo $2^{t+2}$ where $t=v(r-1)$: 
  		\begin{enumerate}[label=(\alph*)]
        \item $H(X,2Y) \equiv 2rX^{r-1}Y+2r(r-1)X^{r-2}Y^{2}\equiv
    \begin{cases}
        2rX^{r-1}Y + 2r(r-1)X^{r-2}Y^{2} &\text{ for odd } r,
        \\
        0 &\text{ for even } r.
        \end{cases}$
  		\item {\small $H(X,2Y-X) \equiv -2\delta(r)X^r-2rX^{r-1}Y+2r(r-1)X^2Y^{r-2}
        \equiv
    \begin{cases}
        -2rX^{r-1}Y + 2r(r-1)X^{r-2}Y^{2}  &\text{for odd } r,\\
        -{{2}}X^r &\text{for even } r.
        \end{cases}$ }\label{Hb}
  		\item $H(2X,Y) \equiv  2rXY^{r-1} + 2r(r-1)X^{2}Y^{r-2}  \equiv   \begin{cases}
        2rXY^{r-1} + 2r(r-1)X^{2}Y^{r-2}
        &\text{for odd } r,
        \\
        0 &\text{ for even } r.
        \end{cases}$ \label{Hc}
  		\end{enumerate}   
   \end{lemma}

\begin{proof}
Applying Lemma \ref{CMB2}, we get
    \begin{eqnarray*}
   	     H(X,2Y)& = & \displaystyle\sum_{0<j<r}2^{j}{r \choose j}X^{r-j}Y^{j} \equiv 2rX^{r-1}Y +2r(r-1)X^{r-2}Y^{2} \mod 2^{t+2}. 
    \end{eqnarray*}
    It is clear from the above equation that $H(X,2Y) \equiv 0\mod 2^{2}$, when $r$ is even. 
    \begin{eqnarray*}
   	     H(X,2Y-X)& = & \displaystyle\sum_{0<j<r}{r \choose j}X^{r-j}(2Y-X)^{j} =  \displaystyle\sum_{j=1}^{r-1}{r \choose j}\sum_{i=0}^{j}2^{i}(-1)^{j-i}{j \choose i}X^{r-i}Y^{i}	     	\\
         &=&  \displaystyle\sum_{j=1}^{r-1}\sum_{i=0}^{j}2^{i}(-1)^{j-i}{r \choose j}{j \choose i}X^{r-i}Y^{i}	 \\&=&\sum_{j=1}^{r-1}\sum_{i=0}^{j}(-1)^{j-i}2^{i}{r \choose i}{r-i \choose j-i}X^{r-i}Y^{i}	 \\
         &\equiv &\sum_{i=0}^{2}\sum_{j=1}^{r-1}(-1)^{j-i}2^{i}{r \choose i}{r-i \choose j-i}X^{r-i}Y^{i}	\mod 2^{t+2}
         \end{eqnarray*}
         by Lemma \ref{CMB2}. Next we calculate the coefficients of $X^{r-i}Y^i$ in the expression above, for $i=0,1$ and $2$.
    The coefficient of $X^{r}$ in the above expression is given by 
   	   \begin{align*}
   	    \displaystyle \sum_{j=1}^{r-1}(-1)^j{r \choose j}= & (1-1)^{r}- {r \choose 0}(-1)^{0} - {r \choose r}(-1)^{r}	= -2 \cdot \delta(r).
      \end{align*}
   The coefficient of $X^{r-1}Y$ is given by 
   \begin{align*}
     2r\displaystyle \sum_{j'=0}^{r-2}{r-1 \choose j'}(-1)^{j'} = 2r\left(\sum_{j'=0}^{r-1}{r-1 \choose j'}(-1)^{j'}- {r-1 \choose r-1}(-1)^{r-1}\right)
     = (-1)^r2r.
   \end{align*}
Similar calculation shows that the coefficient of $X^{r-2}Y^{2}$ is $ (-1)^{r-1}2r(r-1).$

Note that if $r$ is even, then $t=0$ and so the coefficients of $X^{r-i}Y^i$ vanish mod $2^{t+2}$ for $i=1,2$. If $r$ is odd, the $X^r$-term vanishes mod $2^{t+2}$. In both cases, adding all three monomials computed above, we get part \ref{Hb} of the statement.

For part \ref{Hc}, we check that
\begin{eqnarray*}
H(2X,Y) &= &\displaystyle\sum_{0<j<r}2^{r-j}{r \choose j}X^{r-j}Y^{j} = \displaystyle\sum_{j'=1}^{r-1}2^{j'}{r \choose j'}X^{j'}Y^{r-j'}\\ 
   	   &\equiv& 2rXY^{r-1}+2r(r-1)X^{2}Y^{r-2} \mod 2^{t+2}.
\end{eqnarray*}
  The last congruence follows from Lemma \ref{CMB2}.	   
\end{proof}
\begin{lemma}\label{hprime}
We have the following congruences modulo $2^2$ for $H'(X,Y)$.
\begin{enumerate}[label=(\alph*)]	
\item  $H'(X,2Y) \equiv 0 \mod 2^2.$ \label{H'a}
\item  $H'(X,2Y-X) \equiv -2\delta(r)X^{r}\equiv\begin{cases}
 0\mod 2^2& \text{ if } r \text{ is odd, }\\
 -2X^{r} \mod 2^2 & \text{ if }r \text{ is even}.
 \end{cases}$ \label{H'b}
	\item  $H'(2X,Y)\equiv \begin{cases}
   -2Y^2 \mod 2^2 & \text{ if }r=2,\\
    0  \mod 2^{2} & \text{ if } r=3,\\
    2rXY^{r-1} \mod 2^{2} & \text{ if }r>3,
    \end{cases} $    \label{H'c}\\
    
implying that it vanishes mod $2^2$ whenever $r>3$ is even.
 \end{enumerate}
\end{lemma}
\begin{proof}
We have $H'(X,Y)= H(X,Y) -rX^{r-1}Y -rX^{r-2}Y^{2}$. 
 Using the definition of $H(X,Y)$ in \eqref{Hdef}, we can see that $Y^2$ divides $H'(X,Y)$. Thus $H'(X,2Y) \equiv 0\mod 2^2$, proving part \ref{H'a}.
 
  For part $\ref{H'b}$, we note that 
  \begin{eqnarray*}
   H'(X,2Y-X) &=& H(X,2Y-X)- rX^{r-1}(2Y-X) - rX^{r-2}(2Y-X)^{2}\\
   &\equiv & H(X,2Y-X) -2rX^{r-1}Y \mod 2^2.
 \end{eqnarray*} 
Now we apply Lemma \ref{H}. When r is odd, we get
 \begin{equation*}
   H'(X,2Y-X) \equiv  -2rX^{r-1}Y + 2r(r-1)X^{r-2}Y^{2}-2rX^{r-1}Y \equiv 0 \mod 2^{2}.
   \end{equation*}
   and when r is even, we get
   \begin{equation*}
     H'(X,2Y-X) \equiv  -2X^{r} -2rX^{r-1}Y 
     \equiv  -2X^{r} \mod 2^{2}.   
   \end{equation*}
   For part \ref{H'c}, we note that 
   \begin{eqnarray*}
   H'(2X,Y) &\equiv & 2rXY^{r-1} -2^{r-2}rX^{r-2}Y^2\mod 2^2\\
   & \equiv & \begin{cases}
   -2Y^2 \mod 2^2 & \text{ if }r=2,\\
    0  \mod 2^{2} & \text{ if } r=3,\\
    2rXY^{r-1} \mod 2^{2} & \text{ if }r>3.
    \end{cases} 
   \end{eqnarray*}
\end{proof}
Let us define $\alpha':= \dfrac{a_{2}^{2}-2r^{2}}{2a_{2}}$, and let $\tau' := v\left( \alpha'\right)$. For the rest of this section, let $t=v(r-1) $ as above.
 The following lemma will be useful in our calculation.
\begin{lemma}\label{Lemmatomt}
If $\tau'\geq t$, then $r$ is odd and $v(a_{2})=\frac{1}{2}$.
   \end{lemma}
\begin{proof}
Using basic valuation theory and the fact that $0<v(a_2)<1$, we note that
\begin{eqnarray}
    \tau' = v(a_{2}^2-2r^2)-1-v(a_{2})=\begin{cases}
    v(a_2)-1, & \text{ if } r \text{ is even }\\
    v(a_2)-1, & \text{ if } r \text{ is odd}, v(a_2)<\frac{1}{2}\\
    -v(a_2),& \text{ if } r \text{ is odd}, v(a_2)>\frac{1}{2}.
    \end{cases}
\end{eqnarray}
Clearly $\tau'<0$ in the cases listed above.
Thus if $\tau'\geq t \,(\geq 0)$, then $r$ must be odd and further one must have $v(a_2)=\frac{1}{2}$.
\end{proof}
Let us now define \begin{equation}\label{beta}
\beta := \begin{cases}
    \alpha' &\text{ if } \tau'<t\\
    r(r-1) &\text{ if } \tau'\geq t.
\end{cases}
\end{equation}
By Lemma \ref{Lemmatomt}, we can see that $v(\beta ) = \min(\tau',t)$. Recall that $\wp$ is the maximal ideal in $\bar{\Z}_{2}$. 
\begin{lemma}\label{rem1}
The quantity $\frac{a_{2}^2-2r}{a_{2}^2\beta}$ is always integral. If $\tau'\geq t$, then it is a unit congruent to $\frac{2}{a_2^2}\mod\wp$. Otherwise, it vanishes mod $\wp$.
\end{lemma}
\begin{proof}
We have $a_{2}^2-2r = a_{2}^2-2r^2 + 2r(r-1)
     =2a_{2}\alpha'+2r(r-1)$, thus
$$\dfrac{a_{2}^2-2r}{a_{2}^2\beta} = \dfrac{2\alpha'}{a_{2}\beta} + \dfrac{2r(r-1)}{a_{2}^2\beta}\equiv\dfrac{2r(r-1)}{a_{2}^2\beta}\mod\wp.$$
We now compute the valuation of $\frac{2r(r-1)}{a_{2}^2\beta}$  case by case. 
     \begin{enumerate}
         \item When $\tau'<t$ and $v(a_{2})\leq \dfrac{1}{2}$, it is easy to see that the term has positive valuation.
         \item If $\tau'<t$ and $v(a_{2})>\dfrac{1}{2}$, then  $\tau'=-v(a_{2})$ or $v(a_2)-1$, depending on parity of $r$. In any case $\tau'<0$, and we have $v\left(\frac{2r(r-1)}{a_{2}^2\beta}\right)\geq 2-2v(a_2)-\tau'>0$. 
         \item If $\tau'\geq t$, applying \eqref{beta} the term becomes $\frac{2}{a_{2}^2}$, which is a unit by Lemma \ref{Lemmatomt}. 
             \end{enumerate}
    Thus all the claims are proved.         
\end{proof}
We will now analyze certain functions on the Bruhat-Tits tree. Recall that each function in the induction space $\mathcal{I}(\Sym^{r}(\bar{\mathbb{{Q}}}^{2}_{2}))$ is supported on finitely many vertices on the Bruhat-tits tree and hence can be denoted by a finite sum of elementary functions on it. 
Given $f,g \in \mathcal{I}(\Sym^{r}(\bar{\mathbb{{Q}}}^{2}_{2}))$, we write that $f \equiv g \mod \wp$ if $f-g \in \mathcal{I}(\Sym^{r}(\wp^{2}))$. We also recall from \S 2 that given $n\geq 1$, the matrices $g_{n,\lambda}^{0}$ $($or $ g_{n-1,\lambda}^{1})$ for $\lambda \in I_{n}$\,(or $I_{n-1}$) represent vertices of radius $n\,($or $-n)$ and that the identity matrix $g_{0,0}^{0}$ is considered to be of radius zero. With this convention, we are now ready to prove the main theorem of this section.   
\begin{theorem}\label{irred}
Let $k=r+2\geq 4$ and $0<v(a_2)<1$, and $\bar{\Theta}:=\Theta_{k,a_2} \otimes \bar{\F}_{2}.$
\begin{enumerate}
    \item If $\tau'<t$, then the map $P$ factors as $P:\dfrac{\mathcal{I}(V_{0})}{(T)} \to \bar{\Theta}$. 
    \item If $\tau'\geq t$, then $P$ factors as $P:\dfrac{\mathcal{I}(V_0)}{\left(T^{2}-cT+1\right)}\to \bar{\Theta}$  where, $c=\overline{\frac{\alpha'}{r-1}} \in \bar{\F}_{2}.$ \end{enumerate} 
\end{theorem}
\begin{proof}
Consider the function $f=\sum f_n\in \mathcal{I}(\Sym^{r}\bar{\mathbb{{Q}}}^{2}_{2})$ given by
		\begin{eqnarray*}
			f_0	&=& \left[\mathrm{Id},\enspace -\dfrac{H(X,Y)}{2a_2\beta}-\dfrac{r(r-1)(X^r+F(X,Y))}{a_2\beta} + \dfrac{\delta(r)X^{r-1}Y}{a_2\beta}\right], \\[1ex]
			f_1 &=& \left[g_{1,0}^0,\enspace \dfrac{F(X,Y)}{2\beta}+\dfrac{r(r-1)H'(X,Y)}{a_{2}^2\beta}\right] -\left[g_{1,1}^0,\enspace \dfrac{F(X,Y)}{2\beta}+\dfrac{r(r-1)H'(X,Y)}{a_{2}^2\beta}\right],\\[1ex]
			f_2 &=& \left[g_{2,0}^0,\enspace \dfrac{rF(X,Y)}{a_2\beta}\right]+\left[g_{2,2}^0, \enspace \dfrac{r(r-1)F(X,Y)}{a_2\beta}\right]
			-\left[g_{2,1}^0,\enspace \dfrac{rF(X,Y)}{a_2\beta}\right]\\ & &-\left[g_{2,3}^0,\enspace \dfrac{r(r-1)F(X,Y)}{a_2\beta}\right],\\[1ex]
			f_{-1}& =& \left[g_{0,0}^1,\enspace \dfrac{G(X,Y)}{2\beta}+\dfrac{r(r-1)H''(X,Y)}{a_2^2\beta}\right], 
			\\[1ex]
			f_{-2}& =& \left[g_{1,0}^1,\enspace \dfrac{rG(X,Y)}{a_{2}\beta}\right]+ \left[g_{1,1}^1,\enspace \dfrac{r(r-1)G(X,Y)}{a_2\beta}\right], \text{ and }\\[1ex]
		  f_n &=&
            \frac{r}{a_2\beta}
            \left(\frac{2r}{a_2}\right)^{n-2}\cdot\left(\left[g_{n,0}^0,\enspace F(X,Y)\right]-\left[g_{n,1}^0, \enspace F(X,Y)\right]\right),\\
		f_{-n}&=& \frac{r}{a_2\beta}
            \left(\frac{2r}{a_2}\right)^{n-2}\cdot\left[g_{n-1,0}^1,\enspace G(X,Y)\right]\quad\text{for  }  3 \leq n \leq \frac{t+1}{1-v(a_{2})+v(r)}+1.
		\end{eqnarray*}
  where $\delta$ is the characteristic function from \eqref{delta} and let $\beta$ is as defined in \eqref{beta}. As we know $\left(T-a_{2}\right)f$ is supported on a finite number of radii of the Bruhat-Tits tree, we now compute $\left(T-a_{2}\right)f$ radius by radius. We first treat the case $r>3$.
  
\noindent \textbf{ \underline{Radius $0$}}\\
  For $r>3$, we compute modulo $\wp$ (where $\wp\mid 2$).
  \begin{eqnarray*}
T^-f_1 &=&
\left[\mathrm{Id},\enspace\dfrac{F(2X,Y)}{2\beta}+\dfrac{r(r-1)H'(2X,Y)}{a_2^2\beta}- \dfrac{F(2X,X+Y)}{2\beta}-\dfrac{r(r-1)H'(2X,X+Y)}{a_{2}^2\beta}\right]\nonumber\\ 
  &\equiv& 
      \left[1,\enspace \dfrac{1}{2\beta}Y^r+\dfrac{2r^2(r-1)}{a_2^2\beta}XY^{r-1}-\dfrac{1}{2\beta}(X+Y)^r-\dfrac{2r^2(r-1)X(X+Y)^{r-1}}{a_2^2\beta}\right] \\
  &\equiv& 
  \begin{cases}
      \left[1,\enspace -\dfrac{X^r+H(X,Y)}{2\alpha'}\right]\mod \wp, &\text{ when } \tau'<t,\\
      \left[1,\enspace  -\dfrac{X^r+H(X,Y)}{2r(r-1)}+ \dfrac{2r}{a_{2}^2}\left( \displaystyle\sum_{j=0}^{r-2}{r-1 \choose j}X^{r-j}Y^{j}\right)\right]\mod \wp, &\text{ when } \tau'\geq t .     
      \end{cases}
 \end{eqnarray*}
  
Here the first congruence follows from Lemmas \ref{F} and \ref{hprime} since $v(\beta)=\min(\tau',t)$, and the first part of the last congruence further follows since $v(2r^2(r-1))>\tau'+2v(a_2)$ whenever $\tau'<t$ by the following observations.
\begin{enumerate}
    \item If $r$ is odd and $v(a_2)\leq 1/2$ then LHS $= t+1>\tau'+1\geq$ RHS. 
    \item If $r$ is odd and $v(a_2)> 1/2$, then $\tau'=-v(a_2)$ and RHS = $v(a_2)<1<t+1=$ L.H.S.
    \item If $r$ is even, then $v(\beta)=\tau'<t$ and further LHS $= V(2r^2)+t>2+\tau'>2v(a_{2})+\tau'= $RHS.
    \end{enumerate} 
If $\tau'\geq t$, then $\dfrac{2r}{a_{2}^{2}}$ is a unit by Lemma \ref{Lemmatomt}. 
Further $\displaystyle\sum_{j=1}^{r-2}{r-1 \choose j}X^{r-j}Y^{j} \in V_{r}^{(1)}$ by Lemma \ref{Vr*} since 
$\displaystyle\sum_{j=1}^{r-2}{r-1 \choose j} = 2^{r-1}-2 \equiv 0 \mod 2$ and also $X^r\in X_r$. Thus we can simplify
\begin{eqnarray}\label{T-f1}
T^-f_{1} &\equiv& \begin{cases}
      \left[1,-\dfrac{X^r+H(X,Y)}{2\alpha'}\right] &\text{ when } \tau'<t,\\
      \left[1,-\dfrac{X^r+H(X,Y)}{2r(r-1)} \right]&\text{ when } \tau'\geq t      
      \end{cases}\\
     &\equiv & \left[1,-\dfrac{X^r+H(X,Y)}{2\beta}\right]\mod \wp,\, X_r+V_r^{(1)}.\nonumber
\end{eqnarray}
Now we compute 
\begin{eqnarray}\label{a2f0}
    -a_2f_0&\equiv &\left[1,\dfrac{1}{2\beta}H(X,Y)+\dfrac{r(r-1)}{\beta}(X^r+F(X,Y))-\dfrac{\delta(r)X^{r-1}Y}{\beta}\right]\nonumber\\
    &\equiv& \begin{cases}
\left[1,\dfrac{1}{2\beta}H(X,Y) \right] & \text{ if }\tau'<t\\
\left[1, \dfrac{1}{2\beta}H(X,Y)+X^r+F(X,Y)\right]& \text{ if }\tau'\geq t    \end{cases}
\end{eqnarray}
Indeed if  $r$ is even, then $\tau'<t,\beta = \alpha'$ and $\delta(r)=1$. The term $\dfrac{-X^{r-1}Y}{\alpha'}$ vanishes $\mod\wp$ as $v(\alpha') = v(a_{2})-1<0$ in this case. When $r$ is odd, then $\delta(r)=0$ and further if $\tau'<t$ , then $v\left(\dfrac{r(r-1)}{\beta}\right)$ is of positive valuation and hence vanishes$\mod\wp$. \begin{eqnarray}\label{T-f-1}
T^-f_{-1} & \equiv&
\left[\mathrm{Id},\enspace\dfrac{G(X,2Y)}{2\beta}+\dfrac{r(r-1)H''(X,2Y)}{a_2^2\beta}\right]\nonumber\\
  &\equiv& \left[\mathrm{Id},\dfrac{1}{2\beta}X^r+\dfrac{2r^2(r-1)}{a_2^2\beta}X^{r-1}Y\right] \nonumber\\  & \equiv & \begin{cases}
      \left[\mathrm{Id},\dfrac{1}{2\beta}X^r\right] &  \text{ if }\tau'<t \\
\left[\mathrm{Id},\dfrac{1}{2\beta}X^r+\dfrac{2r}{a_2^2}X^{r-1}Y\right]& \text{ if }\tau'\geq t
  \end{cases}
  \end{eqnarray}
  
The last congruence follows as $v(2r^2(r-1))>\tau'+2v(a_2)$ whenever $\tau'<t$, as shown in the computation of $T^-f_{1}$. 

Now adding congruences \eqref{T-f1}, \eqref{a2f0} and \eqref{T-f-1} we have ,
  %
\begin{eqnarray}
   T^-f_1-a_2f_0-T^-f_{-1}&\equiv&
   \begin{cases}
       0\mod\wp, &\tau'<t\\
       \left[\mathrm{Id}, X^r+Y^r+\left(\dfrac{2r}{a_{2}^2}-1\right)X^{r-1}Y \right], &\tau'\geq t.
       \end{cases}\nonumber \\
   &\equiv & 0 \mod \wp, X_{r}+V_{r}^{(1)}
\end{eqnarray}

The last congruence follows since from Lemma \ref{rem1} we have,
    $v\left(\frac{a_{2}^2-2r}{a_{2}^2}\right)\geq v(\beta)$. Further if $\tau'\geq t$, then $\beta=r(r-1)$ and hence $v(\beta)\geq 1.$
    
\noindent\textbf{\underline{Radius $1$}}\\
Expanding $H'(X,Y)$ as $H(X,Y)-rX^{r-1}Y-rX^{r-2}Y^2$ in $f_{1}$, we have,
\begin{eqnarray}\label{a2f1odd}
    -a_2f_1 = \left[g_{1,0}^0, \dfrac{-a_2F(X,Y)}{2\beta} + \dfrac{-r(r-1)}{a_{2}\beta}H(X,Y)+ \dfrac{r^2(r-1)}{a_2\beta}(X^{r-1}Y+X^{r-2}Y^2)\right] \nonumber\\
    - \left[g_{1,1}^0,\dfrac{-a_2F(X,Y)}{2\beta} + \dfrac{-r(r-1)}{a_{2}\beta}H(X,Y)+\dfrac{r^2(r-1)}{a_2\beta}(X^{r-1}Y+X^{r-2}Y^2) \right]
\end{eqnarray} 

We now compute $T^-f_2$, using Lemma \ref{F}. 
\begin{eqnarray*}
    T^-f_2 &=& \left[g_{1,0}^0, \dfrac{rF(2X,Y)}{a_2\beta} + \dfrac{r(r-1)F(2X,X+Y)}{a_2\beta}\right]\\
    &-& \left[g_{1,1}^{0}, \dfrac{rF(2X,Y)}{a_2\beta} + \dfrac{r(r-1)F(2X,X+Y)}{a_2\beta}\right]\\
    &\equiv& \left[g_{1,0}^{0}, \dfrac{rY^r+r(r-1)(X+Y)^r}{a_2\beta} \right] - \left[g_{1,1}^{0}, \dfrac{rY^r+r(r-1)(X+Y)^r}{a_2\beta}  \right].
    \end{eqnarray*}
    For ease of calculation, we shall write the above expression as,
    \begin{eqnarray}\label{T-f2odd}
     T^-f_2 &\equiv& \left[g_{1,0}^{0}, \dfrac{r^{2}Y^{r}}{a_{2}\beta}+ \dfrac{r(r-1)X^r}{a_2\beta}+ \dfrac{r(r-1)H(X,Y)}{a_{2}\beta}\right]\nonumber\\
     &-& \left[g_{1,1}^{0},  \dfrac{r^{2}Y^{r}}{a_{2}\beta}+ \dfrac{r(r-1)X^r}{a_2\beta}+ \dfrac{r(r-1)H(X,Y)}{a_{2}\beta}\right]. \label{T-f2}
     \end{eqnarray}

   We now compute $T^+f_{0}$.  
\begin{eqnarray*}
   T^+f_0 &=& \left[g_{1,0}^0, \dfrac{-H(X,2Y)}{2a_{2}\beta} - \dfrac{r(r-1)(X^r+F(X,2Y))}{a_2\beta} + \dfrac{2\delta(r)X^{r-1}Y}{a_2\beta}\right]\\
   &+& \left[g_{1,1}^{0}, \dfrac{-H(X,2Y-X)}{2a_{2}\beta} - \dfrac{r(r-1)(X^r+F(X,2Y-X))}{a_2\beta} +\dfrac{\delta(r)(2X^{r-1}Y- X^r)}{a_2\beta}\right]  
\end{eqnarray*}
We note from Lemma \ref{F}, that the coefficients of monomials in $F(X,2Y)$ and $F(X,2Y-X)$ computed$\mod 2^{t+2}$ have valuation at least 1. Hence both $\dfrac{r(r-1)}{a_2\beta}F(X,2Y)$ and $\dfrac{r(r-1)}{a_2\beta}F(X,2Y-X)$ vanishes$\mod \wp$. Even though  we have $\delta(r)=1$ when $r$ is even, the extra term $\dfrac{2X^{r-1}Y}{a_{2}\beta}$ still vanishes as $\dfrac{2}{a_{2}\beta}$ has positive valuation. Thus we have,
\begin{eqnarray} \label{T+f0odd}
  T^+f_0 &\equiv& 
  \left[g_{1,0}^0, \dfrac{-rX^{r-1}Y - r(r-1)X^{r-2}Y^2}{a_{2}\beta} - \dfrac{r(r-1)X^r}{a_{2}\beta} \right] \\
   &+&\left[g_{1,1}^{0}, \dfrac{rX^{r-1}Y-r(r-1)X^{r-2}Y^{2}}{a_2\beta}-\dfrac{r(r-1)X^r}{a_2\beta}  \right]\nonumber \mod \wp. 
\end{eqnarray}
It follows from the definition of $\beta$ and Lemma \ref{Lemmatomt} that the expressions for both $T^+f_0$ and $T^{-}f_{2}$ above vanish modulo $\wp$ when $r$ is even. 

On adding the congruences in \eqref{a2f1odd}, \eqref{T-f2odd} and \eqref{T+f0odd}, making the obvious cancellations and collecting the like terms, we get, 
\begin{eqnarray}\label{rad1}
T^+f_{0}-a_2f_1+T^-f_2 &\equiv&\left[ g_{1,0}^{0}, \dfrac{(a_2^2-2r + 2r^2(r-1))}{2a_2\beta}X^{r-1}Y + \dfrac{r(r-1)^2}{a_2\beta}X^{r-2}Y^2 - \dfrac{(a_2^2-2r^2)}{2a_2\beta}Y^r \right]\nonumber\\
&+ &\left[ g_{1,1}^{0}, \dfrac{(-a_2^2+2r - 2r^2(r-1))}{2a_2\beta}X^{r-1}Y + \dfrac{-r(r-1)^2}{a_2\beta}X^{r-2}Y^2 + \dfrac{(a_2^2-2r^2)}{2a_2\beta}Y^r \right]\nonumber\\
&\equiv&  \left[ g_{1,0}^{0}, \dfrac{(a_2^2-2r^2 + 2r(r^2-1))}{2a_2\beta}X^{r-1}Y - \dfrac{\alpha'}{\beta} Y^r \right]\nonumber\\
&&\enspace + \left[ g_{1,1}^{0}, \dfrac{(-a_2^2+2r^2 - 2r(r^2-1))}{2a_2\beta}X^{r-1}Y + \dfrac{\alpha'}{\beta} Y^r \right]\nonumber\\
&\equiv& \left[ g_{1,0}^{0}, \dfrac{\alpha'}{\beta}X^{r-1}Y  \right]
+ \left[ g_{1,1}^{0}, -\dfrac{\alpha'}{\beta}X^{r-1}Y  \right] \mod \wp,X_{r}+V_{r}^{(1)}.
\end{eqnarray} 

\noindent\textbf{\underline{Radius $2$}}\\
Assuming $r>3$ and using Lemma \ref{F}\ref{Fc},
 we have 
 \begin{eqnarray}
   T^-f_{3} 
 \equiv
   \left[g_{2,0}^0, \dfrac{2r^2Y^r}{a_2^2\beta} \right] -\left[g_{2,1}^{0},  \dfrac{2r^2Y^r}{a_2^2\beta} \right]\mod\wp
   \label{T-f3}   
\end{eqnarray}
 \begin{eqnarray}\label{a2f2}
   -a_2f_{2} &=& \left[g_{2,0}^0,\enspace \dfrac{-rF(X,Y)}{\beta}\right]-\left[g_{2,2}^0, \enspace \dfrac{r(r-1)F(X,Y)}{\beta}\right]+\left[g_{2,1}^0,\enspace \dfrac{rF(X,Y)}{\beta}\right]\nonumber\\ & &+\left[g_{2,3}^0,\enspace \dfrac{r(r-1)F(X,Y)}{\beta}\right], 
 \end{eqnarray}
\begin{eqnarray*}
T^+f_{1}&=&  \left[g_{2,0}^0,\enspace \dfrac{F(X,2Y)}{2\beta}+\dfrac{r(r-1)H'(X,2Y)}{a_{2}^2\beta}\right] -\left[g_{2,1}^0,\enspace \dfrac{F(X,2Y)}{2\beta}+\dfrac{r(r-1)H'(X,2Y)}{a_{2}^2\beta}\right]\\
&+ &\left[g_{2,2}^0, \dfrac{F(X,2Y-X)}{2\beta}+\dfrac{r(r-1)H'(X,2Y-X)}{a_{2}^2\beta}\right] \\ 
&-&\left[g_{2,3}^0, \dfrac{F(X,2Y-X)}{2\beta}+\dfrac{r(r-1)H'(X,2Y-X)}{a_{2}^2\beta}\right]\mod\wp.
 \end{eqnarray*}
 Applying Lemma \ref{F} and \ref{hprime}, we get that,
 \begin{eqnarray}\label{T+f1}
 T^+f_{1} &\equiv& \left[g_{2,0}^0,\enspace \dfrac{-X^{r-1}Y}{\beta}\right] -\left[g_{2,1}^0,\enspace \dfrac{-X^{r-1}Y}{\beta}\right]
+ \left[g_{2,2}^0, \dfrac{(r-1)X^{r-1}Y-r(r-1)X^{r-2}Y^2}{\beta}\right] \nonumber \\
 &- &\left[g_{2,3}^0, \dfrac{(r-1)X^{r-1}Y-r(r-1)X^{r-2}Y^2}{\beta}\right]
\end{eqnarray}

On adding congruences \eqref{T-f3},\eqref{a2f2} and \eqref{T+f1}, we can immediately see that, 
\begin{eqnarray}\label{sumrad2}
    T^{-}f_3-a_2f_2+T^+f_{1} &\equiv& 
    \left[g_{2,0}^0, \dfrac{r-1}{\beta}X^{r-1}Y + \dfrac{r(2r-a_2^2)Y^r}{a_2^2\beta}\right]- \left[g_{2,1}^0, \dfrac{r-1}{\beta}X^{r-1}Y + \dfrac{r(2r-a_2^2)Y^r}{a_2^2\beta}\right] \nonumber\\   
&- &\left[g_{2,2}^0, \dfrac{r-1}{\beta}\left(rY^{r}-(r+1)X^{r-1}Y+rX^{r-2}Y^2\right)\right] \\
& +& \left[g_{2,3}^0, \dfrac{r-1}{\beta}\left(rY^{r}-(r+1)X^{r-1}Y+rX^{r-2}Y^2\right)\right].\nonumber
\end{eqnarray}
Further simplifying using Lemma \ref{rem1} and Lemma \ref{Lemmatomt} we get modulo $\wp,X_{r}+V_{r}^{(1)}$, $T^{-}f_3-a_2f_2+T^+f_{1}$ is 
\begin{eqnarray*}
     &\equiv& \begin{cases}
0& \text{ if }\tau'<t\\
   \left [g_{2,0}^{0}, \dfrac{X^{r-1}Y}{r} \right]-\left[g_{2,1}^{0}, \dfrac{X^{r-1}Y}{r}\right] -\left[g_{2,2}^{0},  X^{r-2}Y^{2}\right]  +\left[g_{2,3}^{0},  X^{r-2}Y^{2}\right]    &\text{ if }\ \tau'\geq t
    \end{cases}\\
    &\equiv & 
    \begin{cases}
0 & \text{ if }\tau'<t\\
   \left [g_{2,0}^{0}, X^{r-1}Y \right]+\left[g_{2,1}^{0}, X^{r-1}Y\right] +\left[g_{2,2}^{0},  X^{r-2}Y^{2}\right]  +\left[g_{2,3}^{0},  X^{r-2}Y^{2}\right]  &\text{ if }\ \tau'\geq t.
    \end{cases}
\end{eqnarray*}

\noindent\textbf{\underline{Radius $n\geq3$}}\\
Here we use Lemma \ref{F} repeatedly to compute the part of $(T-a_2)f$ supported on the $n$-th radius of the Bruhat-Tits tree, for a general $n\geq 3$. 

We first compute $T^+f_{2}$. Note from Lemma \ref{F} that the coefficients of monomials in $F(X,2Y-X)$ are of valuation greater than or equal to $t+1$ for $r$ both odd and even. Hence the corresponding branches of $T^{+}f_2$ vanishes$\mod \wp$ as $v(a_2\beta)< t+1$. Therefore, 
\begin{eqnarray}
    T^+f_{2} &\equiv& \left[g_{3,0}^{0}, \dfrac{rF(X,2Y)}{a_{2}\beta}\right] + \left[g_{3,2}^{0}, \dfrac{r(r-1)F(X,2Y)}{a_{2}\beta}\right]\nonumber \\
    && -\left[g_{3,1}^{0}, \dfrac{rF(X,2Y)}{a_{2}\beta}\right] - \left[g_{3,3}^{0}, \dfrac{r(r-1)F(X,2Y)}{a_{2}\beta}\right]\nonumber\\
    &\equiv& \left[g_{3,0}^{0}, \dfrac{-2rX^{r-1}Y}{a_{2}\beta}\right] - \left[g_{3,1}^{0}, \dfrac{-2rX^{r-1}Y}{a_{2}\beta}\right] \mod \wp.\label{T^+f_2}
\end{eqnarray}
With a simpler computation for $n\geq 4$, the above congruence modulo $\wp$ can be generalised  to all $n\geq 3$ as
\begin{eqnarray}\label{T^+f_{n-1}}
T^{+}(f_{n-1}) &\equiv& \frac{1}{\beta}\left(\dfrac{2r}{a_{2}}\right)^{n-2}\cdot\left(-\left[g_{n,0}^{0}, X^{r-1}Y\right] + \left[g_{n,1}^{0}, X^{r-1}Y\right]\right).
\end{eqnarray}
Similarly we compute (modulo $\wp$) for all $n\geq 3$,
\begin{eqnarray}\label{a_2f_n}
 -a_{2}(f_{n}) &\equiv& \frac{1}{\beta}\left(\frac{2r}{a_{2}}\right)^{n-2}\cdot\left(-\left[g_{n,0}^{0},\: rF(X,Y)\right] + \left[g_{n,1}^{0},\: rF(X,Y)\right]\right) .  
\end{eqnarray}
\begin{eqnarray}\label{T^-f_{n+1}}
T^{-}(f_{n+1}) &\equiv& \frac{1}{\beta}\left(\frac{2r}{a_{2}}\right)^{n-2}\cdot\left(\left[g_{n,0}^0, \left(\frac{2r^2}{a_{2}^2}\right)Y^r \right] - \left[g_{n,1}^0, \left(\frac{2r^2}{a_{2}^2}\right)Y^r \right]\right).
\end{eqnarray}

 Combining \eqref{T^+f_{n-1}}, \eqref{a_2f_n} and \eqref{T^-f_{n+1}} we get,
 \begin{eqnarray*}
 T^+f_{n-1}-a_{2}f_{n}+T^{-}f_{n+1}& \equiv & \left(\frac{2r}{a_{2}}\right)^{n-2}\cdot\left[g_{n,0}^{0},\: \dfrac{-r(a_{2}^2-2r)}{a_{2}^2\beta}Y^r + \frac{(r-1)}{\beta}X^{r-1}Y\right]\\
 & & +\left(\frac{2r}{a_{2}}\right)^{n-2}\cdot\left[g_{n,1}^{0}, \:\dfrac{r(a_{2}^2-2r)}{a_{2}^2\beta}Y^r - \dfrac{(r-1)X^{r-1}Y}{\beta}\right]\\
 & \equiv& 0 \mod \wp, X_{r}+V_{r}^{(1)}.
 \end{eqnarray*}
Since $n \geq 3$,  the factor $\left(\frac{2r}{a_{2}}\right)^{n-2}$ has positive valuation. Further by equation \eqref{beta} and Lemma \ref{rem1} respectively, both $\frac{r-1}{\beta}$ and $\frac{r(a_{2}^2-2r)}{a_{2}^2\beta}$   are integral. Hence the  last congruence follows.

\noindent\textbf{\underline{Radius $-1$}}\\
We use Lemma \ref{F} and Lemma \ref{H} to arrive at the following congruence relations.
\begin{eqnarray} \label{T-f0}
T^{-}f_{0}&=&\left[g_{0,0}^{1}, -\dfrac{H(2X,Y)}{2a_2\beta}-\dfrac{r(r-1)(2^rX^r+F(2X,Y))}{a_2\beta} + \dfrac{\delta(r)2^{r-1}X^{r-1}Y}{a_{2}\beta}\right]\nonumber\\
&\equiv& 
 \left[g_{0,0}^{1}, \dfrac{-rXY^{r-1}-r(r-1)X^2Y^{r-2}}{a_{2}\beta}-\dfrac{r(r-1)Y^r}{a_{2}\beta}\right]\mod\wp. 
\end{eqnarray}

Using the definition of $H''(X,Y)$, we compute
\begin{equation}\label{a2f-1}
    -a_2f_{-1}= 
        \left[g_{0,0}^1, \dfrac{-a_2G(X,Y)}{2\beta} - \dfrac{r(r-1)H(X,Y)}{a_{2}\beta}+ \dfrac{r^2(r-1)(XY^{r-1}+X^2Y^{r-2})}{a_2\beta}\right] 
    \end{equation}
Note that now  we are analyzing the action of the Hecke operator on the part of $f$ supported on the left side of the tree. 
For this we apply Lemma \ref{T+toT-}, to reduce the computations in the LHS of the tree to equivalent ones already studied in the RHS. For example, we have
\begin{eqnarray*}
    T^-(f_{-2})&=& T^-\left[g_{1,0}^1,\enspace \dfrac{rG(X,Y)}{a_{2}\beta}\right]+ T^-\left[g_{1,1}^1,\enspace \dfrac{r(r-1)G(X,Y)}{a_2\beta}\right]\\
    &=&  w\cdot T^-\left[g_{2,0}^0,\enspace \dfrac{rF(X,Y)}{a_{2}\beta}\right]+ \begin{pmatrix}
        0 & 1\\
        1 & -1
    \end{pmatrix}\cdot T^-\left[g_{2,3}^{0},\enspace \dfrac{r(r-1)F(X,Y)}{a_2\beta}\right]
   \end{eqnarray*}
by equations \eqref{T+toT-c} and \eqref{T+toT-d} of Lemma \ref{T+toT-}.  We remark that $T^-$ of the elementary functions above have been already dealt with while computing (part of) $T^-f_2$. Using those computations and using the matrix equality $w\cdot g_{1,0}^{0}= 
    \begin{psmallmatrix}
    0 & 1\\
    1 & -1
    \end{psmallmatrix}\cdot g_{1,1}^{0} = g_{0,0}^{1}\cdot w.$  
    we get
\begin{eqnarray}\label{T-f-2}
    T^{-}f_{-2}&
    \equiv & w\cdot \left[g_{1,0}^{0}, \dfrac{rY^{r}}{a_{2}\beta}\right] + \begin{pmatrix}
        0 & 1\\
        1 & -1
\end{pmatrix}\cdot\left[g_{1,1}^{0}, \dfrac{r(r-1)(X+Y)^r}{a_{2}\beta}\right] \nonumber\\
    &\equiv & \left[g_{0,0}^{1}, \dfrac{rX^{r}+r(r-1)(Y+X)^r}{a_{2}\beta}\right] \nonumber\\
 &\equiv &\left[g_{0,0}^{1}, \dfrac{r^{2}X^{r}}{a_{2}\beta}+ \dfrac{r(r-1)Y^r}{a_2\beta}+ \dfrac{r(r-1)H(X,Y)}{a_{2}\beta}\right] \mod\wp.
\end{eqnarray}
Now we note that the polynomial expressions in equations \eqref{T-f0}, \eqref{a2f-1} and \eqref{T-f-2} have already appeared in the calculations of radius 1 with $X$ and $Y$ swapped. Hence their sum can be obtained by swapping $X$ and $Y$ in the polynomial expression of \eqref{rad1}:
\begin{eqnarray*}
    T^-f_0-a_2f_{-1}+T^-f_{-2} & \equiv&  
        \left[g_{0,0}^{1}, \dfrac{-a_2^2+2r^2}{2a_2\beta}X^r + \dfrac{a_{2}^2-2r+2r^2(r-1)}{2a_2\beta}XY^{r-1}+\dfrac{r(r-1)^2}{a_2\beta}X^2Y^{r-2}\right]\\
       &\equiv& \left[g_{0,0}^1, -\dfrac{\alpha'}{\beta}X^r + \dfrac{\alpha'}{\beta}XY^{r-1}\right] \equiv \left[g_{0,0}^{1},\dfrac{\alpha'}{\beta}XY^{r-1}\right]\mod \wp,X^{r}+V_{r}^{(1)}.
\end{eqnarray*}

\noindent\textbf{\underline{Radius $-n,\: n\geq 2$}}

We have checked  that 
\begin{eqnarray}\label{conga2}
-a_{2}f_{-n} 
&\equiv & \begin{cases}
-\dfrac{r}{\beta}\cdot \left( \left[g_{1,0}^{1}, G(X,Y)   \right] + \left[g_{1,1}^{1}, (r-1)G(X,Y) \right] \right)&\text{ if } n=2\\
-\dfrac{r}{\beta}\left(\dfrac{2r}{a_2}\right)^{n-2}\cdot\left[g_{n-1,0}^1,\,G(X,Y)\right] &\text{ if } n\geq3.

\end{cases}
\end{eqnarray} 
Further for $n\geq 2$ and $r\geq 4$, we have
\begin{eqnarray*}\label{congT^-}
    T^-f_{-(n+1)}& =& T^-\left[g_{n,0}^{1}, \left(\dfrac{2r}{a_2}\right)^{n-1}\dfrac{rG(X,Y)}{a_{2}\beta}\right]
    \equiv w\cdot T^-\left[g_{n+1,0}^{0}, \left(\dfrac{2r}{a_2}\right)^{n-1}\dfrac{rF(X,Y)}{a_{2}\beta}\right]\nonumber\\
    &\equiv&  \left[w \cdot g_{n,0}^{0}, \left(\dfrac{2r}{a_2}\right)^{n-1}\dfrac{rY^{r}}{a_{2}\beta}\right]
    \equiv \left[g_{n-1,0}^{1} \cdot w, \left(\dfrac{2r}{a_2}\right)^{n-1}\dfrac{rY^{r}}{a_{2}\beta}\right]\nonumber\\
      &\equiv&  \left[g_{n-1,0}^{1}, \left(\dfrac{2r}{a_2}\right)^{n-1}\dfrac{rX^{r}}{a_{2}\beta}\right]\mod\wp.
\end{eqnarray*}
Similarly for $n\geq 3$, one computes using Lemma \ref{T+toT-},
\begin{eqnarray}\label{congT^+}
     T^+f_{-(n-1)} &\equiv &\left[g_{n-1,0}^1, -\left(\dfrac{2r}{a_2}\right)^{n-2}\cdot\dfrac{XY^{r-1}}{\beta}\right]\mod\wp,
\end{eqnarray}
with slightly different argument in the case $n=3$. Applying Lemma \ref{T+toT-} again and using part of the computation done for $T^+f_1$ (see \eqref{T+f1}), we have
\begin{eqnarray*}\label{T+f-1}
    T^{+}f_{-1} 
   &=& w\cdot T^{+}\left[g_{1,0}^0,  \dfrac{F(X,Y)}{2\beta} + \dfrac{r(r-1)H'(X,Y)}{a_{2}^2\beta}\right] \nonumber\\
&\equiv & -\left[g_{1,0}^1, \dfrac{XY^{r-1}}{\beta}\right] 
+ \left[g_{1,1}^1, \dfrac{(r-1)XY^{r-1}-r(r-1)X^2Y^{r-2}}{\beta}\right]\mod \wp.
\end{eqnarray*}
The last congruence follows since $w\cdot g_{2,0}^{0}=g_{1,0}^{1}\cdot w \text{ and } w\cdot g_{2,2}^{0} = g_{1,1}^{0}\cdot w$. Combining the above two congruences, we have 
\begin{eqnarray}
    T^{+}f_{-(n-1)} & \equiv & \begin{cases}
        -\left[g_{1,0}^1, \dfrac{XY^{r-1}}{\beta}\right] 
+ \left[g_{1,1}^1, \dfrac{(r-1)XY^{r-1}-r(r-1)X^2Y^{r-2}}{\beta}\right] , & n=2 \\
-\left(\dfrac{2r}{a_2}\right)^{n-2}\cdot\left[g_{n-1,0}^1, \dfrac{XY^{r-1}}{\beta}\right]\mod\wp, &n\geq 3
\end{cases}
\end{eqnarray}
Adding the congruences
\eqref{conga2}, \eqref{congT^-} and \eqref{T+f-1} for $n=2$ and emulating a part of the calculation at radius $2$ (see equation \eqref{sumrad2}), we get that
\begin{eqnarray*}
T^+f_{-1}+T^-f_{-3}-a_{2}f_{-2} &\equiv &\left[g_{1,0}^{1}, \frac{r(2r-a_{2}^2)X^{r}}{a_{2}^2\beta} +\frac{(r-1)XY^{r-1}}{\beta} \right] \\
&& -\left[g_{1,1}^{1}, \frac{r-1}{\beta}\left(rX^r -(r+1)XY^{r-1}+rX^{2}Y^{r-2}\right) \right]\\
    &\equiv& \begin{cases}
0 & \text{ if } \tau'<t\\
     \left[g_{1,0}^{1}, XY^{r-1}\right] + \left[g_{1,1}^{1},  X^2Y^{r-2} \right]\mod \wp,X_{r}+V_{r}^{(1)} & \text{ if } \tau'\geq t.
    \end{cases}
    \end{eqnarray*}

If $n\geq3$, the calculation is simpler and we have
\begin{eqnarray*}
T^+f_{-(n-1)}+T^-f_{-(n+1)}-a_{2}f_{-n} &\equiv& \left(\dfrac{2r}{a_2}\right)^{n-2}\cdot\left[g_{n-1,0}^1, \frac{r(2r-a_{2}^2)X^{r}}{a_{2}^2\beta}+\frac{(r-1)XY^{r-1}}{\beta}\right] \\
&\equiv& 0 \mod \wp,X_{r}+V_{r}^{(1)}.
\end{eqnarray*}



Thus we observe that the part of $(T-a_2)f $ supported on each radius is integral. Further going modulo $\wp,X_{r}$ and $V_{r}^{(1)}$, the parts supported on radius $n\geq 3$ or $\leq -3$ vanish. 
To be explicit, we have  
\begin{enumerate}
\item  
 If $\tau'<t$, then
 $$\overline{(T-a_{2})f} \equiv \left[ g_{1,0}^{0}, X^{r-1}Y \right]
+ \left[ g_{1,1}^{0}, -X^{r-1}Y\right] +\left[g_{0,0}^1, XY^{r-1}\right],$$
which further maps to $\left[ g_{1,0}^{0}, 1 \right]
+ \left[ g_{1,1}^{0}, 1\right] +\left[g_{0,0}^1, 1\right]
     =T[1,1] \in \mathcal{I}(V_{0})$ under the map in Lemma \ref{VrtoV0}.
     \hspace{20mm}
\item If  $\tau'\geq t$, then \begin{eqnarray*}
\overline{(T-a_{2})f} &\equiv& \overline{\frac{\alpha'}{\beta}} \cdot \left(\left[g_{1,0}^{0}, X^{r-1}Y\right] + \left[g_{1,1}^{0}, X^{r-1}Y\right] + \left[g_{0,0}^{1}, XY^{r-1}\right]\right) 
+\left[g_{2,0}^{0}, X^{r-1}Y  \right] \\
&& + \left[g_{2,2}^{0},X^{r-2}Y^2  \right] 
+ \left[g_{2,1}^{0},X^{r-1}Y  \right]
+\left[g_{2,3}^{0},X^{r-2}Y^2 \right] \\
&& + \left[g_{1,0}^{1}, XY^{r-1} \right] + \left[g_{1,1}^{1}, X^{2}Y^{r-2}  \right],  
\end{eqnarray*}
which by Lemma \ref{VrtoV0} further maps to 
\begin{eqnarray*}
\left[g_{2,0}^{0}, 1 \right] + \left[g_{2,2}^{0},1 \right] + \left[g_{2,1}^{0},1  \right]
+\left[g_{2,3}^{0},1\right] + \left[g_{1,0}^{1}, 1 \right] + \left[g_{1,1}^{1}, 1\right]&&\\ + c\left(\left[g_{1,0}^{0},1\right] + \left[g_{1,1}^{0}, 1\right]+\left[g_{0,0}^1, 1\right] \right) 
= \left(T^2 + cT + 1 \right)\left[1,1 \right] &\in& \mathcal{I}(V_{0}), 
\end{eqnarray*}
where $c = \overline{\frac{\alpha'}{\beta}}=\overline{\frac{\alpha'}{r(r-1)}} = \overline{\frac{\alpha'}{r-1}}$, since $r$ is odd by Lemma \ref{Lemmatomt}.
\end{enumerate} 
\hspace{20mm}
\\
By the $G$-linearity of $T$,  the surjection $P$ factors via $\dfrac{\mathcal{I}(V_{0})}{(T)}= \pi(0,0,1)$ when $\tau'<t$ and when  $\tau' \geq t$, it factors via  $\dfrac{\mathcal{I}(V_{0})}{\left(T^2+ cT + 1 \right)} \cong \dfrac{\mathcal{I}(V_{0})}{T-\lambda} \oplus \dfrac{\mathcal{I}(V_{0})}{T-\lambda^{-1}} \cong \pi(0,\lambda,1) \oplus \pi(0,\lambda^{-1},1)$ where $\lambda$ is such that $\lambda^{2}-c\lambda +1=0$.

For both $r=2$ and $r=3$, the action of $T$ on $F(X,Y)$ and $H'(X,Y)$ is  different from higher values of $r$ as can be seen from Lemma \ref{F} and Lemma \ref{H'}. Hence $\left(T-a_{2}\right)f$ has to be computed separately for these cases. We note that the function $f$ becomes much simpler for $r=2$ and $r=3$. If $r=2$, then $H(X,Y), H'(X,Y)$ and $H''(X,Y)$ simplify to monomials and further all components of $f$ vanish modulo $\wp$ except for $f_{1}$ and $f_{-1}$. As a consequence, the calculation simplifies and one checks that $\overline{(T-a_{2})f}$ maps to $T[1,1]$ in $\mathcal{I}(V_{0})$. For $r=3$, analysing the function $f$ we note that $H(X,Y)= 3(X^2Y+XY^2)$ and both $H'(X,Y)$ and $H''(X,Y)$ are zero. Applying Lemma \ref{F} and Lemma \ref{H'}, we see that $\overline{(T-a_{2})f}$ maps to $T[1,1]$ if $\tau'<t$ and to $\left(T^2+cT+1\right)[1,1]$ if $\tau'\geq t$. Thus any change brought in by the differed action of $T$ gets nullified$\mod \wp$ and the result for $r\geq4$ can be extended to include $r=2$ and $3$.
\end{proof}
Now applying the mod $p$ local Langlands correspondence for $p=2$, we have the following corollary with $c \in \bar{\F}_{2}$ as in Theorem \ref{irred}.
\begin{cor}
   If $k \geq 4$ and $0<v(a_{2})<1$, then $$\bar{V}_{k,a_2} \cong\begin{cases}
\ind(\omega_{2}) &\text{ if } \tau'<t\\
\mu_{\lambda} \oplus \mu_{\lambda^{-1}} &\text{ if } \tau'\geq t.
\end{cases}$$ where $t$ and $\tau'$ are as in the above theorem and $\lambda \in \bar{\F}_{2}$ is such that  $\lambda^{2}-c\lambda +1=0$. 
\end{cor}

\begin{remark}
Using Lemma \ref{Lemmatomt} and the corollary above, we see that for slopes in $(0,1)$, the reduction $\bar V_{k,a_2}$ can be reducible only if $k$ is odd and $v(a_2)=1/2$. This proves corollary \ref{cor1} stated in the introduction.
\end{remark}

\section{slope $\nu=1$}\label{slope=1}
 In this section we shall treat the case of smallest positive integral slope $\nu= v(a_{2})=1$ for $p=2$. As explained in \S \ref{fil}, there is a natural map $P: \mathcal{I}(V_r) \to \bar\Theta:=\bar\Theta_{k,a_2}.$ By Lemma \ref{RemarkBG09}, the map $P$ further factors through $\mathcal{I}(V_r/V_r^{(2)})$. The following short exact sequences of $\bar{\F}_{2}[\Gamma]$-modules are consequences of Proposition \ref{VrtoV0} and Remark \ref{remVrVr1}.
\begin{eqnarray}
\label{ses1}0\rightarrow \frac{V_r^{(1)}}{V_r^{(2)}}\rightarrow\frac{V_r}{V_r^{(2)}}\rightarrow\frac{V_r}{V_r^{(1)}}\rightarrow 0, \\
\label{ses2}0\rightarrow V_1\rightarrow\frac{V_r}{V_r^{(1)}}\rightarrow V_0\rightarrow 0, \text{ for } r\geq2, \\
\label{ses3}0\rightarrow V_0\rightarrow \frac{V_r^{(1)}}{V_r^{(2)}}\rightarrow V_1\rightarrow 0, \text{ for } r\geq 5.
    \end{eqnarray}
    Though the latter two sequences are split, i.e.,
$\frac{V_r}{V_r^{(1)}}\cong\frac{V_r^{(1)}}{V_r^{(2)}}\cong V_0\oplus V_1$, we have written them in different order as we will use the explicit maps in that order.
Using \eqref{ses1}, \eqref{ses2} and \eqref{ses3} above, and exactness of the functor $\mathcal{I}$, we break $\bar\Theta$ into four components, $\sF_0$, $\sF_1$, $\sF_2$ and $\sF_3$ such that the following diagrams of $\bar\F_2[G]$-modules are commutative.
\begin{equation}\label{diag1}
\begin{tikzcd}
	0 \arrow[r] & \mathcal{I}\left(\dfrac{V_{r}^{(1)}}{V_{r}^{(2)}}\right) \arrow[d, " "]\arrow[r]& \mathcal{I}\left(\dfrac{V_{r}}{V_{r}^{(2)}}\right)  \arrow[r] \arrow[d, "P"] & \mathcal{I}\left(\dfrac{V_{r}}{V_{r}^{(1)}}\right) \arrow [r] \arrow[d, " "]& 0\\
	0 \arrow[r] & \bar{\Theta}_{0}\arrow[r] & \bar{\Theta} \arrow[r] & \bar{\Theta}_{1}\arrow[r] & 0
\end{tikzcd}
\end{equation}
\begin{equation}\label{diag2}
\begin{tikzcd}
	0 \arrow[r] & \mathcal{I}(V_1)\arrow[d, " "]\arrow[r]& \mathcal{I}\left(\dfrac{V_{r}}{V_{r}^{(1)}}\right)  \arrow[r] \arrow[d, "P"] & \mathcal{I}(V_0)\arrow [r] \arrow[d, " "]& 0\\
	0 \arrow[r] & \sF_0 \arrow[r] & \bar{\Theta}_1\arrow[r] & \sF_1\arrow[r] & 0
\end{tikzcd}
\end{equation}
\begin{equation}\label{diag3}
\begin{tikzcd}
	0 \arrow[r] & \mathcal{I} (V_0) \arrow[d, " "]\arrow[r] & \mathcal{I}\left(\dfrac{V_{r}^{(1)}}{V_{r}^{(2)}}\right)  \arrow[r] \arrow[d, "P"] & \mathcal{I} (V_1) \arrow [r] \arrow[d, " "]& 0\\
	0 \arrow[r] & \sF_2 \arrow[r] & \bar{\Theta}_0\arrow[r] & \sF_3\arrow[r] & 0
\end{tikzcd}
\end{equation}

We note that if slope is $1$, then $\mathcal{I}\left(X_{r}\right)\subset \ker(P)$ using Lemma \ref{RemarkBG09}. We will use polynomial $H(X,Y)$ introduced in \eqref{Hdef} in this context as well. Considering $H(X,Y)$ as a polynomial over $\bar{\F}_{2}$, we have the following lemma.
\begin{lemma}\label{X_r^*}
    The polynomial $H(X,Y)\in X_r\cap V_r^{(1)}$. If $r$ is odd, it maps to a non-zero element in $V_r^{(1)}/V_r^{(2)}$ that generates a trivial subspace $V_0$. If $r$ is even, then $H(X,Y)\in V_r^{(2)}$. 
\end{lemma}
\begin{proof}
We note that $H(X,Y) = \begin{pmatrix}
	1 & 0\\
	1 & 1
\end{pmatrix} \cdot X^{r}- X^{r}-\begin{pmatrix}
	0 & 1\\
	1 & 0
\end{pmatrix} \cdot X^{r}$. Hence it belongs to $X_{r}$. Using Lemma \ref{Vr*}, we conclude that $H(X,Y)\in  V_r^{(1)}$ and further in $V_{r}^{(2)}$ if $r$ is even.  If $r$ is odd, then $H(X,Y)\notin V_{r}^{(2)}$ and thus maps to a non-zero element in $V_{r}^{(1)}/V_{r}^{(2)}$. We can check that all the six matrices of $\GL_{2}(\mathbb{F}_{2})$ act trivially on the image of $H(X,Y)$ in the above quotient. Therefore it generates trivial subspace $V_{0}$ of $V_r^{(1)}/V_r^{(2)}$.  \end{proof}

We shall now state a combinatorial lemma which shall be used extensively in the subsequent lemmas and propositions.

\begin{lemma}\label{sl1cmb1}
Let $r\geq 3$, and $t:= v(r-2)$. Then for all $j\geq 3$, we have $v\left({r \choose j}\right)+j \geq t+2$. 
\end{lemma}
\begin{proof}
By applying Legendre's formula, if $ j\leq r$, we have
\begin{equation}\label{sl1cmb1eq1}
        v\left({r \choose j}\right)+ j = \left(\sum_{i=0}^{j-1}v(r-i)\right)+j-v(j!)
        = \sum_{i=0}^{j-1}v(r-i)+S_{2}(j)
\end{equation}
where $S_{2}(j)$ is the sum of the coefficients in the $2$-adic expansion of $j$. 
Clearly $S_{2}(j)\geq1$. Thus  for all $3\leq j\leq r$, we have
\begin{eqnarray*}
v\left({r \choose j}\right)+j\geq v(r)+v(r-1)+v(r-2)+1\geq t+2.
\end{eqnarray*}
If $j>r$ we note that ${r\choose j}=0$ and hence the inequality trivially holds,  proving the lemma.
\end{proof}


\begin{lemma}\label{Heven}
Let $r\geq 3$ and $t=v(r-2)$. Then modulo $2^{t+2}$, we have, 
\begin{enumerate}
	\item $H(X,2Y) \equiv 
    \begin{cases}
    2rX^{r-1}Y + 2rX^{r-2}Y^2 &\text{ if }\, r \text{ is even}\\
    2rX^{r-1}Y &\text{ if }\, r \text{ is odd}    
    \end{cases}
    $ \label{Heven1}
	\item $H(X,2Y-X) \equiv \begin{cases}
-2X^{r}+2rX^{r-1}Y-2rX^{r-2}Y^2 & \text{ if }\, r \text{ is even,}\\
 -2rX^{r-1}Y & \text{ if }\, r \text{ is odd.}
    \end{cases}$    
    \label{Heven2}
	\item $H(2X,Y) \equiv \begin{cases}
2rXY^{r-1} + 2rX^2Y^{r-2} \text{ if }\, r \text{ is even,}\\
2rXY^{r-1} \text{ if }\, r \text{ is odd.}
	\end{cases}$ \label{Heven3}
\end{enumerate}
\end{lemma}

\begin{proof}
Applying Lemma \ref{sl1cmb1} we have 
\begin{eqnarray*}
    H(X,2Y)&=& \sum_{j=1}^{r-1}{r \choose j}X^{r-j}(2Y)^{j} = 2rX^{r-1}Y+ 2^2{r\choose 2}X^{r-2}Y^2+\sum_{j=3}^{r-1}{r \choose j}X^{r-j}(2Y)^{j}\\
    &\equiv & 2rX^{r-1}Y+ 2r(r-1)X^{r-2}Y^2 \mod2^{t+2}\\
    &\equiv& \begin{cases}
        2r(X^{r-1}Y+X^{r-2}Y^2)\mod 2^{t+2} & \text{if } r \text{ is even,}\\
        2rX^{r-1}Y \mod2^{2}& \text{if } r \text{ is odd.}\end{cases}
\end{eqnarray*} 
In the last step, we used that $2r(r-1)\equiv 2r(r-2)+2r\equiv 2r\mod2^{t+2}$ if $r$ is even and further $2r(r-1)\equiv 0\mod 2^{2}$ if $r$ is odd, i.e., $t=0$.

For part \eqref{Heven2}, we have
\begin{eqnarray*}
    H(X,2Y-X) &=& \sum_{j=1}^{r-1}\sum_{i=0}^{j}(-1)^{j-i}2^ i{r \choose j}{j \choose  i}X^{r- i}Y^{ i}\\
              &=& \sum_{j=1}^{r-1}\sum_{ i=0}^{j}(-1)^{j- i}2^ i{r \choose  i}{r- i \choose j- i}X^{r- i}Y^{ i}.
\end{eqnarray*}
On applying Lemma \ref{sl1cmb1}, we can see that terms corresponding to $ i\geq 3$ vanishes modulo $2^{t+2}$. Thus we have
\begin{equation}
H(X,2Y-X) \equiv \sum_{ i=0}^{2}\sum_{j=1}^{r-1}(-1)^{j- i}2^ i{r \choose  i}{r- i \choose j- i}X^{r-i}Y^{i} \mod 2^{t+2}.
\end{equation}
By similar calculations as in the proof of Lemma \ref{H}, we check that the coefficients of $X^{r-i}Y^{i}$ for $i\leq 2$ in the above expression is as follows:
\begin{enumerate}
   \item When $i=0,\,\displaystyle\sum_{j=1}^{r-1}(-1)^{j}{r \choose j} = \begin{cases}
   -2 &\text{ if }\, r \text{ is even, }\\
   0 &\text{ if }\, r \text{ is odd.}
   \end{cases}
   $
   \item When $i=1,\, 2r\displaystyle\sum_{j'=0}^{r-2}(-1)^{j'}{r-1 \choose j'}= \begin{cases}
   2r & \text{ if }\, r \text{ is even,}\\
   -2r & \text{ if }\, r \text{ is odd.}
   \end{cases}$
   \item When $i=2,\, 2^2{r \choose 2}\displaystyle\sum_{j'=0}^{r-3}(-1)^{j'}{r-2 \choose j'}= \begin{cases}
   -2r(r-1) & \text{ if }\, r \text{ is even,}\\
    2r(r-1) &\text{ if }\, r \text{ is odd.}
   \end{cases}$
\end{enumerate}
Thus modulo $2^{t+2}$ we have 
\begin{equation*}
    H(X,2Y-X) \equiv \begin{cases}
-2X^{r}+2rX^{r-1}Y-2rX^{r-2}Y^2 & \text{ if }\, r \text{ is even,}\\
 -2rX^{r-1}Y & \text{ if }\, r \text{ is odd.}
    \end{cases}   
\end{equation*}

For part \eqref{Heven3} we observe that 
\begin{eqnarray*}
    H(2X,Y)&=& \sum_{j=1}^{r-3}{r \choose j}2^{r-j}X^{r-j}Y^{j}+2^2{r\choose 2}X^2Y^{r-2}+2rXY^{r-1} \\
    &\equiv& 2rXY^{r-1}+2r(r-1)X^2Y^{r-2} \mod 2^{t+2}
    \end{eqnarray*}
    by lemma \ref{sl1cmb1}, since ${r\choose j}={r\choose r-j}$ and $r-j\geq 3$ for $j\leq r-3$. Thus
    \begin{eqnarray*}
    H(2X,Y) &\equiv & \begin{cases}
     2rXY^{r-1}+ 2rX^{2}Y^{r-2} &\text{ if }\, r\text{ is even,}\\
     2rXY^{r-1} & \text{ if }\, r \text{ is odd,}
    \end{cases}
\end{eqnarray*}
following similar arguments as in the calculation of $H(X,2Y)$ above.
\end{proof}

\begin{remark}
Though the above lemma is stated for both $r$ even and odd cases, we shall use it only when $r$ is even. 
\end{remark}

\begin{lemma}\label{projection}
    In the short exact sequence \eqref{ses3} the projection map  $\frac{V_r^{(1)}}{V_r^{(2)}}\rightarrow V_1$ is given by
    $\overline{\theta X^{r-3-j}Y^j}\in\frac{V_r^{(1)}}{V_r^{(2)}}\longmapsto
    \begin{cases}
    X, & j=0\\
    Y, & j=r-3\\
    X+Y, & 0<j<r-3.    \end{cases}
    $\\
   As a consequence, the images of both the polynomials $H(X,Y) $ and $X^{r-1}Y+XY^{r-1}$  maps to $0\in V_1$ for odd $r$.\end{lemma}
\begin{proof}
We check that $\overline{\theta(X^{r-3}+Y^{r-3}+X^{r-4}Y)}$ is $\GL_{2}(\F_{2})$-invariant, 
hence we can identify its $\bar\F_2$- span as the subspace $V_{0}$ of  $V_{r}^{(1)}/V_{r}^{(2)}$. Since $V_{1} \cong \left(V_{r}^{(1)}/{V_{r}^{(2)}}\right)/V_{0}$, the image of  $\overline{\theta X^{r-4}Y}$ is equal to that of $\overline{\theta(X^{r-3}+Y^{r-3})}$ in $V_{1}$. The 
 span of $\theta{X^{r-3}}$ is fixed by the upper triangular matrices. So up to scalars, it must  map to $X$, which is a generator of the fixed subspace of this Borel subgroup in $V_{1}$. Now using $\GL_{2}(\F_{2})$- linearity, we can see that $\overline{Y^{r-3}}\theta$ maps to $Y$ and therefore $\overline{X^{r-4}Y}\theta$ maps to $X+Y$. Now the explicit map follows since we know $\theta X^{r-4}Y\equiv \theta X^{r-3-j}Y^j\mod V_r^{(2)}$ for all $0\leq j<r-3$.

For $r$ odd, the image of $H(X,Y)$ in $\frac{V_{r}^{(1)}}{V_{r}^{(2)}}$ maps to $0\in V_1$ by Lemma \ref{X_r^*} and we have, $$X^{r-1}Y+Y^{r-1}X=\sum_{j=0}^{r-3}\theta X^{r-3-j}Y^j\equiv \theta\left(X^{r-3}+X^{r-4}Y+Y^{r-3}\right)\mod V_r^{(2)}.$$ The last congruence follows since $r$ is odd. Now using the explicit formula proved above, we get that $X^{r-1}Y+Y^{r-1}X$ maps to $0$ in $V_{1}$.
 \end{proof}
Next, we are going to discuss the general case $r\geq 5$ first, and the small weights $r=2,3,4$ will be  treated separately later.

\begin{prop}\label{theta1}
    Let $r\geq 5 $ and $v(a_{2})=1$. Then $\bar{\Theta}_{1}=0$ and hence $\bar{\Theta} \cong \bar{\Theta}_{0}$.
\end{prop}
\begin{proof}
    Consider the function $f=\left[1, \dfrac{X^{r-1}Y-X^{r-3}Y^{3}}{a_{2}}\right]$. 
 \begin{eqnarray*}
  (T-a_{2})(f) &\equiv& \left[1, -(X^{r-1}Y - X^{r-3}Y^{3})\right] + \left[g_{1,0}^{0}, \dfrac{2}{a_{2}}X^{r-1}Y\right] \mod \wp\\
  &\equiv& \left[g_{1,0}^{0}, \dfrac{2}{a_{2}}X^{r-1}Y\right] \mod\wp, V_{r}^{(1)}.
  \end{eqnarray*}The quotient $\frac{V_{r}}{V_{r}^{(1)}}$ is spanned by the image of three monomials $X^r$, $X^{r-1}Y$ and $Y^r$, all of which lie in the $\bar{\F}_{2}[K]$-module generated by $\overline{X^{r-1}Y}$. Therefore $\left[g_{1,0}^{0}, \frac{2}{a_{2}}X^{r-1}Y\right]$ generates $\mathcal{I}\left(\frac{V_{r}}{V_{r}^{(1)}}\right)$ over $\bar{\F}_{2}[G]$. The congruence displayed above ensures that this generator lies in $\ker(P)$ and thus $\bar{\Theta}_{1}=0$ in the diagram \eqref{diag2}. \end{proof}
\begin{remark}
    We note that the proof above breaks down for $r=3$ and $4$. In the proof of Theorem \ref{k=4,5,6}, we will see that $\mathcal{F}_{1}$ is also a possible contributor to $\bar{\Theta}$ when $r=3$ and $4$.
\end{remark}
Now we will use the observations above to solve the problem when $r\geq 5$ is odd.
\begin{theorem}\label{oddthm}
If $k\geq 7$ is odd and $v(a_2)=1$, then $\bar V_{k,a_2}$ is irreducible.
    \end{theorem}
\begin{proof}
  By proposition \ref{theta1}, we know that $\bar{\Theta}=\bar{\Theta}_{0}$ which has two factors $\sF_2$ and $\sF_3$ as in diagram \eqref{diag3}. When $k$ is odd, we claim that  $\sF_2=0$, and $\sF_3$ factors through $\pi(1,0,1)$. This implies that $\bar{\Theta}\cong\sF_3$ is supercuspidal, hence the associated local Galois representation is irreducible.  


 By Lemma \ref{X_r^*}, we know the submodule $V_0$ in $\frac{V_r^{(1)}}{V_r^{(2)}}$ is spanned by the image of $H(X,Y)\in X_r$, but we know $\mathcal{I} (X_r)\subset \ker P$ by Lemma \ref{RemarkBG09}. Therefore $\sF_2=P\left(\mathcal{I}(V_0)\right)=0$, and consequently $\bar\Theta\cong\bar\Theta_0\cong\sF_3$.

  Next we consider the function $f=f_{1}+f_{-1}\in\mathcal{I}(\Sym^r\bar\Q_2^2)$ with
  $$f_{1}=\left[g_{1,0}^0, \frac{H'(X,Y)}{a_2}\right],\qquad\,
  f_{-1}=\left[g_{0,0}^1, \frac{H''(X,Y)}{a_2}\right],$$
   where $H'$ and $H''$ are as in \eqref{H'} and \eqref{H''} respectively. By Lemma \ref{hprime} and Lemma \ref{T+toT-},  $$Tf=Tf_{1} +Tf_{-1}=Tf_1+w\cdot Tf_1
   \equiv\left[1, \frac{2rXY^{r-1}}{a_2}\right]+ \left[1, \frac{2rX^{r-1}Y}{a_2}\right] \mod\wp.$$ 
   Thus we have that
  $$(T-a_2)f\equiv -\left[g_{1,0}^0, H'(X,Y)\right]-\left[g_{0,0}^1, H''(X,Y)\right]+\left[1, \frac{2r}{a_2}(X^{r-1}Y+XY^{r-1})\right]\mod \wp.$$ Under the map in Lemma \ref{projection}, the polynomials $H'(X,Y), H''(X,Y)$  and $X^{r-1}Y+ XY^{r-1}$  map to $X,Y$ and $0$ respectively in $V_{1}$. Hence $(T-a_{2})f$ maps to $\left[g_{1,0}^0, X\right]+\left[g_{0,0}^1, Y\right]\equiv T([1, X+Y]) \in \mathcal{I}(V_1)$, that generates all of $T\left(\mathcal{I}(V_1)\right) $ as a $\bar{\F}_{2}[G]$-module. Hence $\bar\Theta\cong\sF_3$ must be a quotient of $\mathrm{coker}(T)=\pi(1,0,1)$, as claimed.
  \end{proof}
Now the case of odd $r\geq 5$ is done. The answer is yet to be computed in the case when $r$ is even. Recall that using diagrams \eqref{diag1} and \eqref{diag2} and Proposition \ref{theta1}, we know that $\bar{\Theta}\cong \bar{\Theta}_{0}$ is a quotient of $\mathcal{I}\left(\frac{V_{r}^{(1)}}{V_{r}^{(2)}}\right)$. We note that $\frac{V_{r}^{(1)}}{V_{r}^{(2)}}= \frac{\theta V_{r-3}}{\theta V_{r-3}^{(1)}}$. 
 Further applying Lemma \ref{VrtoV0} for $r-3$, we have the following Lemma. 
\begin{lemma}\label{Vr1/Vr2}
 For $r\geq5,$  we have the following exact sequence,
 \begin{equation}\label{Vr1/Vr2b}
 0 \to V_{1} \to \frac{V_r^{(1)}}{V_{r}^{(2)}} \to V_{0} \to 0
 \end{equation} 
 where the injection map is given by 
    $ X \longmapsto \overline{\theta X^{r-3}},\,
    Y \longmapsto \overline{\theta Y^{r-3}} $
  and the projection map  is given by $\,
     \overline{ \theta X^{r-3-j}Y^j}\longmapsto 1 \text{ if } 0<j<r-3$ and $0$ otherwise.
\end{lemma}
 Recall that to analyse the case of odd $r$ in Theorem \ref{oddthm},  we used a different short exact sequence for $V_r^{(1)}/V_r^{(2)}$ that induces commutative diagram \eqref{diag3}. However, the exact same  approach do not work for even $r$. Instead we  utilise the following commutative diagram induced by \eqref{Vr1/Vr2b} above to find the structure of  $\bar{\Theta}_{0}$. 
\begin{equation}\label{diag4}
\begin{tikzcd}
 		0 \arrow[r] & \mathcal{I}\left(V_{1}\right) \arrow[d, ""]\arrow[r] & \mathcal{I}\left(\dfrac{V_{r}^{(1)}}{V_{r}^{(2)}}\right)  \arrow[r] \arrow[d, "P"] & \mathcal{I}(V_{0}) \arrow [r] \arrow[d, ""]& 0\\
 		0 \arrow[r] &\mathcal{F'}_{2}  \arrow[r] & \bar{\Theta}_{0} \arrow[r] & \mathcal{F'}_{3}\arrow[r] & 0
 \end{tikzcd}
\end{equation}

With the notation defined by the diagram above, (which is valid irrespective of parity of $r$) we have

\begin{prop}\label{F3'}
    Let $r\geq 6$ be even and $v(a_{2})=1$, then $\mathcal{F'}_{3}$ is a quotient of $\pi(0,\frac{a_{2}}{2},1)$. 
\end{prop}
\begin{proof}
For $r\geq 6$ even, we consider the function $f$ described as follows:
  \begin{eqnarray}
  	f_{3}&=& \left[g_{3,2}^{0}, \dfrac{(Y^{r}-X^{r-2}Y^{2})}{a_{2}}\right] + \left[g_{3,2^{2}+2}^{0},\dfrac{(Y^{r}-X^{r-2}Y^{2})}{a_{2}}\right] + \left[g_{3,3}^{0}, \dfrac{(Y^{r}-X^{r-2}Y^{2})}{a_{2}}\right]\nonumber\\
  	&& + \left[g_{3,2^{2}+3}^{0},\dfrac{(Y^{r}-X^{r-2}Y^{2})}{a_{2}} \right],\nonumber\\
	f_{2} &=&  \left[g_{2,2}^{0}, \dfrac{1}{2}(X^{r-1}Y-X^{r-2}Y^{2})+ \dfrac{2}{a_{2}^{2}}X^{r-2}Y^{2} + \dfrac{H(X,Y)}{a_{2}^{2}}\right]\nonumber\\
	&  & + \left[g_{2,3}^{0}, \dfrac{1}{2}(X^{r-1}Y-X^{r-2}Y^{2})+ \dfrac{2}{a_{2}^{2}}X^{r-2}Y^{2} + \dfrac{H(X,Y)}{a_{2}^{2}}\right],\nonumber\\
	f_{1} &= & \left[g_{1,0}^{0}, \frac{X^{r-3}Y^{3}}{a_{2}} \right]
	 + 
	\left[g_{1,1}^{0}, \frac{X^{r-3}Y^{3}}{a_{2}} \right],\nonumber\\
	f_{0} &=& \left[1, \dfrac{XY^{r-1}+X^{r-1}Y + X^{r-2}Y^{2} + X^{r-3}Y^{3}}{2}\right],\nonumber\\
	f_{-1} & =& \left[g_{0,0}^{1}, \frac{ X^{3}Y^{r-3}}{a_{2}} \right],\nonumber\\
	f_{-2} & = & \left[g_{1,1}^{1},\dfrac{1}{2}(XY^{r-1}-X^{2}Y^{r-2})+ \dfrac{2}{a_{2}^{2}}X^{2}Y^{r-2} + \dfrac{H(X,Y)}{a_{2}^{2}}\right], \nonumber\\
	f_{-3} &=& \left[g_{2,1}^{1}, \dfrac{(X^{r}-X^{2}Y^{r-2})}{a_{2}}\right] + \left[g_{2,3}^{1},\dfrac{(X^{r}-X^{2}Y^{r-2})}{a_{2}}\right]. \label{F3'eqeven}
\end{eqnarray}
We now compute $(T-a_{2})f \mod \wp, X_{r}, V_{r}^{(2)}+\theta X_{r-3} $. Since $f$ is supported on radii $n$ with $-3\leq n\leq 3$, it is enough to do this computation for radii between  $-4$ and $4$.\\
\noindent \underline{\textbf{Radius 0}}\\
Using lemmas \ref{T} and \ref{T+toT-} and that $r>4$, one checks that both $T^-f_1$ and $T^-f_{-1}$ vanish modulo $ \wp$. Therefore we have \begin{eqnarray}
 -a_{2}f_{0}+ T^{-}f_{1}+ T^{-}f_{-1} &\equiv& -a_2f_0\mod\wp\nonumber\\
 &\equiv & \left[1,\dfrac{-a_{2}(X^{r-1}Y+XY^{r-1}+\theta X^{r-4}Y)}{2} \right] \nonumber\\
 &\equiv& \left[1, \dfrac{-a_{2}(\theta X^{r-4}Y)}{2} \right]\mod (V_{r}^{(2)}+\theta X_{r-3}), \label{sl1F3'rad0}
\end{eqnarray}
 where in the last step we use the fact that $X^{r-1}Y +  XY^{r-1} \equiv \theta (X^{r-3} + Y^{r-3}+ (r-4)X^{r-4}Y)\mod V_r^{(2)}$, which is further congruent to $ \theta (X^{r-3}+Y^{r-3}) \mod \wp$ as $r$ is even.

\noindent\underline{\textbf{Radius 1}}\\
We compute using $r$ is even, that
\begin{equation}\label{F3'T+f0}
T^{+}(f_{0}) \equiv [g_{1,0}^{0}, X^{r-1}Y] + [g_{1,1}^{0}, X^{r-1}Y-X^{r}] \mod\wp, 
\end{equation}
\begin{equation}\label{F3'a2f1}
    -a_{2}f_{1} = [g_{1,0}^{0}, -X^{r-3}Y^3] + [g_{1,1}^{0}, -X^{r-3}Y^3].
\end{equation}
Further looking at the function $f_2$ given in \eqref{F3'} and using $r>3$, we conclude that
\begin{eqnarray*}
   T^{-}f_{2}&\equiv & \left[g_{1,0}^{0}, \dfrac{H(2X,X+Y)}{a_{2}^2} \right] + \left[g_{1,1}^{0}, \dfrac{H(2X,X+Y)}{a_{2}^2}\right]\\
   &\equiv& \left[g_{1,0}^{0}, \dfrac{2r(X(X+Y)^{r-1}+ X^2(X+Y)^{r-2})}{a_{2}^2}\right]+ \left[g_{1,1}^{0}, \dfrac{2r(X(X+Y)^{r-1}+ X^2(X+Y)^{r-2})}{a_{2}^2}\right]\\
   &\equiv & \left[g_{1,0}^{0}, \dfrac{2rX(X+Y)^{r-2}(2X+Y)}{a_{2}^2} \right]+ \left[g_{1,1}^{0}, \dfrac{2rX(X+Y)^{r-2}(2X+Y)}{a_{2}^2} \right] \\
& \equiv & \left[g_{1,0}^{0}, \dfrac{2rXY(X+Y)^{r-2}}{a_{2}^2} \right]+ \left[g_{1,1}^{0}, \dfrac{2rXY(X+Y)^{r-2}}{a_{2}^2} \right] \mod \wp.
\end{eqnarray*}
Here in the second step we use part (3) of Lemma \ref{Heven} and that $t\geq 1$.
Next we check that the polynomial  $XY(X+Y)^{r-2}$ lies in the submodule $V_r^{(2)}+\theta X_{r-3}$ of $V_r$, since
\begin{eqnarray*}
XY(X+Y)^{r-2} & = &X^{r-1}Y + XY^{r-1} + \displaystyle\sum_{k=1}^{r-3}{r-2 \choose k}X^{r-1-k}Y^{k+1}\\
&\equiv &X^{r-1}Y + XY^{r-1}\mod V_{r}^{(2)}\\
&\equiv &  \theta (X^{r-3} + (r-4)X^{r-4}Y + Y^{r-3})\\
&\equiv& 0 \mod \theta X_{r-3},
\end{eqnarray*}
as $r$ is even.
Indeed, the first congruence above follows from lemma \ref{Vr*} since
\begin{eqnarray*}
\displaystyle\sum_{k=1}^{r-3}{r-2 \choose k} &= &2^{r-2}-2 \equiv 0\mod 2, \text{\,  and further}\\
\displaystyle\sum_{k=1}^{r-3}(k+1){r-2 \choose k} &\equiv&\displaystyle\sum_{k=1}^{r-3}k{r-2 \choose k}=(r-2)\displaystyle\sum_{k=1}^{r-3}{r-3 \choose k-1}  \equiv 0\mod 2.
\end{eqnarray*}

Thus we have obtained
\begin{equation}\label{F3'T-f2}
   T^{-}f_{2}\equiv 0 \mod \wp,X_{r}, V_{r}^{(2)} + X_{r-3}\theta.
\end{equation}

On summing up \eqref{F3'a2f1}, \eqref{F3'T+f0} and \eqref{F3'T-f2} we get that modulo $( \wp,X_{r}, V_{r}^{(2)}+ \theta
X_{r-3}$)
\begin{eqnarray}
T^{-}f_{2} -a_{2}f_{1} + T^{+}f_{0} &\equiv &\left[g_{1,0}^{0}, X^{r-1}Y-X^{r-3}Y^3 \right] + \left[g_{1,1}^{0},X^{r-1}Y-X^{r-3}Y^3-X^{r}\right]\nonumber\\
&\equiv &\left[g_{1,0}^{0}, \theta(X^{r-3}+X^{r-4}Y) \right]+ \left[g_{1,1}^{0},\theta(X^{r-3}+X^{r-4}Y)\right]\mod X_r  \nonumber\\
&\equiv& \left[g_{1,0}^{0}, \theta X^{r-4}Y\right] + \left[g_{1,1}^{0},\theta X^{r-4}Y\right]\mod 
\theta X_{r-3}. \label{sl1F3'rad1}
\end{eqnarray}
\noindent\underline{\textbf{Radius 2}}\\
We now check the part of $(T-a_{2})f$ that is supported on radius $2$. Computing modulo $\wp$
\begin{eqnarray}\label{F3'T^+f1}
T^{+}f_{1} &\equiv& \left[g_{2,2}^{0}, \dfrac{2X^{r-1}Y}{a_{2}}-\dfrac{X^{r}}{a_{2}}\right] + \left[g_{2,3}^{0}, \dfrac{2X^{r-1}Y}{a_{2}}-\dfrac{X^{r}}{a_{2}}\right] ,\\
-a_{2}f_{2} &=& \left[g_{2,2}^{0}, \frac{-a_{2}(X^{r-1}Y-X^{r-2}Y^2)}{2} -\dfrac{2X^{r-2}Y^2}{a_{2}}-\dfrac{H(X,Y)}{a_{2}} \right] \nonumber\\
&& + \left[g_{2,3}^{0},\frac{-a_{2}(X^{r-1}Y-X^{r-2}Y^2)}{2} -\dfrac{2X^{r-2}Y^2}{a_{2}}-\dfrac{H(X,Y)}{a_{2}} \right], \label{F3'a2f2}\\
\text{ and } \:T^{-}f_{3}&\equiv& \left[g_{2,2}^{0}, \frac{Y^{r}+(X+Y)^{r}}{a_{2}}\right] + \left[g_{2,3}^{0}, \frac{Y^{r}+(X+Y)^{r}}{a_{2}}\right] \nonumber\\
&\equiv&\left[g_{2,2}^{0}, \frac{2Y^{r}+ H(X,Y)+X^{r}}{a_{2}}\right] + \left[g_{2,3}^{0},\frac{2Y^{r}+ H(X,Y)+X^{r}}{a_{2}}\right].\label{F3'T-f3}
\end{eqnarray}
Combining \eqref{F3'T^+f1}, \eqref{F3'a2f2} and \eqref{F3'T-f3}, we get
\begin{eqnarray*}
    T^{+}f_{1}-a_{2}f_{2}+T^{-}f_{3}&\equiv&\left[g_{2,2}^{0}, \left(\dfrac{2}{a_{2}}-\dfrac{a_{2}}{2}\right)\theta X^{r-3}+\dfrac{2Y^{r}}{a_{2}}\right]
    + \left[g_{2,3}^{0}, \left(\dfrac{2}{a_{2}}-\dfrac{a_{2}}{2}\right)\theta X^{r-3}+ \dfrac{2Y^{r}}{a_{2}}\right] \nonumber\\
    &\equiv& 0\mod(\wp,X_{r},V_{r}^{(2)}+\theta X_{r-3}).
\end{eqnarray*}
\noindent\underline{\textbf{Radius 3}}\\
We now compute radius $3$ of $(T-a_{2})f$. Going modulo $\wp$ and modulo $X_{r}$ we get
\begin{eqnarray}
    T^{+}f_{2} &\equiv& \left[g_{3,2}^{0}, X^{r-1}Y+\dfrac{H(X,2Y)}{a_{2}^2}\right] + \left[g_{3,2^2+2}^{0}, X^{r-1}Y-X^{r}+ \dfrac{2X^{r}}{a_{2}^2} + \dfrac{H(X,2Y-X)}{a_{2}^2}\right]\nonumber\\
    && + \left[g_{3,3}^{0},X^{r-1}Y+\dfrac{H(X,2Y)}{a_{2}^2}\right] + \left[g_{3,2^2+3}^{0}, X^{r-1}Y-X^{r}+ \dfrac{2X^{r}}{a_{2}^2} + \dfrac{H(X,2Y-X)}{a_{2}^2}\right] \nonumber\\
    &\equiv& \left[g_{3,2}^{0}, X^{r-1}Y + \dfrac{2rX^{r-3}\theta}{a_{2}^2}\right] + \left[g_{3,2^2+2}^{0}, X^{r-1}Y+ \dfrac{2rX^{r-3}\theta}{a_{2}^2}\right]\nonumber\\
    & & +\left[g_{3,3}^{0},X^{r-1}Y + \dfrac{2rX^{r-3}\theta}{a_{2}^2}\right] + \left[g_{3,2^2+3},X^{r-1}Y +\dfrac{2rX^{r-3}\theta}{a_{2}^2}\right], \label{F3'T+f2}     
    \end{eqnarray}
 by part \eqref{Heven2} of Lemma \ref{Heven}. Further adding this to
\begin{eqnarray}\label{F3'a2f3}
  -a_{2}f_{3} &=& \left[g_{3,2}^{0}, -Y^{r}+X^{r-2}Y^{2}\right]  + \left[g_{3,2^{2}+2}^{0},-Y^{r}+X^{r-2}Y^{2}\right] + \left[g_{3,3}^{0}, -Y^{r}+X^{r-2}Y^{2}\right] \nonumber\\
  & &+ \left[g_{3,2^2+3}^{0},-Y^{r}+X^{r-2}Y^{2}\right],
\end{eqnarray}
we get that
\begin{eqnarray*}
  T^{+}f_{2}-a_{2}f_{3} &\equiv& \left[g_{3,2}^{0}, \left(1+\dfrac{2r}{a_{2}^2}\right)\theta X^{r-3} -Y^{r}\right] + \left[g_{3,2^2+2}^{0},\left(1+\dfrac{2r}{a_{2}^2}\right)\theta X^{r-3} -Y^{r}\right] \nonumber\\
   && + \left[g_{3,3}^{0},\left(1+\dfrac{2r}{a_{2}^2}\right)\theta X^{r-3}-Y^{r}\right]
    + \left[g_{3,2^2+3}^{0},\left(1+\dfrac{2r}{a_{2}^2}\right)\theta X^{r-3}-Y^{r} \right] \nonumber\\
    &\equiv & 0\mod(\wp,X_{r},V_{r}^{(2)}+\theta X_{r-3}).
    \end{eqnarray*}
\noindent\underline{\textbf{Radius 4}}\\
From the definition of $f_{3}$ as in \eqref{F3'}, since $r$ is even, it follows that $T^{+}f_{3} \equiv 0 \mod \wp$.\\

We now compute  the parts of $(T-a_{2})f$ supported on the negative radii which involves studying the action of $T$ in the left side of the tree. In the steps that follow, we shall use Lemma \ref{T+toT-} to simplify the calculations as already done in $\S \ref{slope<1}$ for slopes in $(0,1)$.

\noindent\underline{\textbf{Radius -1}}\\   
We calculate, using the hypothesis $r>4$, that
\begin{equation}\label{F3'T-f0-a2f-1}
    T^{-}f_{0} -a_{2}f_{-1} \equiv \left[g_{0,0}^{1}, XY^{r-1}-X^{3}Y^{r-3}\right] \equiv \left[g_{0,0}^{1}, -\theta(Y^{r-3}+XY^{r-4})\right]\mod \wp.
\end{equation}
 Next using equation \eqref{T+toT-d} from Lemma \eqref{T+toT-}, we get that
 \begin{eqnarray}
T^{-}f_{-2} &=& T^{-}\left[g_{1,1}^{1}, \dfrac{(XY^{r-1}-X^2Y^{r-2})}{2}+\dfrac{2X^{2}Y^{r-2}}{a_{2}^2} + \dfrac{H(X,Y)}{a_{2}^2} \right] \nonumber\\
&=& \begin{pmatrix}
    0 & 1\\
    1 & -1
\end{pmatrix}T^{-}\left[g_{2,3}^{0},\dfrac{(X^{r-1}Y-X^{r-2}Y^2)}{2}+\dfrac{2X^{r-2}Y^{2}}{a_{2}^2} + \dfrac{H(X,Y)}{a_{2}^2}\right] \nonumber\\
&\equiv & 0\mod(\wp, X_{r}, V_{r}^{(2)}+\theta X_{r-3})\label{F3'T-f-2},
\end{eqnarray}
where the last congruence follows from the computations leading to the equation \eqref{F3'T-f2} above. Now adding \eqref{F3'T-f0-a2f-1} and \eqref{F3'T-f-2} we get
\begin{equation}\label{sl1F3'rad-1}
T^-f_{-2}-a_{2}f_{-1}+T^{-}f_{0} \equiv \left[g_{0,0}^{1}, -\theta XY^{r-4}    \right]\mod(\wp, X_{r},V_{r}^{(2)}+\theta X_{r-3}).
\end{equation} 

 
\noindent\underline{\textbf{Radius -2}}\\ 
Using the equation \eqref{T+toT-a} in Lemma \ref{T+toT-}, we have
\begin{eqnarray}
T^{+}f_{-1} &=& T^{+}\left[g_{0,0}^{1}, \dfrac{X^{3}Y^{r-3}}{a_{2}}\right] =  w.T^{+}\left[g_{1,0}^{0}, \dfrac{X^{r-3}Y^{3}}{a_{2}}\right] \nonumber\\
&\equiv& w.\left[g_{2,2}^{0}, \dfrac{6X^{r-1}Y-X^{r}}{a_{2}}\right]\equiv \left[g_{1,1}^{1}.w,\:\dfrac{2X^{r-1}Y-X^{r}}{a_{2}}\right] \nonumber\\
&\equiv& \left[g_{1,1}^{1}, \dfrac{2XY^{r-1}-Y^{r}}{a_{2}}\right] \mod\wp.\label{F3'T+f-1}
\end{eqnarray}
It follows from \eqref{F3'} that
\begin{equation}\label{F3'a2f-2}
    -a_{2}f_{2} = \left[g_{1,1}^{1},  \dfrac{-a_{2}(XY^{r-1}-X^{2}Y^{r-2})}{2}-\dfrac{2X^{2}Y^{r-2}}{a_{2}}-\dfrac{H(X,Y)}{a_{2}}\right].
\end{equation}
We now compute  by directly applying Lemma \ref{T} that 
\begin{eqnarray}
    T^{-}f_{-3} &\equiv &\left[g_{1,1}^{1}, \dfrac{X^{r}}{a_{2}} + \dfrac{(X+Y)^{r}}{a_{2}}\right] \mod\wp \nonumber\\
    &\equiv & \left[g_{1,1}^{1}, \dfrac{2X^{r}+ H(X,Y)+ Y^{r}}{a_{2}}\right] \mod \wp.  \label{F3'T-f-3}
\end{eqnarray}

Adding \eqref{F3'T+f-1}, \eqref{F3'a2f-2} and \eqref{F3'T-f-3}, we get 
\begin{eqnarray*}
     T^{+}f_{-1} - a_{2}f_{2} + T^{-}f_{-3} &\equiv& \left[g_{1,1}^{1}, \left(\dfrac{a_{2}}{2}-\dfrac{2}{a_2}\right)\theta Y^{r-3} + \dfrac{2X^{r}}{a_{2}}\right] \\
    &\equiv & 0\mod (\wp, X_{r},V_{r}^{(2)}+\theta X_{r-3}).
\end{eqnarray*}

\noindent\underline{\textbf{Radius -3}}\\
Using equation \eqref{T+toT-b} of Lemma $\ref{T+toT-}$ and going modulo $(\wp,X_{r})$ we get
\begin{eqnarray}
    T^{+}f_{-2}&=& T^{+}\left[g_{1,1}^{1},\dfrac{1}{2}(XY^{r-1}-X^{2}Y^{r-2})+ \dfrac{2}{a_{2}^{2}}X^{2}Y^{r-2} + \dfrac{H(X,Y)}{a_{2}^{2}}\right] \nonumber\\
   &=& \begin{pmatrix}
       0 & 1\\
       1 & -1
   \end{pmatrix}T^{+}\left[g_{2,3}^{0}, \dfrac{(X^{r-1}Y-X^{r-2}Y^{2})}{2}+\dfrac{2X^{r-2}Y^2}{a_{2}^2}+ \dfrac{H(X,Y)}{a_{2}^2}\right] \nonumber\\
   &\equiv& \begin{pmatrix}
       0 & 1\\
       1 & -1
   \end{pmatrix}\left(\left[g_{3,3}^{0}, X^{r-1}Y + \dfrac{2r\theta X^{r-3}}{a_{2}^2}\right]+ \left[g_{3,2^2+3}^{0}, X^{r-1}Y+ \frac{2r\theta X^{r-3}}{a_{2}^2} \right]\right).\nonumber
   \end{eqnarray}
The above congruence follows from the part of \eqref{F3'T+f2} supported on $g_{3,3}^{0}$ and $g_{3,2^2+3}^{0}$ in the R.H.S of the tree. Since $\begin{pmatrix}
    0 & 1\\
    1 & -1
\end{pmatrix}g_{3,3}^{0} = g_{2,1}^{1}.w$ and $\begin{pmatrix}
    0 & 1\\
    1 & -1
\end{pmatrix}g_{3,2^{2}+3}^{0} = g_{2,3}^{1}w$, we have 
\begin{eqnarray}\label{F3'T+f-2}
 T^{+}f_{-2}&\equiv& \left[g_{2,1}^{1},XY^{r-1}+ \dfrac{-2r\theta Y^{r-3}}{a_{2}^{2}}\right]+ \left[g_{2,3}^{1},XY^{r-1}+ \dfrac{-2r\theta Y^{r-3}}{a_{2}^{2}}\right] \mod(\wp,X_{r}) \nonumber\\
 &\equiv& \left[g_{2,1}^{1},XY^{r-1}\right]+ \left[g_{2,3}^{1},XY^{r-1}\right] \mod(\wp,X_{r},V_{r}^{(2)}+\theta X_{r-3}),
 \end{eqnarray}
since $r$ is even. Further we have
\begin{eqnarray}\label{F3'a2f-3}
    -a_{2}f_{-3} &=&\left[g_{2,1}^{1},-X^{r}+X^2Y^{r-2}\right]+\left[g_{3,2}^{0}, -X^{r}+X^2Y^{r-2}\right] \nonumber\\
    &\equiv&\left[g_{2,1}^{1},X^2Y^{r-2}\right]+\left[g_{3,2}^{0},X^2Y^{r-2}\right]\mod(\wp,X_{r}).\end{eqnarray}

Combining \eqref{F3'T+f-2} and \eqref{F3'a2f-3}, we get
\begin{equation*}
 T^{+}f_{-2}-a_{2}f_{-3}\equiv \left[g_{2,1}^{1}, \theta Y^{r-3}\right]+\left[g_{2,3}^{1}, \theta Y^{r-3}\right] \equiv 0 \mod (\wp,X_{r},V_{r}^{(2)}+\theta X_{r-3})
    \end{equation*}
\noindent\underline{\textbf{Radius -4}}\\
Using the definition of the action of  $T^{+}$ in the L.H.S of the tree as in Lemma \ref{T}, we get that radius $-4$ of $(T-a_{2})f$ vanishes modulo $(\wp,X_{r},V_{r}^{(2)}+\theta X_{r-3})$.\\

From the calculations done in all the radii above, we note that $(T-a_2)f$ is integral and its reduction modulo $(\wp, X_r, \theta X_{r-3}+V_r^{(2)})$ is supported on radii $0,1$ and $-1$. To be more precise, on adding the congruences \eqref{sl1F3'rad0}, \eqref{sl1F3'rad1} and \eqref{sl1F3'rad-1} we have
\begin{eqnarray}\label{sl1V1even}
(T-a_{2})f&\equiv&\left[g_{1,0}^{0}, \theta X^{r-4}Y\right] + \left[g_{1,1}^{1},  \theta X^{r-4}Y\right] + \left[g_{0,0}^{1}, -\theta XY^{r-4}\right] \nonumber \\
&& -\dfrac{a_{2}}{2}\left[1, \theta X^{r-4}Y \right ] \mod (\wp,X_{r}, V_{r}^{(2)}+ \theta X_{r-3}).
\end{eqnarray} 

Thus the reduction of $(T-a_{2})f$ modulo $X_{r}$ lands in $\dfrac{V_{r}^{(1)}}{V_{r}^{(2)}}$ and maps to $\left(T-\dfrac{a_{2}}{2}\right)[1,1]$ in $\mathcal{I}(V_{1}) \cong  \mathcal{I}\left(\dfrac{V_{r}^{(1)}}{V_{r}^{(2)}+\theta X_{r-3}}\right)$  via the projection map given in Lemma \ref{Vr1/Vr2}. We recall that $X_{r} \subset (T-a_{2})\left(\mathcal{I}(\Sym^{r}(\bar{\Q}_{2}^{2})\right)$ as shown in the proof of Lemma \ref{RemarkBG09}. Thus we can conclude using \eqref{diag4} that $\left(T-\dfrac{a_{2}}{2}\right)\mathcal{I}(V_{1})$, which is generated as a $\bar{\F}_{2}[G]$-module by the element $\left(T-\dfrac{a_{2}}{2}\right)[1,1]$ in $\mathcal{I}(V_{1})$, is in the kernel of the projection map $\mathcal{I} \left(V_{0}\right) \to 
 \sF'_{3}$.
\end{proof}
Next, we want to determine the factor $\sF'_2$ of $\bar\Theta_{k,a_2}$. We begin by stating some technical results similar to Lemma \ref{sl1cmb1} and \ref{Heven} above. 
\begin{lemma}\label{sl1cmb2}
Let $r\geq 4$ and $t=v(r-2)$. Then  we have  
\begin{enumerate}
    \item 
$v\left({r-1 \choose j}\right)+j \geq t+1$ for all $j\geq2$, and \label{sl1cmb2a}
\item $v\left({r-2 \choose j}\right)+j \geq t+1$ for all $j\geq1$.\label{sl1cmb2b}
\end{enumerate}
\end{lemma}
\begin{proof}
    The part (1) with $j>2$ and the part (2) both follow by applying Lemma \ref{CMB2} and Lemma \ref{CMB1} respectively with $r$ replaced by $r-1$ and $n$ by $j$. For $j=2$ in part (1), one notes that left hand side equals $v(r-1)+t+1$ and hence the inequality holds.
\end{proof}
\begin{lemma}\label{K}
For $r\geq 6$, let $K(X,Y):= \displaystyle\sum_{j=3}^{r-3}\left[{r-1 \choose j}+{r-2\choose j} \right]X^{r-j}Y^j$. 
Let $t:=v(r-2)$. Then we have,
\begin{enumerate}
    \item $K(X,2Y) \equiv \begin{cases}
        0\mod 2^{t+1} & \text{ if }r\text{ even},\\
        0\mod 2^2 & \text{ if }r\text{ odd }.
        \end{cases} $
    \item $K(X,2Y-X) \equiv\begin{cases}
    -rX^{r}-2^2X^{r-1}Y\mod 2^{t+1} & \text{if $r$ even,}\label{Kb}\\
     -2{r-1 \choose 2}X^{r} \mod 2^{2}& \text{if $r$ odd.}    \end{cases}$
    \item $K(2X,Y) \equiv \begin{cases}
        0\mod 2^{t+1} & \text{ if }r\text{ even},\\
        0\mod 2^2 & \text{ if }r\text{ odd }.
        \end{cases} $
    \end{enumerate}
\end{lemma}
\begin{proof}
 We prove the lemma for $r$ even. For $r$ odd, the claim follows from a similar and simpler calculation. The first part and third part of the lemma for even $r$ are direct consequences of Lemma \ref{sl1cmb2}.
 For the second part, we compute
 \begin{eqnarray}
     && K(X,2Y-X) = \displaystyle\sum_{j=3}^{r-3}\left[{r-1 \choose j}+{r-2\choose j} \right]X^{r-j}(2Y-X)^j \nonumber\\
     &=& \displaystyle\sum_{j=3}^{r-3}\left[{r-1 \choose j}+{r-2\choose j} \right]\sum_{i=0}^{j}2^{i}(-1)^{j-i}{j \choose i}X^{r-i}Y^{i}\nonumber \\
     &=& \sum_{i=0}^{r-3}
     2^i\left[{r-1 \choose i}\sum_{j'\geq 0,\, 3-i}^{r-3-i}(-1)^{j'}{r-1-i\choose j'}
     +{r-2\choose i}\sum_{j'\geq 0,\, 3-i}^{r-3-i}(-1)^{j'}{r-2-i\choose j'}\right] X^{r-i}Y^{i}\nonumber\\
     &\equiv & 
     \left(\sum_{j'=3}^{r-3}(-1)^{j'}\left({r-1\choose j'}+{r-2\choose j'}\right)\right) X^{r}  +
     \left(2(r-1)\sum_{j'=2}^{r-4}(-1)^{j'}{r-2\choose j'}\right)X^{r-1}Y\mod 2^{t+1} \nonumber\\
     &\equiv & 
     \left(\sum_{j'=3}^{r-3}(-1)^{j'}\left({r-1\choose j'}+{r-2\choose j'}\right)\right) X^{r} +
     \left(2\sum_{j'=2}^{r-4}(-1)^{j'}{r-2\choose j'}\right)X^{r-1}Y\mod2^{t+1}, \label{Kproof}
\end{eqnarray}
where the second last congruence above follows from Lemma \ref{sl1cmb2}. We shall now compute the coefficient of $X^r$ in the expression above as
\begin{eqnarray*}
   \sum_{j'=3}^{r-3}(-1)^{j'}\left({r-1\choose j'}+{r-2\choose j'}\right) 
  & =& (2r-5)-(r-2)^2 +(1-r) \\
  &\equiv &-1-(r-2)^2+(1-r)\mod 2^{t+1}\\
   &\equiv & 
       -r \mod 2^{t+1}
   \end{eqnarray*}
   Similarly the coefficient of $X^{r-1}Y$ above can be simplified as
   \begin{eqnarray*}
   2\sum_{j'=2}^{r-4}(-1)^{j'}{r-2\choose j'}\equiv 
       -2^2 \mod 2^{t+1}. 
   \end{eqnarray*}
By substituting these two coefficients in \eqref{Kproof} above, we obtain part \eqref{Kb} of the lemma for $r$ even. 
\end{proof}

Define $\alpha := \dfrac{a_{2}^2-2^2{r\choose2}}{2a_{2}}, \:\tau:= v(\alpha)$ and $t:=v(r-2).$ Let $\beta_{0}$ be defined as 
$$\beta_{0} = \begin{cases}
    \alpha & \text{ if } \tau \leq t-1\\
    \frac{r-2}{2} & \text{ if } \tau>t-1
\end{cases}$$
Note that $v(\beta_{0})= \min(\tau,t-1)$. 
\begin{remark}\label{remparityslope1}
   It is clear from the definition that $\tau\geq 0$. Further if $r$ is odd, then $t=0$ and hence $\tau> t-1$. Therefore if $\tau \leq t-1$, then $r$ has to be even. 
\end{remark}

\begin{lemma}\label{sl1par}  
Let $v(a_2)=1$ and let $\alpha$ and $\beta_0$ be defined as above. Then 
\begin{enumerate}
      \item the quantity $c_{1}:=\frac{a_{2}^2-2r}{a_{2}^2\beta_{0}}$ is always integral and further we have $$c_{1}
      \equiv\begin{cases}
      \frac{2}{a_2}\mod\wp &\text{ if } \tau\leq t-1\\
      \frac{2^2r}{a_2^2}\mod\wp
      &\text{ if } \tau > t-1,
           \end{cases}
    $$ \label{sl1par1}
   \item and the quantity $c_{2}:=\frac{a_{2}^2-2^2}{2a_{2}\beta_{0}}$ is always integral and  further we have $$ c_{2}
   \equiv\begin{cases}
   1+\frac{(r-2)}{a_2\alpha}\mod\wp &\text{ if }\tau\leq t-1\\
   \frac{2(r+1)}{a_2}\mod\wp &\text{ if }\tau> t-1.
   \end{cases}
   $$\label{sl1par2} 
   \end{enumerate} 
\end{lemma}
\begin{proof}
It is an easy check and left to the reader.
\end{proof}

\begin{prop}\label{F2'}
    Let $r\geq6$ be even, $v(a_{2})=1$ and let $\tau$ and $t$ be as above.
    \begin{enumerate}
        \item If $\tau\leq t-1$, then $\mathcal{F'}_{2}$ is a quotient of $\dfrac{\mathcal{I}(V_{1})}{(T^{2}-c_{0}T+1)}$ where $c_{0} = \overline{\dfrac{a_{2}}{2} + \dfrac{r-2}{2\alpha}}$ .
        \item If $\tau>t-1$, then $\mathcal{F'}_{2}$ is a quotient of $\dfrac{\mathcal{I}(V_{1})}{(T)}$.                             \end{enumerate}
\end{prop}
\begin{proof}
Consider function $f=\underset{n\in\Z}\sum f_n$, where $f_n$'s are defined as follows:
\begin{eqnarray}
f_{0}&=& \left[1, \frac{X^{r}+X^{r-1}Y+X^{2}Y^{r-2}}{a_{2}\beta_{0}} + \frac{2r}{a_{2}^{3}\beta_{0}}K(X,Y)\right]\nonumber\\
f_{1} &=& \left[g_{1,0}^{0}, \dfrac{-H(X,Y)}{a_{2}^{2}\beta_{0}}+ \dfrac{X^{r-1}Y}{2\beta_{0}}+ \dfrac{c_{2}}{a_{2}}X^{r-2}Y^{2}\right]\label{F2'eqeven}\\
& & - \left[g_{1,1}^{0},  \dfrac{-H(X,Y)}{a_{2}^{2}\beta_{0}}+ \dfrac{X^{r-1}Y}{2\beta_{0}}+  \dfrac{c_{2}}{a_{2}}X^{r-2}Y^{2}                \right]  \nonumber\\
f_{n} &=& \left[g_{n,0}^{0}, \dfrac{a_{2}^{n-3}}{\beta_{0}}(Y^{r}-X^{r-2}Y^{2})\right] -  \left[g_{n,2}^{0}, \dfrac{a_{2}^{n-3}}{\beta_{0}}(Y^{r}-X^{r-2}Y^{2})\right] \nonumber\\ 
& & - \left[g_{n,1}^{0}, \dfrac{a_{2}^{n-3}}{\beta_{0}}(Y^{r}-X^{r-2}Y^{2})\right] +  \left[g_{n,3}^{0}, \dfrac{a_{2}^{n-3}}{\beta_{0}}(Y^{r}-X^{r-2}Y^{2})\right]\text{ for } 2\leq n \leq t+2 \nonumber\\
f_{-n} &= &\left[g_{n-1,0}^{1}, -\dfrac{a_{2}^{n-1}}{\beta_{0}}(X^{r}-X^{2}Y^{r-2})\right] \text{ for }\,1 \leq n\leq t,\nonumber
\end{eqnarray}
where $K(X,Y)$ is as defined in Lemma \ref{K} and $c_{2}$ is as in Lemma \ref{sl1par}.
We now compute  $(T-a_{2})f$ radius by radius.

\noindent\underline{\textbf{Radius 0}}\\
 Applying Lemma \ref{T}, we get 
\begin{equation}\label{F2'T-f-1}
    T^{-}f_{-1} \equiv 
        \left[1, \dfrac{-X^r}{\beta_{0}}\right]\mod\wp,
\end{equation}
since $v(\beta_0)\leq t-1<r-2$.
Further we have ,
\begin{equation}\label{F2'a2f0}
 -a_{2}f_{0} = \left[1, \dfrac{-\left(X^{r}+X^{r-1}Y+X^{2}Y^{r-2}\right)}{\beta_{0}} -\dfrac{2r}{a_{2}^2\beta_{0}}K(X,Y) \right].
 \end{equation}
We now calculate $T^{-}f_{1}$ using Lemma \ref{Heven} and the condition that $r$ is even:
\begin{eqnarray}\label{F2'T-f1}      T^{-}f_1 &\equiv& \left[1, \dfrac{-H(2X,Y)}{a_{2}^2\beta_{0}}+\dfrac{H(2X,X+Y)}{a_{2}^2\beta_{0}}\right] \mod \wp \nonumber\\
    &\equiv&  \left[1, -\dfrac{(2rXY^{r-1}+2r    X^{2}Y^{r-2})}{a_{2}^2\beta_{0}}+\dfrac{2rX(X+Y)^{r-1}+2rX^{2}(X+Y)^{r-2}} {a_{2}^2\beta_{0}}\right]\mod\wp \nonumber\\
    &=&\left[1, \frac{2r}{a_2^2\beta_0}\left(2X^r+(2r-3)X^{r-1}Y+(r-2)^2X^{r-2}Y^2+K(X,Y)+(r-1)X^2Y^{r-2}\right)\right] \nonumber\\
    &\equiv&\left[1, \dfrac{2^2rX^r}{a_{2}^2\beta_{0}} + \dfrac{2r(X^{r-1}Y+X^{2}Y^{r-2})}{a_{2}^2\beta_{0}}+\dfrac{2r}{a_{2}^2\beta_{0}}K(X,Y)\right]\mod\wp.
\end{eqnarray}
The last congruence follows reducing the previous expression modulo $\wp$, using the fact that for $r$ even, the valuation $v\left(\frac{2r(r-2)}{a_2^2\beta_0}\right)= t-1-v(\beta_0)+v(r)\geq v(r)>0$.


 On adding \eqref{F2'T-f-1},\eqref{F2'a2f0} and \eqref{F2'T-f1}, we get that,
 \begin{eqnarray}
T^{-}f_{-1}-a_{2}f_{0}+T^{-}f_{1} &\equiv& \left[1, -\dfrac{(a_{2}^2-2r)}{a_{2}^2\beta_{0}}(X^{r-1}Y+X^2Y^{r-2}+2X^r)\right] \nonumber\\
&\equiv& \left[1, -c_{1}(X^{r-1}Y-X^2Y^{r-2})\right]\mod\wp \nonumber
\end{eqnarray}
 since $c_{1}$ is integral by Lemma \ref{sl1par}.  
By using the method of alternating sum, we note that $X^{r-1}Y-X^{2}Y^{r-2} \equiv \theta X^{r-3} + \theta \displaystyle\sum_{j=1}^{r-4}X^{r-3-j}Y^{j} \equiv \theta X^{r-3} \mod V_{r}^{(2)}$ since $r$ is even. 
\begin{equation}\label{sl1F2'rad0}
T^{-}f_{-1}-a_{2}f_{0}+T^{-}f_{1} \equiv \begin{cases} 
\left[   1, -\dfrac{2}{a_{2}}\theta X^{r-3}\right] \mod(\wp, V_{r}^{(2)}+X^{r}) & \text{ if }\tau\leq t-1,\\
0 \mod(\wp, V_{r}^{(2)}+X^{r}) & \text{ if } \tau>t-1. 
\end{cases}
\end{equation}
\noindent\underline{\textbf{Radius 1}}\\
Using Lemma \ref{sl1cmb2} and Lemma \ref{K} along with the definition of $T^+$, we compute
\begin{eqnarray}\label{F2'T+f0}
    T^{+}f_{0} &\equiv& \left[g_{1,0}^{0}, \dfrac{X^{r}+ 2X^{r-1}Y}{a_{2}\beta_{0}}\right] + \left[g_{1,1}^{0}, \dfrac{X^{r}+2X^{r-1}Y}{a_{2}\beta_{0}}+\dfrac{2r}{a_{2}^{3}\beta_{0}}K(X,2Y-X) \right] \nonumber\\
    &\equiv& \left[g_{1,0}^{0},\dfrac{X^{r}+ 2X^{r-1}Y}{a_{2}\beta_{0}}\right]+ \left[g_{1,1}^{0}, \dfrac{X^{r}+ 2X^{r-1}Y}{a_{2}\beta_{0}} + \dfrac{-2r^2X^{r}-2^3rX^{r-1}Y}{a_{2}^3\beta_{0}}   \right] \nonumber\\
    &\equiv& \left[g_{1,0}^{0},\dfrac{X^{r}+ 2X^{r-1}Y}{a_{2}\beta_{0}}\right]+ \left[g_{1,1}^{0}, \dfrac{(a_{2}^2-2r^2)X^{r}}{a_{2}^3\beta_{0}} + \dfrac{(2a_{2}^2-2^3r)X^{r-1}Y}{a_{2}^3\beta_{0}}   \right] \nonumber\\
   &\equiv& \left[g_{1,0}^{0},\dfrac{X^{r}+ 2X^{r-1}Y}{a_{2}\beta_{0}}\right]+ \left[g_{1,1}^{0}, \dfrac{(a_{2}^2-2r^2)X^{r}}{a_{2}^3\beta_{0}} - \dfrac{2X^{r-1}Y}{a_{2}\beta_{0}}   \right]   \mod \wp.
    \end{eqnarray}
The last congruence follows from Lemma \ref{sl1par} part \eqref{sl1par1}. 
Next we have
\begin{eqnarray}\label{F2'a2f1}
    -a_{2}f_{1} &=& \left[g_{1,0}^{0}, \dfrac{H(X,Y)}{a_{2}\beta_{0}}-\dfrac{a_{2}X^{r-1}Y}{2\beta_{0}} -c_{2}X^{r-2}Y^{2}\right] \nonumber\\
    &&+ \left[g_{1,1}^{0},-\dfrac{H(X,Y)}{a_{2}\beta_{0}}+\dfrac{a_{2}X^{r-1}Y}{2\beta_{0}} +c_{2}X^{r-2}Y^{2}\right]. 
\end{eqnarray}
Now using the definition of $T^-$, we compute
\begin{equation}\label{F2'T-f2}
    T^{-}f_{2} = \left[g_{1,0}^{0}, \dfrac{-X^{r}-H(X,Y)}{a_{2}\beta_{0}}\right] + \left[g_{1,1}^{1}, \dfrac{X^{r}+H(X,Y)}{a_{2}\beta_{0}}\right].
\end{equation}

On adding up \eqref{F2'T+f0},\eqref{F2'a2f1} and \eqref{F2'T-f2} we get modulo $\wp$,
\begin{equation*}
  T^{+}f_{0}-a_{2}f_{1}+T^{-}f_{2}\equiv \left[g_{1,0}^{0}, -c_{2}\theta X^{r-3}\right] + \left[g_{1,1}^{0}, \dfrac{2(a_{2}^2-r^2)X^{r}}{a_{2}^3\beta_{0}}+ c_{2}\theta X^{r-3} \right].
    \end{equation*}
     The coefficient of $X^{r}$ above is congruent to $\frac{2^2a_{2}\alpha-2r(r-2)}{a_{2}^3\beta_{0}} = \frac{2^2\alpha}{a_{2}^2\beta_{0}}-\frac{2r(r-2)}{a_{2}^3\beta_{0}} \mod \wp$. Since $v(\beta_{0})= \min(\tau,t-1)$, this coefficient is integral. Thus going modulo $(\wp,V_{r}^{(2)}+X_{r})$ and further applying  Lemma \ref{sl1par} \eqref{sl1par2} for $r$  even, we have 
    \begin{eqnarray}\label{sl1F2'rad1}
   T^{+}f_{0}-a_{2}f_{1}+T^{-}f_{2}
    &\equiv &\begin{cases} 
    \left[g_{1,0}^{0}, -\left(1+\frac{r-2}{a_{2}\alpha} \right)\theta X^{r-3}\right] + \left[g_{1,1}^{0}, \left(1+\frac{r-2}{a_{2}\alpha}\right)\theta X^{r-3} \right] & \text{ if } \tau\leq t-1,\\
     \left[g_{1,0}^{0}, \frac{-2\theta X^{r-3}}{a_{2}}\right] + \left[g_{1,1}^{0}, \frac{2\theta X^{r-3}}{a_{2}}\right]  &\text{ if } \tau>t-1.
     \end{cases}\nonumber\\
     &&
     \end{eqnarray}
\noindent\underline{\textbf{Radius $n\geq 2$}}\\
We now compute $T^{+}f_{n-1}$ for $n\geq 3.$ In  this computation, the second branch of $T^{+}$ with the action $(X,Y)\to(X,2Y-X)$ vanishes  by \ref{sl1cmb2}\eqref{sl1cmb2b}, and we have
\begin{eqnarray}\label{F2'T+fn-1}
  T^{+}f_{n-1} &\equiv& \left[g_{n,0}^{0}, \dfrac{-2^2a_{2}^{n-4}X^{r-2}Y^2}{\beta_{0}}\right] +\left[g_{n,2}^{0}, \dfrac{2^2a_{2}^{n-4}X^{r-2}Y^2}{\beta_{0}}\right]\nonumber\\
  &&+\left[g_{n,1}^{0}, \dfrac{2^2a_{2}^{n-4}X^{r-2}Y^2}{\beta_{0}}\right] +\left[g_{n,3}^{0}, \dfrac{-2^2a_{2}^{n-4}X^{r-2}Y^2}{\beta_{0}}\right] \mod\wp.
\end{eqnarray}
%

However for n=2, the calculation is more involved and we have
\begin{eqnarray*}
    T^{+}f_{1} &=& \left[g_{2,0}^{0}, \dfrac{-H(X,2Y)}{a_{2}^2\beta_{0}}+\dfrac{X^{r-1}Y}{\beta_{0}} +\frac{2^2c_{2}}{a_{2}}X^{r-2}Y^2\right]\\ 
    &&+ \left[g_{2,2}^{0}, \dfrac{-H(X,2Y-X)}{a_{2}^2\beta_{0}}+\dfrac{X^{r-1}(2Y-X)}{2\beta_{0}} +\dfrac{c_{2}}{a_{2}}X^{r-2}(2Y-X)^2\right] \\
    &&-\left[g_{2,1}^{0}, \dfrac{-H(X,2Y)}{a_{2}^2\beta_{0}}+\dfrac{X^{r-1}Y}{\beta_{0}}+\frac{2^2c_{2}}{a_2}X^{r-2}Y^2\right] \\ 
    &&-\left[g_{2,3}^{0}, \dfrac{-H(X,2Y-X)}{a_{2}^2\beta_{0}}+\dfrac{X^{r-1}(2Y-X)}{2\beta_{0}}+\dfrac{c_{2}}{a_{2}}X^{r-2}(2Y-X)^2\right]. 
\end{eqnarray*}
By using Lemma \ref{Heven} for $r$ even and part \eqref{sl1par2} of Lemma \ref{sl1par}, we get modulo $\wp$,
\begin{eqnarray}\label{F2'T+f1}
T^{+}f_{1}&\equiv &\left[ g_{2,0}^{0}, \left( \dfrac{a_{2}^2-2r}{a_{2}^2\beta_{0}}\right)X^{r-1}Y-\dfrac{2rX^{r-2}Y^2}{a_{2}^2\beta_{0}}  \right] + \left[ g_{2,2}^{0}, \left( \dfrac{a_{2}^2-2r}{a_{2}^2\beta_{0}}\right)X^{r-1}Y+\dfrac{2rX^{r-2}Y^2}{a_{2}^2\beta_{0}}  \right] \nonumber\\
 && - \left[ g_{2,1}^{0},  \left( \frac{a_{2}^2-2r}{a_{2}^2\beta_{0}}\right)X^{r-1}Y-\frac{2rX^{r-2}Y^2}{a_{2}^2\beta_{0}}\right] - \left[ g_{2,3}^{0},  \left( \frac{a_{2}^2-2r}{a_{2}^2\beta_{0}}\right)X^{r-1}Y+\frac{2rX^{r-2}Y^2}{a_{2}^2\beta_{0}}  \right] \nonumber\\
 && \equiv \left[ g_{2,0}^{0},  c_{1}X^{r-1}Y-\dfrac{2rX^{r-2}Y^2}{a_{2}^2\beta_{0}}  \right] + \left[ g_{2,2}^{0},  c_{1}X^{r-1}Y+\dfrac{2rX^{r-2}Y^2}{a_{2}^2\beta_{0}}  \right] \nonumber\\
 && - \left[ g_{2,1}^{0},   c_{1}X^{r-1}Y-\frac{2rX^{r-2}Y^2}{a_{2}^2\beta_{0}}\right] - \left[ g_{2,3}^{0}, c_{1}X^{r-1}Y+\frac{2rX^{r-2}Y^2}{a_{2}^2\beta_{0}}  \right].
 \end{eqnarray}
We note that in the second branch of $T^{+}$, the coefficients of the $X^{r}$ terms add up to zero.

For $n\geq 2$, we have,
\begin{eqnarray}\label{F2'a2fn}
-a_{2}f_{n}& = &\left[g_{n,0}^{0}, \dfrac{-a_{2}^{n-2}(Y^{r}-X^{r-2}Y^{2} )}{\beta_{0}}\right] +  \left[g_{n,2}^{0}, \dfrac{a_{2}^{n-2}(Y^{r}-X^{r-2}Y^{2})}{\beta_{0}}\right] \nonumber\\
&&+ \left[g_{n,1}^{0},\dfrac{a_{2}^{n-2}(Y^{r}-X^{r-2}Y^{2})}{\beta_{0}}\right] + \left[g_{n,3}^{0},\dfrac{-a_{2}^{n-2}(Y^{r}-X^{r-2}Y^{2})}{\beta_{0}}\right].
\end{eqnarray}
Using the definition of $T^-$ and that $r-2>t-1$ in this context, we get that for all $n\geq2$,
\begin{equation}\label{F2'T-fn+1}
T^{-}f_{n+1} \equiv \left[g_{n,0}^{0}, \dfrac{a_{2}^{n-2}Y^{r}}{\beta_{0}}\right] -\left[g_{n,2}^{0}, \dfrac{a_{2}^{n-2}Y^{r}}{\beta_{0}}\right]-\left[g_{n,1}^{0}, \dfrac{a_{2}^{n-2}Y^{r}}{\beta_{0}}\right]+\left[g_{n,3}^{0}, \dfrac{a_{2}^{n-2}Y^{r}}{\beta_{0}}\right] \mod\wp.
\end{equation}
 Thus using \eqref{F2'a2fn}, \eqref{F2'T-fn+1}, \eqref{F2'T+fn-1} and \eqref{F2'T+f1}, we get that for $n\geq 3$, \begin{eqnarray}\label{sl1F2'radn}
& &T^{+}f_{n-1}-a_{2}f_{n}+T^{-}f_{n+1}  \nonumber \\
 &\equiv&
\left[g_{n,0}^{0},\dfrac {a_{2}^{n-4}(a_{2}^2-2^2)}{\beta_{0}}X^{r-2}Y^2 \right]
-\left[g_{n,2}^{0},\dfrac {a_{2}^{n-4}(a_{2}^2-2^2)}{\beta_{0}}X^{r-2}Y^2\right]  \nonumber\\
 & &  -\left[g_{n,1}^{0},\dfrac {a_{2}^{n-4}(a_{2}^2-2^2)}{\beta_{0}}X^{r-2}Y^2\right]+\left[g_{n,3}^{0},\dfrac {a_{2}^{n-4}(a_{2}^2-2^2)}{\beta_{0}}X^{r-2}Y^2\right]  \nonumber\\
 & \equiv& 2a_{2}^{n-3}c_{2}\left(\left[g_{n,0}^{0}, X^{r-2}Y^2 \right]-\left[g_{n,2}^{0}, X^{r-2}Y^2\right] -\left[g_{n,1}^{0}, X^{r-2}Y^2\right]+\left[g_{n,3}^{0}, X^{r-2}Y^2\right] \right)  \nonumber\\
& \equiv& 0 \mod \wp 
\end{eqnarray}
and for $n=2$,
\begin{eqnarray*}
 T^{+}f_{1}-a_{2}f_{2}+T^{-}f_{3}\equiv c_{1}\left(\left[g_{2,0}^0, \theta X^{r-3}\right] + \left[g_{2,2}^0, \theta X^{r-3}\right] -\left[g_{2,1}^0, \theta X^{r-3} \right]-\left[g_{2,2}^0, \theta X^{r-3} \right]\right).\enspace
\end{eqnarray*}
Thus going$\mod(\wp, X_{r}+V_{r}^{(2)})$, we have
\begin{equation}\label{sl1F2'rad2}   
    T^{+}f_{1}-a_{2}f_{2}+T^{-}f_{3}
\equiv\begin{cases} 
\dfrac{2}{a_{2}}\left([g_{2,0}^0, \theta X^{r-3}] + [g_{2,2}^0, \theta X^{r-3}]-[g_{2,1}^0, \theta X^{r-3}]-[g_{2,2}^0,X^{r-3}\theta]\right)&\text{ if } \tau\leq t-1\\
0 &\text{ if } \tau>t-1
\end{cases}
\end{equation}
by Lemma \ref{sl1par} \eqref{sl1par1} together with the fact that $r$ is even.

\noindent\underline{\textbf{Radius $-n$ for $n\geq 1$}}\\
We have
\begin{equation}\label{F2'a2f-n}
-a_{2}f_{-n}=\left[g_{n-1,0}^{1}, \dfrac{a_{2}^{n}(X^{r}-X^2Y^{r-2})}{\beta_{0}}\right]
\end{equation}
Using Lemma \ref{T+toT-} equation \eqref{T+toT-c} we have
\begin{equation}\label{F2'T-f-n-1}
    T^{-}f_{-(n+1)} \equiv \left[g_{n-1,0}^{1}, -\dfrac{a_{2}^{n}X^{r}}{\beta_{0}}\right]\mod\wp.
\end{equation}
We now compute $T^{+}f_{-(n-1)}$ for $n\geq2$ and $T^{-}f_{0}$. We have modulo $\wp$,
\begin{eqnarray}\label{F2'T+f-n+1}
    T^{+}f_{-(n-1)}&\equiv& \left[g_{n-1,0}^{1}, \dfrac{-a_{2}^{n-2}(2^rX^{r}-2^2X^2Y^{r-2})}{a_{2}^n\beta_{0}}\right] \equiv\left[g_{n-1,0}^{1}, \dfrac{2^2a_{2}^{n-2}X^2Y^{r-2}}{\beta_{0}}\right]\label{F2'T+f-n+1a} \\
    T^{-}f_{0}&\equiv& \left[g_{0,0}^{1}, \dfrac{2^2X^2Y^{r-2}}{a_{2}\beta_{0}}\right] \label{F2'T-f0}
\end{eqnarray}
We note that the second branch of $T^{+}$ in \eqref{F2'T+f-n+1a} vanishes modulo $\wp$ applying Lemma \ref{sl1cmb2} \eqref{sl1cmb2b}. Using \eqref{F2'a2f-n}, \eqref{F2'T-f-n-1}, \eqref{F2'T+f-n+1} and applying Lemma \ref{sl1par}\eqref{sl1par2}, we get that for $n\geq 2$,
\begin{eqnarray*}
-a_{2}f_{-n}+T^{-}f_{-(n+1)}+T^{+}f_{-(n-1)} \equiv \left[g_{n-1,0}^{1}, -2a_{2}^{n-1}c_{2}X^2Y^{r-2} \right] 
\equiv 0 \mod\wp
\end{eqnarray*}
and for $n=1$ further using \eqref{F2'T-f0},\begin{eqnarray*} 
-a_{2}f_{-1}+T^{-}f_{-2}+T^{-}f_{0} \equiv \left[g_{0,0}^{1}, -2c_2X^2Y^{r-2}\right]
\equiv 0 \mod\wp
\end{eqnarray*}
Thus radii $-n, n\geq1$ vanishes irrespective of the relation between $\tau$ and $t$. 

 We recall that in diagram \eqref{diag4}, the $\bar{\F}_{2}[G]$-module $\sF'_{2}$ is a quotient of $\mathcal{I}(V_{1})$ where $V_{1}$ is isomorphic to the sub-module $\dfrac{\theta X_{r-3}}{\theta X_{r-3}^{(1)}}$ of $\dfrac{V_{r}^{(1)}}{V_{r}^{(2)}}$. From the above computations we get that $(T-a_{2})f$ is integral. Its reduction modulo $(\wp, X_{r}+V_{r}^{(2)})$ is supported on radius $0,1$ and $2$ and in fact only on radius $1$ when $\tau>t-1$. Further we observe from equations    \eqref{sl1F2'rad0}, \eqref{sl1F2'rad1} and \eqref{sl1F2'rad2} that this reduction lands in the submodule $\mathcal{I}(V_{1})$ of diagram \eqref{diag4}. In fact its image in $\mathcal{I}(V_{1})$ coincides with 
 \begin{enumerate}
\item $\frac{2}{a_{2}}\cdot \left(T^2- \overline{\left(\frac{a_{2}}{2}+ \frac{r-2}{2\alpha}\right)}T + 1\right)\left[1, X\right]$ \enspace   if   $\tau\leq t-1$ and
\item $\frac{2}{a_{2}} \cdot T\left[1,X\right]$ \enspace  if  $\tau>t-1.$
\end{enumerate}
 Since $\frac{2}{a_{2}}$ is a unit, the operator $T$ is $G$-equivariant and $\mathcal{I}(V_{1})$ is generated by the element $[1,X]$ over $G$, we get that $\sF'_{2}$ is a quotient of $\dfrac{\mathcal{I}(V_{1})}{(T^2-c_{0}T+1)}$ if $\tau \leq t-1$ and $\dfrac{\mathcal{I}(V_{1})}{(T)}$ if $\tau>t-1$.
\end{proof}
 We note that in Theorem \ref{oddthm}, we use diagram \eqref{diag3} to prove that $\bar{\Theta}$ is irreducible when $r$ is odd. However, later we found that one can also use the alternative diagram \eqref{diag4}  to treat the case of odd $r$ for slope 1. 
\begin{prop} \label{bartheta0odd}
If $r\geq5$ is odd then we have,
\begin{enumerate}
\item $\mathcal{F}'_{2}$ is a quotient of $\dfrac{\mathcal{I}(V_{1})}{(T)}.$ \label{theta0odd1}
\item $\mathcal{F}'_{3}=0.$ \label{theta0odd2}
\end{enumerate}
As a result we have $\overline\Theta_0=\mathcal{F'}_{2}\cong \pi(1,0,{1})$. 
\end{prop}

\begin{proof}
For part \eqref{theta0odd1}, consider the function
\begin{equation}\label{frodd}
    f = \left[g_{1,0}^0,\frac{H'(X,Y)}{a_2}\right]+\left[g_{1,1}^0,\frac{H'(X,Y)}{a_2}\right].
    \end{equation}
We note using Lemma \ref{X_r^*} that $H'(X,Y) \equiv \theta X^{r-3} \mod (\wp,X_{r})$. By Lemma \ref{H'}  we see that 
\begin{eqnarray*}
    (T-a_{2})f &\equiv & \left[g_{1,0}^{0}, -H'(X,Y)\right] +\left[ g_{1,1}^{0}, -H'(X,Y) \right]+ \left[1, \frac{2r}{a_{2}}(XY^{r-1} +X(X+Y)^{r-1})\right] \\
    &\equiv & \left[g_{1,0}^{0}, \theta X^{r-3}\right] +\left[ g_{1,1}^{0},  \theta X^{r-3}\right]+ \left[1, \frac{2r}{a_{2}}\displaystyle\sum_{j=2}^{r-2}{r-1 \choose j }X^{r-j}Y^j  \right] \mod (\wp, X_{r}). 
\end{eqnarray*}
We check using Lemma \ref{Vr*} that  $\displaystyle\sum_{j=2}^{r-2}{r-1 \choose j }X^{r-j}Y^j $ lies in $V_{r}^{(2)}$. 
    Thus we have $(T-a_{2})f $ is integral and modulo $(\wp, V_{r}^{(2)}+ X_{r})$ lands in the submodule $\mathcal{I}(V_{1})$  of diagram \eqref{diag4} and its image coincides with $T[1,X]$ in $\mathcal{I}(V_{1})$.

    Now for part $\eqref{theta0odd2}$, 
      we note that $\displaystyle \sum_{j=2}^{r-2}{r \choose j}X^{r-j}Y^j \in V_{r}^{(2)}$, which can be checked by applying Lemma \ref{Vr*}(2) and using that $r$ is odd. Hence \begin{eqnarray*}
          H(X,Y)= r(X^{r-1}Y+XY^{r-1})+ \sum_{j=2}^{r-2}{r \choose j}X^{r-j}Y \equiv X^{r-1}Y+XY^{r-1} \mod V_{r}^{(2)}\\
           \equiv \theta\left(X^{r-3} + (r-4)X^{r-1}Y + Y^{r-3}\right)  \mod V_{r}^{(2)}.
           \end{eqnarray*}
      The last congruence is obtained by alternate addition and subtraction of terms. 
      It follows from the computation above that $[1,H(X,Y)]$ maps to $[1,1]$ in $\mathcal{I}(V_{0})$ by the map induced from projection map in \eqref{Vr1/Vr2b}, 
      thereby generating it as a $\bar{\F}_{2}[G]$-module.   On the other hand, by Lemma \ref{X_r^*} we have  $H(X,Y)\in X_{r}$ and thus $[1,H(X,Y)]$ maps to zero in $\bar{\Theta}$ under $P$. Thus $\mathcal{I}(V_{0})$ maps to zero in $\bar{\Theta}$ under $P$.  Therefore we have $\mathcal{F}'_{3}=0$ in \eqref{diag4}.
   
    Finally, we conclude from part \eqref{theta0odd1}, \eqref{theta0odd2} and Proposition \ref{theta1} that $\bar{\Theta}_{0}=\mathcal{F}'_{2} \cong \pi(1,0,1).$
    \end{proof}

\begin{remark}
    If $r$ is odd, then $t=v(r-2)=0$ and hence $t-1<0\leq\tau$. Indeed the function in \eqref{F2'eqeven} simplifies a lot for odd $r$, in the sense that most terms vanish mod $\wp$ or are at least integral. A variation of the non-integral part of the said function leads us to the one in \eqref{frodd}. Note that when $r$ is odd, we arrive at the same conclusion as in Proposition \ref{F2'} in the case  $\tau>t-1$. \end{remark}
We have not considered the cases $r=2,3$ and $4$ so far. We shall treat them separately in the following theorem. 
\begin{theorem} The description of $\bar{\Theta}_{k,a_{2}}$ for $k=4,5$ and $6$ are as follows.\label{k=4,5,6}
\begin{enumerate}
    \item $\bar{\Theta}^{ss}_{4,a_{2}} \cong \pi(0,\lambda,1)\oplus \pi(0,\lambda^{-1},1)$ where $\lambda^2-\frac{a_{2}}{2}\lambda+1=0$.\label{Theta4a2}
    \item $\bar{\Theta}^{ss}_{5,a_{2}} \cong \pi(0,0,1). $  \label{Theta5a2}
    
    \item $\bar{\Theta}^{ss}_{6,a_{2}} \cong\pi(0,\lambda,1)\oplus \pi(0,\lambda^{-1},1)$ where $\lambda^2-\frac{a_{2}^2-2^2}{2a_{2}}\lambda+1=0.$    \label{Theta6a2}
\end{enumerate}
\end{theorem}

\begin{proof}
For part \eqref{Theta4a2}, we note that $\bar{\Theta}_{1} = \bar{\Theta}_{4,a_{2}}$  when $r=2$ as $V_{2}^{(1)} = V_{2}^{(2)}=0.$ Further in \eqref{diag2}, we have $\mathcal{I}(V_{1})=\mathcal{I}(X_{2})\subset \ker(P)$. Thus $\bar{\Theta}_{1}=\sF_{1}$. Considering function $f$ described as  $$f=f_{1}+f_{-1} = \left[g_{1,0}^0, \dfrac{Y^{2}-XY}{a_{2}}\right]+\left[g_{1,1}^0, \dfrac{Y^{2}-XY}{a_{2}}\right] +  \left[g_{0,0}^{1}, \dfrac{X^{2}-XY}{a_{2}}\right],$$
we get that $(T-a_{2})f$ is integral. Further its reduction mod $\wp$ maps to $\left(T^{2}-\frac{a_{2}}{2}T+1\right)\left[1, \frac{2}{a_{2}}\right]$ in $\mathcal{I}(V_{0})$ via \eqref{VrtoV0}, thereby generating the image of 
$\left(T^{2}-\frac{a_{2}}{2}T+1\right)$ as an $\bar\F_{2}[G]$-module. Therefore we get that $\sF_{1}$ is a quotient of $\dfrac{\mathcal{I}(V_{0})}{\left(T^2-\frac{a_{2}}{2}T+1\right)}$. By analysing the image of mod-$2$ LLC, we get that $\bar{\Theta}_{4,a_{2}}\cong \sF_{1} \cong \pi(0,\lambda,1)\oplus \pi(0,\lambda^{-1},1)$ where $\lambda^2-\frac{a_{2}}{2}\lambda+1=0$. 

We shall now prove part \eqref{Theta5a2}. We have $V_{3}^{(1)} \cong V_{0}$ and $V_{3}^{(2)}=0.$   Thus $\bar{\F}_{2}[G]$-module $\bar{\Theta}_{0}$ is the image of $\mathcal{I}(V_{0})$ under $P$. We now have three factors to consider; $\bar{\Theta}_{0}$, $\sF_{0}$ and $ \sF_{1}$. Using Lemma \ref{X_r^*} we get that $\bar{\Theta}_{0}=0$ since $[1,H(X,Y)]$ generates $\mathcal{I}(V_{3}^{(1)})\cong \mathcal{I}(V_{0})$ in $\mathcal{I}(V_{r})$ and also belongs to $\mathcal{I}(X_{r})\subset \ker(P)$. Further we have $\sF_{0}=0$ in $\eqref{diag2}$ since $\mathcal{I}(V_{1})$ is isomorphic to  $\mathcal{I}\left(\frac{X_{3}}{X_{3}^{(1)}}\right)$ as an $\bar{\F}_{2}[G]$-module and  $\mathcal{I}(X_{3})\subset \ker(P)$. Thus, we have $\bar{\Theta} \cong \sF_{1}$. Now using Lemma \ref{Heven} it can be seen that $(T-a_{2})\left[1,\dfrac{H(X,Y)}{2}\right]$ is integral and that its reduction maps to $T[1,1]\in \mathcal{I}(V_{0})$  under \eqref{VrtoV0}. Thus we get that $\bar{\Theta}_{5,a_{2}} \cong \dfrac{\mathcal{I}(V_{0})}{(T)}.$

For part \eqref{Theta6a2}, we observe that  $V_{4}^{(1)} \cong V_{1}$ and $V_{r}^{(2)}=0$. Let us now consider the function $f$ as in \eqref{F2'eqeven}. When $r=4$, $\beta_{0}=\alpha$ is a unit, $(T-a_{2})f$ is integral and its reduction$\mod (\wp,X_{4})$ lies in $\mathcal{I}(V_{r}^{(1)})$. Further its image coincides with $\frac{2}{a_{2}}.\left(T^2-\left(\frac{a_{2}^2-2^2}{2a_{2}}\right)T+1 \right)[1,X]$ in $  \mathcal{I}(V_{1})\cong \mathcal{I}(V_{r}^{(1)})$. Thus, $\bar{\Theta}_{0}$ is a quotient of $\dfrac{\mathcal{I}(V_{1})}{\left(T^2-\left(\frac{a_{2}^2-2^2}{2a_{2}}\right)T+1 \right)}$. Since we have $\sF_{0}=0$ by the same reasoning as given in the proof of part \eqref{Theta5a2}, it remains to analyze the factor $\sF_{1}$. To this end, we consider the function $g=f-f_{0}$ where $f$ and $f_{0}$ is as given in \eqref{F3'eqeven}. We can see that $(T-a_{2})g$ is integral and its reduction mod $(\wp,V_{r}^{(1)}+X_{r})$ maps to $\left(T-\frac{2}{a_{2}}\right)[1,1]$ in $\mathcal{I}(V_{0}).$ Thus $\sF_{1}$ is a quotient of $\pi(0,\frac{2}{a_{2}},1).$ Our inferences about $\bar{\Theta}_{0}$ and $\sF_{1}$ together with an analysis of the image of mod-$2$ LLC enables us to conclude that $\bar{\Theta}_{6,a_{2}} \cong\pi(0,\lambda,1)\oplus \pi(0,\lambda^{-1},1)$ where $\lambda^2- \left(\frac{a_{2}^2-2^2}{2a_{2}}\right)\lambda+1=0.$
\end{proof}
For $r=2,3$ and $4$, we note that $\tau \leq t-1$, $\tau>t-1$ and $\tau=t-1$ respectively. Thus Theorem \ref{k=4,5,6} implies the main Theorem \ref{ThmB} for these small values of $r$. Next we use Propositions \ref{F3'}, \ref{F2'} and \ref{bartheta0odd} to arrive at the following theorem which holds true for arbitrary $r>4$.
\begin{theorem}\label{mtslope1p=2}
    Let $r\geq 5$ and $v(a_{2})=1$, then the description of $\bar{V}_{k,a_{2}}$ is as follows:
    \begin{enumerate}
        \item If $\tau\leq t-1$, $\bar{V}_{k,a_{2}} \cong \mu_{\lambda}\oplus \mu_{\lambda^{-1}}$ where $\lambda \in \bar{\F}_{2}^{\times}$ is such that $\lambda^2-c_{0}\lambda + 1$, where $c_{0}= \overline{\frac{a_{2}}{2} + \frac{r-2}{2\alpha}}$. 
        \item If $\tau>t-1, \,\bar{V}_{k,a_{2}} \cong \ind(\omega_{2})$.
    \end{enumerate}
\end{theorem}

\begin{proof}
By Proposition \ref{theta1}, we know irrespective of the parity of $r$ that $\bar{\Theta}^{ss}  \cong \bar{\Theta}_{0}$. So we will analyse the factors $\mathcal{F}'_{2}$ and $\mathcal{F}'_{3}$ of $\bar{\Theta}_{0}$ from diagram \eqref{diag4}.

We start the analysis with the case $\tau \leq t-1$. We note that this case occurs only when $r$ is even. In this case, $\mathcal{F}_{3}^{'}$ is a quotient of $\dfrac{\mathcal{I}(V_{0})}{(T-\frac{a_{2}}{2})}$ by Proposition \ref{F3'} and $\mathcal{F}_{2}^{'}$ a quotient of $\dfrac{\mathcal{I}(V_{1})}{(T^{2}-c_{0}T+1)}$ where $c_{0} = \overline{\frac{a_{2}}{2} + \frac{r-2}{2\alpha}}$ by Proposition \ref{F2'}. Then by analysing the image of mod-$2$ LLC, we get that that $\bar{\Theta}$ is reducible and that there are two possibilities and both occurs with varying conditions:
\begin{enumerate}
    \item $\mathcal{F}'_{2} \cong \dfrac{\mathcal{I}(V_{1})}{(T^{2}-c_{0}T+1)}$ and $\mathcal{F}'_{3}=0$.
    \item $\mathcal{F}'_{2}$ is a non-zero quotient of $  \dfrac{\mathcal{I}(V_{1})}{(T-\frac{2}{a_{2}})}$ and 
    $\mathcal{F}'_{3}$ is a non-zero quotient of $\dfrac{\mathcal{I}(V_{0})}{(T-\frac{a_{2}}{2})}$.     
    
    \end{enumerate}

We shall now check the condition for which $\mathcal{F}_{3}'\neq 0$ case can happen. To this end, we have to find the condition for which $\frac{2}{a_{2}}$ satisfies $T^{2}-c_{0}T+1 = 0$ where $c_{0} = \overline{\frac{a_{2}}{2}}-\overline{\frac{2-r}{2\alpha}}$. This happens only when $\overline{\frac{2}{a_{2}}} = \overline{\frac{2-r}{2\alpha}}.$
Using this condition, we can see that the second possibility can occur only when $\tau=t-1$.

 Now let us consider the case when $\tau>t-1$. If $r$ is odd, it is clear from Proposition \ref{oddthm} or alternatively by Proposition \ref{bartheta0odd} that $\bar{\Theta} \cong \pi(1,0,1)$. 
 When $r$ is even,  Proposition \ref{F3'} and Proposition \ref{F2'} respectively gives us that $\mathcal{F}'_{3}$ is a quotient of $\dfrac{\mathcal{I}(V_{0})}{(T-\frac{a_{2}} {2})}$ and $\mathcal{F}'_{2}$ is a quotient of $\dfrac{\mathcal{I}(V_{1})}{(T)}$. On analysing the image of mod-$2$ LLC, we get that only possibility is that $\mathcal{F}'_{3}=0$ and $ \bar{\Theta} \cong \mathcal{F}'_{2} \cong \pi(1,0,1)$.  
\end{proof}
Thus we have arrived at a complete answer to the reduction problem at slope $1$ as claimed in the introduction (cf. Theorem \ref{ThmB}). Now Corollary \ref{cor2} follows as a direct application, using Remark \ref{remparityslope1}.
\section{Numerical examples}
 In this section, we compute some numerical examples using the algorithm introduced in \cite{[Roz17]} and show that the outputs are consistent with our theorems.\\
   
\underline{\bf{Slopes in $(0,1)$}} \\ 
     We recall from \S \ref{slope<1} the following notations:
     $r=k-2 \geq 2, t=v(r-1),0<v(a_{2})<1$ and $\tau'= v\left(\dfrac{a_{2}^2-2r^2}{2a_{2}}\right)$. To compare with the zig-zag conjecture for odd primes, we recall that $\tau= v\left(\dfrac{a_{2}^2-2r}{2a_{2}}\right)$ \cite{[G22]}.
     
     {\underline{Table 1}}\\
     \begin{tabular}{|c|c|c|c|c|c|c|c|}
 		\hline
 		S.No & r & $a_{2}$ & $\tau'$ & $\tau $& $t$ & Contributors to $\bar{\Theta}$& Reducibility of \\
        & & & & & &  \text{ by SAGE output }& $\bar{V}_{k,a_{p}} \text{ (assuming LLC)}$\\
        \hline
 		 $1$& $2$  &$\sqrt{2}$ & $-0.5$ & $-0.5$ & $0$  & $\mathcal{I}(V_{0})/(T)$& Irreducible \\ 		\hline
         $2$& $3$ & $\sqrt{2}$ & $2.5$ & $0.5$ & $1$ & $\mathcal{I}(V_{0})/(T^{2}+1)$ & Reducible\\
        \hline 
         $3$& $4$ & $2^{\frac{1}{3}}$ & $-\frac{2}{3}$ & $-\frac{2}{3}$ & $0$ & $\mathcal{I}(V_{0})/(T)$   & Irreducible\\
        \hline         
 		 $4$&$5$& $\sqrt{6}$&$0.5$ &$0.5$ & $2$ & $\mathcal{I}(V_{0})/(T)$   & Irreducible \\
 		\hline
 		 $5$& $5$ & $13\sqrt{2}$  & $3.5$& $1.5$ &$2$ & $\mathcal{I}(V_{0})/(T^2+1)$ & Reducible   \\
 		\hline
        $6$& $6$  &$\sqrt{6}$ & $-0.5$ & $-0.5$ & $0$ &$\mathcal{I}(V_{0})/(T)$  & Irreducible  \\
 		\hline
        $7$& $7$  &$6+5\sqrt{2}$ & $0.5$ & $1$ & $1$ & --   & Irreducible  \\
 		\hline 
        $8$& $7$  &$4+\sqrt{30}$ & $0.5$ & $2$ & $1$ & --   & Irreducible  \\
 		\hline 
        
        $9$& $9$ & $2^{\frac{1}{7}}$ & $-\frac{6}{7}$ & $-\frac{6}{7}$ & $3$ & --   & Irreducible\\
        \hline
        $10$& $9$ & $2^{\frac{3}{4}} $ & $-0.75 $ & $-0.75 $ & $3$  &-- & Irreducible  \\
        \hline
         $11$& $9$ & $8\pm \sqrt{226} $ & $3 $ & $2.5$ & $3$ &   $\mathcal{I}(V_{0})/(T^2-T+1)$   & Reducible \\
        \hline  
   \end{tabular}
   \vspace{10mm}
   
\noindent 
Note that the quantities $\tau$ and $\tau'$ are not only different, but their relation with $t$ may also vary. For example, in the fifth row, we can see that $\tau<t$ but $\tau'>t$. 
We also note that the data above are consistent with the following consequences of our theorem.\\
$\bullet$  If $r$ is even, then $\bar{V}_{k,a_{2}}$ must be irreducible.\\
$\bullet$ If $\bar{V}_{k,a_{2}}$ is reducible, then slope must be $\frac{1}{2}.$\\

    \underline{\bf{Slope = $1$}}\\
    We recall from \S \ref{slope=1} the following notations: $r=k-2 \geq 2, t=v(r-2),v(a_{2})=1$ and $\tau= v\left(\dfrac{a_{2}^2-2^2{r \choose 2}}{2a_{2}}\right)$. 
    
    \underline{Table 2}\\
   	\begin{tabular}{|c|c|c|c|c|c|c|}
 		\hline
 		S.No & r & $a_{2}$& $\tau$ &$t-1$& Contributors to $\bar{\Theta}$ & Reducibility of  \\
        & & & & &  by SAGE output &$\bar{V}_{k,a_{p}} \text{ (assuming LLC)}$\\
        \hline
        1 & $4$ & $2\sqrt{3}$ &$0$ &$0$& $\mathcal{I}(V_1)/(T^2+1)$ \text{ and } $\mathcal{I}(V_0 )/(T-1)$ & Reducible\\	 	
        \hline
        2 & $5$ & $6$ &$0$ &$-1$& $\mathcal{I}(V_1)/(T) $&Irreducible\\	
 		\hline 
        3 & $6$ & $2\sqrt{3}$ & $2$& $1$& $\mathcal{I}(V_1)/(T) $& Irreducible\\	
 		\hline 
        4 & $6$ & $2\zeta_{3}$ &$0$  &$1$ & ${\mathcal{I}(V_1)}/{(T^2-\zeta_3 T+1) }$&Reducible \\
 		\hline
        5 & $7$ & $2\sqrt{3}$ &$1$  &$-1$ & $\mathcal{I}(V_1)/(T) $ &Irreducible \\
 		\hline       
        6 & $8$ & $ 46 $ & 0&0 & $\mathcal{I}(V_1)/(T^2+1) \text{ and }\mathcal{I}(V_0)/(T-1) $ &Reducible \\
 		\hline	
        7 & $9$ & $2\zeta_{3}$ & $0$ & $-1$ &$\mathcal{I}(V_1)/(T) $ &Irreducible\\
        \hline
        8 & $10$ & $2$ & $2$ &$2$ & $\mathcal{I}(V_1)/(T^2+1) \text{ and }\mathcal{I}(V_0)/(T-1) $  &Reducible\\
 		\hline
        9 & $12$  & $2\zeta_{3}$ &$0$ &  $0$ &$\mathcal{I}(V_1)/(T^2-T+1) $  & Reducible \\
 		\hline
 		10 & $12$ & 2& 0 & 0 & $\mathcal{I}(V_1)/(T^2+1) \text{ and }\mathcal{I}(V_0)/(T-1) $ & Reducible \\
 		\hline
        11 & $14$ & $2\sqrt{7}$ & $2$ & $1$ & $\mathcal{I}(V_1)/(T)$ & Irreducible\\
        \hline
        12 & $18$ & $-10\pm 2\sqrt{178}$ & $1$ & $3$ & $\mathcal{I}(V_1)/(T^2-T+1) $ & Reducible \\
 		\hline
        
 	\end{tabular} 
\vspace{10mm}

 \noindent We note that the data above align with the following consequences of our work for slope $1$.\\
  $\bullet$ If $r$ is odd, then $\bar{V}_{k,a_{p}}$ must be irreducible.\\
  $\bullet$ If $\tau>t-1$, then the sole contributor to $\bar{\Theta}$ is $\dfrac{I(V_{1})}{(T)}$, as seen in rows $2,3,5,7$ and $11$. \\
  $\bullet$ If $\tau<t-1$, then the sole contributor to $\bar{\Theta}$ is $\dfrac{I(V_{1})}{(T^2-c_{0}T+1)},$ as seen in rows $4$ and $12$.\\
  $\bullet$ If $\tau=t-1$, then both $\dfrac{I(V_{1})}{(T^2-c_{0}T+1)}$ and $\dfrac{I(V_{0})}{(T-{2}/{a_{2}})}$ can possibly contribute. In fact, it follows from the proof of Theorem \ref{mtslope1p=2} that the only contributor to $\bar{\Theta}$ is $\dfrac{I(V_{1})}{(T^2-c_{0}T+1)}$, unless $\overline{\dfrac{2}{a_{2}}} = \overline{\dfrac{2-r}{2\alpha}}$. 
  
  The examples in rows $6,8,9$ and $10$ fall under the case $\tau=t-1$. In row $9$, we observe that $\overline{\dfrac{2}{a_{2}}} = \overline{-1-\zeta_{3}} \neq  \overline{\zeta_{3}}=\overline{\dfrac{2-r}{2\alpha}}$ while in the other cases $\overline{\dfrac{2}{a_{2}}} = 1 =\overline{\dfrac{2-r}{2\alpha}}$. We note that the SAGE output in these cases is consistent with the above observation.

\section*{Acknowledgements}
 The first author is supported by ARG-Matrics grant from Anusandhan National Research Foundation. 
  The second author is supported by the Prime Minister’s Research Fellowship (PMRF), Ministry of Education, Government of India.

\end{document}